\documentclass[11pt]{article}
\usepackage[a4paper]{geometry}
\usepackage{amsthm}
\usepackage{amsmath}
\usepackage{amssymb}
\usepackage{amsfonts}
\usepackage{graphicx}
\usepackage{float}
\usepackage{color}
\usepackage[bf,SL,BF]{subfigure}
\newcommand{\bfi}{\bfseries\itshape}
\usepackage{epsfig,amsbsy,graphicx,multirow}

\usepackage[all]{xypic}

\usepackage{url}
\usepackage{hyperref}

\newcommand{\imsizebig}{1.0\columnwidth}

 \numberwithin{equation}{section}
\def\eref#1{(\ref{#1})}

\def\N{\mathbb{N}}

\def\P{\mathbb{P}}
\def\R{\mathbb{R}}

\def\E{\mathbb{E}}

\def\<{\big\langle}
\def\>{\big\rangle}

\def\Hess{\operatorname{Hess}}

\newtheorem{Lemma}{Lemma}[section]

\newtheorem{Theorem}{Theorem}[section]
\newtheorem{Proposition}{Proposition}[section]

\newtheorem{Condition}{Condition}[section]

\theoremstyle{remark}
\newtheorem{Remark}{Remark}[section]

\theoremstyle{definition}

\newtheorem{Definition}{Definition}[section]

\theoremstyle{definition}

\begin{document}
\title{Nonintrusive and structure preserving  multiscale integration of  stiff ODEs, SDEs and Hamiltonian systems  with hidden slow dynamics via flow averaging}

\date{\today}

\author{Molei Tao$^1$, Houman Owhadi$^{1,2}$, and Jerrold E. Marsden$^1$}

\footnotetext[1]{California Institute of
Technology, Applied \& Computational Mathematics, Control \&
Dynamical systems, MC 217-50 Pasadena , CA 91125 }
\footnotetext[2]{owhadi@caltech.edu}

\maketitle
\begin{abstract}
We introduce a new class of  integrators for stiff ODEs as well as SDEs. Examples of subclasses of systems that we treat are ODEs and SDEs that are sums of two terms, one of which has large coefficients.
These integrators  are (i) {\it Multiscale}: they  are based on flow averaging and so do not fully resolve the fast variables and have a computational cost determined by slow variables (ii) {\it Versatile}: the method is based on averaging the flows of the given dynamical system (which may have hidden slow and fast processes) instead of averaging the instantaneous drift of assumed separated slow and fast processes. This bypasses the need for identifying explicitly (or numerically) the slow or fast variables (iii) {\it Nonintrusive}: A pre-existing numerical scheme resolving the microscopic time scale can be used as a black box and easily turned into one of the integrators in this paper by turning the large coefficients on over a microscopic timescale  and off during a mesoscopic timescale (iv) {\it Convergent over two scales}: strongly over slow processes and in the sense of measures over fast ones. We introduce the related notion of two-scale flow convergence and analyze the convergence of these integrators under the induced topology (v) {\it Structure preserving}: They inherit the structure preserving properties of the legacy integrators from which they are derived. Therefore, for stiff Hamiltonian systems (possibly on manifolds), they can be made to be symplectic, time-reversible, and symmetry preserving (symmetries are group actions that leave the system invariant) in all variables. They are explicit and applicable to arbitrary stiff potentials (that need not be quadratic). Their application to the Fermi-Pasta-Ulam problems shows accuracy and stability over four orders of magnitude of time scales. For stiff Langevin equations, they are symmetry preserving, time-reversible and Boltzmann-Gibbs reversible, quasi-symplectic on all variables and conformally symplectic with isotropic friction.
\end{abstract}

\tableofcontents

\paragraph{Acknowledgements} Part of this work has been supported by NSF grant CMMI-092600.
We are grateful to  C. Lebris,  J.M. Sanz-Serna, E. S. Titi, R. Tsai and E. Vanden-Eijnden  for  useful comments and providing references. We would also like to thank two anonymous referees for precise and detailed  comments and suggestions.

\section{Overview of the integrator on ODEs}
Consider the following  ODE on $\R^d$,
\begin{equation}\label{fullsystem}
\dot{u}^\epsilon=G(u^\epsilon)+\frac{1}{\epsilon}F(u^\epsilon).
\end{equation}
In Subsections \ref{natfla1}, \ref{subham}, \ref{subsde}, \ref{natfla2} and \ref{sublan} we will consider more general ODEs, stiff deterministic Hamiltonian systems \eref{ksdjjsdshkdjjksdhj}, SDEs (\eref{fullsystemS} and \eref{jkhgsdejgdshhh}) and  Langevin equations (\eref{jdhsjhgdjwghe} and \eref{jsdkjdshdgjdhd}); however for the sake of clarity, we will start the description of our method with \eref{fullsystem}.

\begin{Condition}\label{lsdsddsdeehxA1} Assume that there exists a diffeomorphism $\eta:=(\eta^x,\eta^y)$, from $\R^d$ onto $\R^{d-p}\times \R^p$ (with uniformly bounded $C^1, C^2$ derivatives), separating slow and fast variables, i.e., such that (for all $\epsilon>0$) the process $(x^\epsilon_t,y^\epsilon_t)=(\eta^x(u^\epsilon_t),\eta^y(u^\epsilon_t))$ satisfies an ODE system of the form
\begin{equation}\label{kfgfgdiuuiusedejhd}
\begin{cases}
\dot{x}^\epsilon=g(x^\epsilon,y^\epsilon) & x^\epsilon_0=x_0\\
\dot{y}^\epsilon=\frac{1}{\epsilon}f(x^\epsilon,y^\epsilon) & y^\epsilon_0=y_0
\end{cases} .
\end{equation}
\end{Condition}
\begin{Condition}\label{lsdsddsdeehxA2} Assume that the fast variables in \eref{kfgfgdiuuiusedejhd} are locally ergodic with respect to a family of  measures $\mu$ drifted by slow variables. More precisely, we assume that there exists a family of probability measures $\mu(x,dy)$ on $\R^p$ indexed by $x\in\R^{d-p}$ and a positive function $T \mapsto E(T)$ such that $\lim_{T\rightarrow \infty}E(T)=0$
  and such that for all $x_0,y_0,T$ and $\phi$ uniformly bounded and Lipschitz,  the solution to
\begin{equation}
\dot{Y}_t=f(x_0,Y_t)\quad \quad Y_0=y_0
\end{equation}
satisfies
    \begin{equation}
    \Big|\frac{1}{T}\int_0^T \phi(Y_s)ds-\int_{\R^p} \phi(y)\mu(x_0,dy)\Big|\leq \chi\big(\|(x_0,y_0)\|\big) E(T) (\|\phi\|_{L^\infty}+\|\nabla \phi\|_{L^\infty} )
    \end{equation}
where $r \mapsto \chi(r)$ is bounded on compact sets.
\end{Condition}

Under conditions \ref{lsdsddsdeehxA1} and \ref{lsdsddsdeehxA2}, it is known (we refer for instance to \cite{MR810620} or to Theorem 14, Section 3 of Chapter II of \cite{MR1020057} or to \cite{MR2382139}) that $x^\epsilon$ converges towards $x_t$ defined as the solution to the ODE
\begin{equation}\label{kfgfgdserreredejhd}
\dot{x}=\int g(x,y)\mu(x,dy),\quad x | _{t=0} = x_0
\end{equation}
where $\mu(x,dy)$ is the ergodic measure associated with the solution to the ODE
\begin{equation}\label{kfgfsddfsedejhd}
\dot{y}=f(x,y)
\end{equation}
It follows that the slow behavior of solutions of \eref{fullsystem} can be simulated over coarse time steps by first
 identifying the slow process $x^\epsilon$ and then using numerical approximations of solutions of \eref{kfgfgdiuuiusedejhd} to approximate $x^\epsilon$. Two classes of integrators  have been founded on this observation: The equation free method \cite{MR2041455, KevGio09} and the Heterogeneous Multiscale Method \cite{MR2314852, MR2164093, MR2069938, Ariel:08}. One shared characteristic of the original form of those integrators is, after identification of the slow variables, to use a micro-solver to approximate the effective drift in \eref{kfgfgdserreredejhd} by  averaging the instantaneous drift $g$ with respect to numerical solutions of \eref{kfgfsddfsedejhd} over a time span larger than the mixing time of the solution to  \eref{kfgfsddfsedejhd}.

\subsection{FLAVORS}
In this paper, we propose a new method based on the averaging of the instantaneous flow of the ODE \eref{fullsystem} with hidden slow and fast variables instead of the instantaneous drift of $x^\epsilon$ in ODE \eref{kfgfgdiuuiusedejhd} with separated slow and fast variables.  We have called the resulting class of numerical integrators {\bfi FLow AVeraging integratORS (FLAVORS)}.
Since FLAVORS are directly applied to \eref{fullsystem}, hidden slow variables do not need to be identified, either explicitly  or numerically. Furthermore FLAVORS can be implemented using an arbitrary legacy integrator $\Phi^\frac{1}{\epsilon}_h$ for  \eref{fullsystem} in which the parameter  $\frac{1}{\epsilon}$ can be controlled (figure \ref{FLavorpic}).
 \begin{figure} [h]
\begin{tabular}{c}
{\resizebox{\imsizebig}{!}{\includegraphics{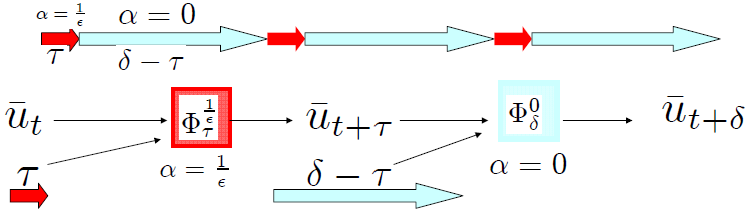}}}
\end{tabular}
\caption{\footnotesize A pre-existing numerical scheme resolving the microscopic time scale can be used as a black box and turned into a FLAVOR by simply turning on and off  stiff parameters over a microscopic timescale $\tau$  (on) and a mesoscopic timescale $\delta$ (off).
The bottom
line of the approach is to (repeatedly) compose an
accurate, short-time integration of the complete set of equations with an accurate, intermediate-time integration of the non-stiff part of the system. 
While the integration over short time intervals is accurate (in a strong sense), this is extended to intermediate time integration (in the sense of measures) using the interplay between the short time integration and the mesoscopic integration.
The computational cost remains bounded independently from the stiff parameter $1/\epsilon$ because: (i) The whole system is only integrated over extremely short ($\tau \ll \epsilon$) time intervals during every intermediate ($\delta$) time intervals. (ii) The intermediate time step $\delta$  (that of the non-stiff part of the system) is  limited not by the fast time scales ($\epsilon$) but by the slow ones ($\mathcal{O}(1)$).
}
\label{FLavorpic}
\end{figure}
More precisely, assume that there exists a constant $h_0>0$ such that $\Phi^\alpha_h$ satisfies for all $h\leq h_0 \min(\frac{1}{\alpha},1) $ and $u\in \R^d$
\begin{equation}\label{lsfassddlfsssekjlkd}
\big|\Phi_h^{\alpha}(u)-u-h G(u)-\alpha h F(u) \big|\leq C h^2 (1+\alpha)^2
\end{equation}
then   FLAVOR can be defined as the algorithm simulating the process
\begin{equation}\label{ksjhjhhjdskdjjhwuel}
\bar{u}_{t}=\big(\Phi^0_{\delta-\tau}\circ \Phi^{\frac{1}{\epsilon}}_{\tau}\big)^k(u_0) \quad \text{for}\quad k\delta \leq t <(k+1)\delta
\end{equation}
where $\tau$ is a fine time step resolving the fast time scale ($\tau \ll\epsilon$) and $\delta$ is a mesoscopic time step independent of the fast time scale satisfying $\tau \ll\epsilon \ll \delta \ll 1$ and
\begin{equation}\label{eqlimits}
(\frac{\tau}{\epsilon})^2\ll\delta \ll \frac{\tau}{\epsilon}
\end{equation}
In our numerical experiments, we have used the ``rule of thumb'' $\delta\sim \gamma \frac{\tau}{\epsilon}$ where $\gamma$ is a small parameter ($0.1$ for instance).

By switching stiff parameters FLAVOR approximates the flow of \eref{fullsystem} over a coarse time step $h$ (resolving the slow time scale) by the flow
\begin{equation}
\Phi_h:=\big(\Phi^0_{\frac{h}{M}-\tau}\circ \Phi^{\frac{1}{\epsilon}}_{\tau}\big)^{M}
\end{equation}
where  $M$ is a positive integer corresponding to the number of ``samples'' used to average the flow ($\delta$ has to be identified with $\frac{h}{M}$). We refer to subsection \ref{rationale} for the distinction between macro and meso-steps, for the
 rationale and mechanism behind FLAVORS and the limits \eref{eqlimits}.

 Since FLAVORS are obtained by flow-composition, we will show in Section \ref{jhhsghdjshdjg} and \ref{kshskjhdklshdkjhjh} that they inherit the structure preserving properties (for instance symplecticity and symmetries under a group action) of the legacy integrator for Hamlitonian systems and Langevin equations.

Under conditions \eref{eqlimits} on $\tau$ and $\delta$, we show that \eref{ksjhjhhjdskdjjhwuel} is strongly accurate with respect to (hidden) slow variables and weakly (in the sense of measures) accurate with respect to (hidden) fast variables .
Motivated by this observation, we introduce  the related notion of  {\bf two-scale flow convergence} in analogy with homogenization theory for elliptic PDEs \cite{Ngu90, Allaire1992} and call it F-convergence for short. $F$-convergence is close in spirit to
the Young measure approach to computing slowly advancing fast
oscillations introduced in \cite{Art07, Art07b}.

\subsection{Two-scale flow convergence}\label{saskjagsjhgs}
Let $(\xi_t^\epsilon)_{t\in \R^+}$ be a sequence of processes on $\R^d$ (functions from $\R^+$ to $\R^d$)  indexed by $\epsilon>0$. Let $(X_t)_{t\in \R^+}$ be a process on $\R^{d-p}$ ($p \geq 0$).
Let $x \mapsto \nu(x,dz)$ be a function from $\R^{d-p}$ into the space of probability measures on $\R^d$.
\begin{Definition}
We say that the process $\xi^\epsilon_t$ F-{\bfi converges to} $\nu(X_t,dz)$ as $\epsilon \downarrow 0$ and write $\xi^\epsilon_t \xrightarrow [\epsilon \rightarrow 0]{F} \nu(X_t,dz)$ if and only if for all functions $\varphi$ bounded and uniformly Lipshitz-continuous on $\R^d$, and for all $t>0$,
\begin{equation}
\lim_{h\rightarrow 0} \lim_{\epsilon \rightarrow 0} \frac{1}{h}\int_{t}^{t+h}\varphi(\xi_s^\epsilon)\,ds=
\int_{\R^d} \varphi(z)\nu(X_t,dz)
\end{equation}
\end{Definition}

\subsection{Asymptotic convergence result}
Our convergence theorem requires that $u_t^\epsilon$ and $\bar{u}_t$ do not blow up as $\epsilon \downarrow 0$; more precisely, we will assume that the following conditions are satisfied.
\begin{Condition}\label{lsdsddsdeehxA3}
Assume that:
\begin{enumerate}
\item[{\rm 1.}]  $F$ and $G$ are  Lipschitz  continuous.
\item[{\rm 2.}] For all $u_0$, $T>0$, the trajectories  $(u_{t}^\epsilon)_{0\leq t\leq T}$ are uniformly bounded in $\epsilon$.
\item[{\rm 3.}] For all $u_0$, $T>0$, the trajectories  $(\bar{u}_{t}^\epsilon)_{0\leq t\leq T}$ are uniformly bounded in $\epsilon$, $0< \delta \leq h_0$, $\tau \leq \min(\tau_0 \epsilon, \delta)$.
\end{enumerate}
\end{Condition}

For $\pi$, an arbitrary measure on $\R^d$, we define $\eta^{-1}*\pi$ to be the push forward of the measure $\pi$ by $\eta^{-1}$.
\begin{Theorem}\label{thm01}
Let $u_t^\epsilon$ be the solution to \eref{fullsystem} and $\bar{u}_{t}$ be defined by \eref{ksjhjhhjdskdjjhwuel}.
 Assume that equation \eref{lsfassddlfsssekjlkd} and conditions \ref{lsdsddsdeehxA1}, \ref{lsdsddsdeehxA2} and \ref{lsdsddsdeehxA3}  are satisfied, then
 \begin{itemize}
\item $u_t^\epsilon$ $F$-converges to $\eta^{-1}*\big(\delta_{X_t}\otimes \mu(X_t,dy)\big)$ as $\epsilon \downarrow 0$ where $X_t$ is the solution to
\begin{equation}\label{skljhdkhdjksdh}
\dot{X}_t=\int g(X_t,y)\,\mu(X_t,dy)\quad \quad X_0=x_0.
\end{equation}
\item
$\bar{u}_{t}$ $F$-converges to $\eta^{-1}*\big(\delta_{X_t}\otimes \mu(X_t,dy)\big)$   for  $\epsilon\leq \delta/(-C\ln \delta)$,
$\frac{\tau}{\epsilon}\downarrow 0$,  $\frac{\epsilon}{\tau} \delta \downarrow 0$ and $(\frac{\tau}{\epsilon})^2 \frac{1}{\delta} \downarrow 0$.
\end{itemize}
\end{Theorem}
\begin{Remark}
The $F$-convergence of $u_t^\epsilon$  to $\eta^{-1}*\big(\delta_{X_t}\otimes \mu(X_t,dy)\big)$ can be restated as
\begin{equation}
\lim_{h\rightarrow 0} \lim_{\epsilon \rightarrow 0} \frac{1}{h}\int_{t}^{t+h}\varphi(u_s^\epsilon)\,ds=
\int_{\R^p} \varphi(\eta^{-1}(X_t,y))\mu(X_t,dy)
\end{equation}
for all functions $\varphi$ bounded and uniformly Lipshitz-continuous on $\R^d$, and for all $t>0$.
\end{Remark}
\begin{Remark}
Observe that $g$ comes from \eref{kfgfgdserreredejhd}. It is not explicitly known and does not need to be explicitly known for the implementation of the proposed method.
\end{Remark}
\begin{Remark}\label{Remspeednew}
The limits on $\epsilon, \tau$ and $\delta$ are in essence stating that FLAVOR is accurate provided that $\tau\ll \epsilon$ ($\tau$ resolves the stiffness of \eref{fullsystem}) and equation \eref{eqlimits} is satisfied.
\end{Remark}
\begin{Remark}
Throughout this paper, $C$ will refer to an appropriately large enough constant independent from $\epsilon,\delta, \tau$. To simplify the presentation of our results, we use the same letter $C$ for expressions such as $2C e^C$ instead of writing it as a new constant $C_1$ independent from $\epsilon,\delta, \tau$.
\end{Remark}

\subsection{Rationale and mechanism behind FLAVORS}\label{rationale}
We will now explain the rationale and mechanism behind FLAVORS. We refer to Subsection \ref{subap1} of the appendix for the detailed proof of Theorem \ref{thm01}. Let us start by considering the case where $\eta$ is the identity diffeomorphism.
Let $\varphi^{\frac{1}{\epsilon}}$ be the flow of \eref{kfgfgdiuuiusedejhd}. Observe that  $\varphi^{0}$  (obtained from $\varphi^{\frac{1}{\epsilon}}$ by setting the parameter $\frac{1}{\epsilon}$ to zero) is
 the flow of \eref{kfgfgdiuuiusedejhd} with $y^\epsilon$ frozen, i.e.,
\begin{equation}\label{kfgfgwdsdejhd}
\varphi^0(x,y)=(\hat{x}_t,y)\quad \text{where $\hat{x}_t$ solves}\quad
\frac{d\hat{x}}{dt}=g(\hat{x},y),\quad \hat{x}_0=x .
\end{equation}
The main effect of FLAVORS is to  average the flow of \eref{kfgfgdiuuiusedejhd} with respect to fast degrees of freedom via splitting and re-synchronization. By splitting, we refer to the substitution of the flow $\varphi^{\frac{1}{\epsilon}}_\delta$ by composition of $\varphi^0_{\delta-\tau}$ and $\varphi^{\frac{1}{\epsilon}}_\tau$, and by re-synchronization we refer to the distinct time-steps $\delta$ and $\tau$ whose effects are to advance the internal clock of fast variables by $\tau$ every step of length $\delta$. By averaging, we refer to the fact that
 FLAVORS approximates the flow $\varphi^{\frac{1}{\epsilon}}_h$ by the flow
\begin{equation}\label{kfgssdejhd}
\varphi_h:=\big(\varphi^0_{\frac{h}{M}-\tau}\circ \varphi^{\frac{1}{\epsilon}}_{\tau}\big)^M
\end{equation}
where $h$ is a coarse time step resolving the slow time scale associated with $x^\epsilon$, $M$ is a positive integer corresponding to the number of samples used to average the flow ($\delta$ is identified with
$\frac{h}{M}$) and $\tau$ is a fine time step resolving the fast time scale, of the order of $\epsilon$, and associated with $y^\epsilon$. In general, analytical formulae are not available for  $\varphi^0$ and $\varphi^{\frac{1}{\epsilon}}$ and numerical approximations are used instead.

Observe that when FLAVORS are applied to systems with explicitly separated slow and fast processes, they lead to integrators that are locally in the neighborhood of those obtained with  HMM (or equation free) methods with a reinitialization of the fast variables at macrotime $n$ by their final value at macrotime step $n-1$ and  with only one microstep per macrostep \cite{MR2165382, Seamless09}.

We will now consider the situation where $\eta$ is not the identity diffeomorphism and give the rationale behind the limits \eref{eqlimits}.
 \[\xymatrix{
 \bar{u}_{n\delta} \ar[r]^{\Phi^\frac{1}{\epsilon}_\tau}\ar@<1ex>[d]^{\eta} & \bar{u}_{n\delta+\tau} \ar[r]^{\Phi^0_{\delta-\tau}}\ar@<1ex>[d]^{\eta} & \bar{u}_{(n+1)\delta}\ar@<1ex>[d]^{\eta}\\
  (\bar{x},\bar{y})_{n\delta} \ar[r]^{\Psi^\frac{1}{\epsilon}_\tau}\ar@<1ex>[u]^{\eta^{-1}} & (\bar{x},\bar{y})_{n\delta+\tau} \ar[r]^{\Psi^0_{\delta-\tau}}\ar@<1ex>[u]^{\eta^{-1}}  & (\bar{x},\bar{y})_{(n+1)\delta}\ar@<1ex>[u]^{\eta^{-1}}
  } \]
 As illustrated in the above diagram, since $(\bar{x}_t,\bar{y}_t)=\eta(\bar{u}_t)$,
 simulating $\bar{u}_{n\delta}$ defined in \eref{ksjhjhhjdskdjjhwuel} is equivalent to simulating the discrete process
\begin{equation}
(\bar{x}_{n\delta},\bar{y}_{n\delta}):=\big(\Psi^{\frac{1}{\epsilon}}_{\delta-\tau}\circ \Psi^0_{\tau}\big)^n (x_{0},y_{0})
\end{equation}
where
\begin{equation}
\Psi^{\alpha}_h:= \eta \circ \Phi^\alpha_h \circ \eta^{-1}
\end{equation}
Observe that the accuracy (in the topology induced by F-convergence) of $\bar{u}_t$ with respect to $u^\epsilon_t$, solution of \eref{fullsystem}, is equivalent to that of $(\bar{x}_t,\bar{y}_t)$ with respect to $(x^\epsilon_t,y^\epsilon_t)$ defined by \eref{kfgfgdiuuiusedejhd}.
Now, for the clarity of the presentation, assume that
\begin{equation}\label{simpleexamplephi}
\Phi^{\alpha}_h(u)=u+h G(u)+\alpha h F(u)
\end{equation}
Using Taylor's theorem and \eref{simpleexamplephi}, we obtain that
\begin{equation}\label{acc1}
\Psi^\alpha_h(x,y)=(x,y)+h \big(g(x,y),0\big)+\alpha h  \big(0,f(x,y)\big)+\int_0^1 v^T \operatorname{Hess}\eta(u+t v) v (1-t)^2\,dt
\end{equation}
with
\begin{equation}\label{acc2}
u:=\eta^{-1}(x,y)\quad \text{and}\quad v:=h ( G+\alpha  F)\circ \eta^{-1}(x,y)
\end{equation}
It follows from equations \eref{acc1} and \eref{acc2} that $\Psi^\frac{1}{\epsilon}_h$ is a first order accurate integrator approximating the flow of \eref{kfgfgdiuuiusedejhd} and $\Psi^0_h$ is a first order accurate integrator approximating the flow of \eref{kfgfgwdsdejhd}.
Let $h$ be a coarse time step and $\delta$ a mesostep.
Since $\bar{x}$ remains nearly constant over the coarse time step,
the switching (on and off) of the stiff parameter $\frac{1}{\epsilon}$ averages the drift $g$ of $\bar{x}$ with respect to the trajectory of $\bar{y}$ over $h$. Since the coarse step $h$ is composed of $\frac{h}{\delta}$ mesosteps, the internal clock of the fast process is advanced by $\frac{h}{\delta}\times \frac{\tau}{\epsilon}$. Since $h$ is of the order of  one, the trajectory of $\bar{y}$ is mixing with respect to the local ergodic measure $\mu$ provided that $\frac{\tau}{\delta \epsilon}\gg 1$, i.e.
\begin{equation}\label{akghhgjdhgsjhdgg}
\delta \ll \frac{\tau}{\epsilon}
\end{equation}
Equation \eref{akghhgjdhgsjhdgg} corresponds to the right hand side of equation \eref{eqlimits}.
If $\eta$ is a non-linear diffeomorphism (with non-zero Hessian), it also follows from equations \eref{acc1} and \eref{acc2} that each invocation of the  integrator $\Psi^\frac{1}{\epsilon}_\tau$ occasions an error (on the accuracy of the slow process) proportional to  $(\frac{\tau}{\epsilon})^2$. Since during the coarse time step $h$, $\Psi^\frac{1}{\epsilon}_\tau$ is solicited $\frac{h}{\delta}$-times, it follows that the  error accumulation during $h$ is $\frac{h}{\delta}\times (\frac{\tau}{\epsilon})^2$. Hence, the accuracy of the integrator requires that $\frac{1}{\delta}\times (\frac{\tau}{\epsilon})^2\ll 1$, i.e.
\begin{equation}\label{akghhgjdhgsjhdgg2}
\big(\frac{\tau}{\epsilon}\big)^2 \ll \delta
\end{equation}
Equation \eref{akghhgjdhgsjhdgg2} corresponds to the left hand side of equation \eref{eqlimits}.

Observe that if $\eta$ is linear, its Hessian is null and the remainder in the right hand side of \eref{acc1} is zero. It follows that if
$\eta$ is linear, the error accumulation due to fine time steps on slow variables is zero and  condition \eref{akghhgjdhgsjhdgg} is sufficient for the accuracy of the integrator.

It has been observed in \cite{Nested07} and in Section 5 of \cite{Eric07HMMlike} that slow variables do not need to be identified with
HMM/averaging type integrators  if the relation between original and slow variables is linear or a permutation and if
\begin{equation}\label{jsgjhgsghdgdue}
\frac{\Delta t}{M}\ll \frac{\tau}{\epsilon}
\end{equation}
where is $M$ the number of fine-step iterations used by HMM to compute the average the drift of slow variables and $\Delta t$ is the coarse time step (in HMM) along the direction of the averaged drift.
The analysis of FLAVORS associated with equation \eref{acc1} reaches a similar conclusion if $\eta$ is linear in the sense that
the error caused by the Hessian of $\eta$ in \eref{acc1} is zero and in the (sufficient) condition \eref{akghhgjdhgsjhdgg} is analogous to \eref{jsgjhgsghdgdue} for $M=1$.
It is also stated  on Page 2 of \cite{Nested07} that \emph{``there are counterexamples showing that algorithms of the same spirit do not work
for deterministic ODEs with separated time scales if the slow variables are not explicitly identified and made use of. But in the present context, the slow variables are linear functions of the original variables, and this is the reason
why the seamless algorithm works.''}
Here, the analysis of FLAVORS associated with equation \eref{acc1} shows an algorithm based on an averaging principle would indeed, in general, not work if $\eta$ is nonlinear (and \eref{akghhgjdhgsjhdgg2} not satisfied) due to the error accumulation (on slow variables) associated with the Hessian of $\eta$. However, the above analysis also shows that if condition \eref{akghhgjdhgsjhdgg2} is satisfied, then, although $\eta$ may be \emph{nonlinear},  flow averaging integrators will \emph{always} work without  \emph{identifying slow variables}.

\subsection{Non asymptotic convergence result}

\begin{Theorem}\label{thm01b}
Under assumptions and notations of theorem \ref{thm01},  there exists $C>0$ such that for $\delta <h_0$, $\tau <\tau_0 \epsilon $ and $t>0$,
\begin{equation}\label{jkdsgkdjshgdsdkjghe}
|x_{t}^\epsilon-\eta^x(\bar{u}_{t})|\leq C e^{C t} \chi_1(u_0,\epsilon,\delta,\tau)
\end{equation}
and
\begin{align}
& \left| \frac{1}{T}\int_{t}^{t+T}\varphi(\bar{u}_{s})\,ds  -\int_{\R^p} \varphi(\eta^{-1}(X_t,y))\mu(X_t,dy) \right| \nonumber \\
& \qquad \qquad \qquad  \leq \chi_2(u_0,\epsilon,\delta,\tau,T,t) (\|\varphi\|_{L^\infty}+\|\nabla \varphi\|_{L^\infty} ) \label{lkslkcdhkhsdedj3}
\end{align}
where  $\chi_1$ and $\chi_2$ are functions converging towards zero as   $\epsilon\leq \delta/(C\ln \frac{1}{\delta})$,
$\frac{\tau}{\epsilon}\downarrow 0$,  $\frac{\epsilon}{\tau} \delta \downarrow 0$ and $(\frac{\tau}{\epsilon})^2 \frac{1}{\delta} \downarrow 0$ (and $T\downarrow 0$ for $\chi_2$).
\end{Theorem}
\begin{Remark}
For  $\epsilon\leq \delta/(-C\ln \delta)$ and $\delta \frac{\epsilon}{\tau}+\frac{\tau}{\epsilon} \leq 1$, the following holds
\begin{equation}
\begin{split}
\chi_1(u_0,\epsilon,\delta,\tau)\leq
\sqrt{\delta}+\big(\frac{\tau}{\epsilon}\big)^2 \frac{1}{\delta}+E\big(\frac{1}{C}\ln \frac{1}{\delta}\big)+\big(\frac{\delta \epsilon}{\tau}\big)^{\frac{1}{2}}+\big(\frac{\tau}{\epsilon}\big)^{\frac{1}{2}}+ E\Big(\frac{1}{C}\ln \Big(\big(\frac{\delta \epsilon}{\tau}+\frac{\tau}{\epsilon}\big)^{-1}\Big)\Big)
\end{split}
\end{equation}
and $\chi_2$ satisfies a similar inequality.
\end{Remark}
\begin{Remark}\label{Remspeed}
Choosing $\tau \sim \gamma \epsilon$ and $\delta \sim \gamma \frac{\tau}{\epsilon}$,where $\gamma$ is a small constant independent from $\epsilon$, Theorem \ref{thm01b} shows that the approximation error of FLAVOR is bounded by a function of $\gamma$ converging towards zero as $\gamma \downarrow 0$. If follows that the speed up is of the order of $\frac{\delta}{\tau}\sim \frac{\gamma}{\epsilon}$, i.e., scales like $\frac{1}{\epsilon}$ at fixed accuracy.  In order to be able to compare FLAVOR with integrators resolving all the fine time steps we have limited the speed up in the numerical experiments  to $200\times$ (but the latter can be arbitrary large as $\epsilon \downarrow 0$).  For sufficiently small $\epsilon$, we observe that FLAVORS with microstep $\tau$ and mesostep $\delta$ overperform their associated legacy integrator with the same microstep $\tau$ over large simulation times (we refer  to Section \ref{FPUsec} on the Fermi-Pasta-Ulam problem). This phenomenon is caused by an error accumulation at each tick (microstep) of the clock of fast variables. Since FLAVORS (indirectly, i.e., without identifying fast variables) slow down  the speed of this clock from $\frac{1}{\epsilon}$ to a value  $\frac{\tau}{\delta \epsilon}\sim \frac{1}{\gamma}$ independent from $\epsilon$ this error does not blow up as $\epsilon \downarrow 0$ (as opposed to an integrator resolving all fine time steps). Because of this reason, if this error accumulation on fast variables is exponential, then the speed up at fixed accuracy does not scale  like $\frac{1}{\epsilon}$, but like $e^{\frac{T}{\epsilon}}$ where $T$ is the total simulation time. A consequence of this phenomenon can be seen in Figure \ref{FPU_long} (associated with the FPU problem) where Velocity Verlet fails to capture the $\mathcal{O}(\epsilon^{-1})$ dynamics with a time step $h=10^{-5}$ whereas FLAVORS remain accurate with $\tau=10^{-4}$ and $\delta=2\cdot 10^{-3}$.
\end{Remark}
\begin{Remark}
\label{exponentialError}
The reader should not be surprised by the presence of the exponential factor $e^{Ct}$ in \eref{jkdsgkdjshgdsdkjghe}.
It is known that global errors for numerical approximations of ODEs grow, in general, exponentially with time (see for instance \cite{MR1227985}). These bounds are, however, already tight; consider, for instance, how error propagates in a generic numerical scheme applied to the special system of $\dot{x}=x$. It is possible to show that the increase of global errors is linear in time only for a restricted class of ODEs (using techniques from Lyapunov's theory of stability \cite{Viswanath01}). Notice that the constant $C$ in the exponential of our bound does not scale with $\epsilon^{-1}$, and therefore the bound is uniform and rather tight.
\end{Remark}
\begin{Remark}
We refer to \cite{MR2164093} for higher order averaging based methods. In particular, \cite{MR2164093} shows how, after identification of slow variables, balancing the different error contributions yields an explicit stable integration method having the order of the macro scheme.
\end{Remark}

\subsection{Natural FLAVORS}
Although convenient, it is not necessary to use legacy integrators to obtain FLAVORS.
More precisely, theorems \ref{thm01} and \ref{thm01b} remain valid if FLAVORS are defined to be algorithms simulating the discrete process
\begin{equation}\label{ksjhdskdjjhwue}
\bar{u}_{t}:=\big(\theta^G_{\delta-\tau}\circ \theta^\epsilon_{\tau}\big)^k(u_0) \quad \text{for}\quad k\delta \leq t <(k+1)\delta
\end{equation}
where $\theta_\tau^{\epsilon}$ and $\theta^{G}_{\delta-\tau}$ are two mappings  from $\R^d$ onto $\R^d$ (the former approximating the flow of the whole system \eref{fullsystem} for time $\tau$, and the  latter approximating the flow of $\dot{v}=G(v)$ for time $\delta -\tau$),  satisfying  the following conditions.
\begin{Condition}\label{lsddehlkssdsdsdeehx}
Assume that:
\begin{enumerate}
\item[{\rm 1.}] There exists $h_0, C>0$ such that for $h\leq h_0$ and any $u\in\R^d$,
\begin{equation}\label{hgfjhgdfsjhgfhf}
\big|\theta^{G}_{h}(u)-u-h G(u)\big|\leq C h^2
\end{equation}
\item[{\rm 2.}] There exists $\tau_0, C>0$,
such that for $\frac{\tau}{\epsilon}\leq \tau_0$  and any $u\in\R^d$,
\begin{equation}\label{lsfddlfskjlkd}
\left| \theta_\tau^{\epsilon}(u)-u-\tau G(u)-\frac{\tau}{\epsilon}F(u)  \right| \leq C \big(\frac{\tau}{\epsilon}\big)^2
\end{equation}
\item[{\rm 3.}] For all $u_0$, $T>0$, the discrete trajectories
$\Big(\big(\theta^G_{\delta-\tau}\circ \theta^\epsilon_{\tau}\big)^k(u_0)\Big)_{0\leq k \leq T/\delta}$
are uniformly bounded in $\epsilon$, $0< \delta \leq h_0$, $\tau \leq \min(\tau_0 \epsilon, \delta)$.
\end{enumerate}
\end{Condition}
Observe that \eref{ksjhjhhjdskdjjhwuel} is a particular case of \eref{ksjhdskdjjhwue} in which $\theta^\epsilon=\Phi^\frac{1}{\epsilon}$ and the mapping $\theta^G$ is obtained from the legacy integrator $\Phi^{\alpha}$ by setting $\alpha$ to zero.

\subsection{Related work}
\emph{Dynamical systems with multiple time scales pose a major problem in simulations because the small time steps required for stable integration
of the fast motions lead to large numbers of time steps required for the observation of slow degrees of freedom} \cite{TBM92, Hairer:04}. Traditionally, stiff dynamical systems have been separated into two classes with distinct integrators: stiff systems with fast transients and stiff systems with rapid oscillations \cite{Ariel:09, MR2069938, SSerna09}. The former has been solved using implicit schemes \cite{MR0315898, MR0080998, Hairer:04, MR1439506}, Chebyshev methods \cite{MR0443314, MR1923724} or the projective integrator approach \cite{MR1976207}.
These latter have been solved using filtering techniques \cite{MR0138200, MR1165724, MR717698} or Poincar\'{e} map techniques \cite{MR654346, MR1489260}.  We also refer to methods based on highly oscillatory quadrature \cite{MR2542877, MR2182817, MR1936107}, an area that has undergone significant developments in the last few years \cite{MR2303638}. It has been observed that \emph{at the present time, there exists no unified strategy for dealing with both classes of problems} \cite{MR2069938}. When slow variables can be identified, effective equations can be obtained by averaging the instantaneous drift driving those slow variables \cite{MR1020057}.
Two classes of numerical methods have been built on this observation: The equation-free method \cite{MR2041455, KevGio09} and the Heterogeneous Multiscale Method \cite{MR2314852, MR2164093, MR2069938, Ariel:08}. Observe that FLAVORS apply in a unified way to both stiff systems with fast transients and stiff systems with rapid oscillations, with or without noise, with a mesoscopic integration time step  chosen independently from the stiffness.

\subsection{Limitations of the method}
The proof of the accuracy of the method (theorems \ref{thm01} and \ref{thm01b}) is based on an averaging principle; hence, if $\epsilon$ is not small (the stiffness of the ODE is weak), although the method may be stable, there is no guarantee of accuracy. More precisely, the global error of the method is an increasing function of $\epsilon$, $\delta$, $\frac{\tau}{\epsilon}$, $\frac{\delta \epsilon}{\tau}$, $(\frac{\tau}{\epsilon})^2\delta$.
Writing $\gamma:=\frac{\tau}{\epsilon}$ the accuracy method requires $\gamma ^2 \ll \delta \ll \gamma$. Choosing $\delta=\gamma^{\frac{3}{2}}$, the condition $\epsilon \ll \delta \ll 1$ (related to computational gain) requires $\epsilon^{\frac{2}{3}}\ll \gamma \ll 1$ which can be satisfied only if $\epsilon$ is small.

The other limitation of the method lies in the fact that a stiff parameter $\frac{1}{\epsilon}$ needs to be clearly identified. In many examples of interest (Navier-Stokes equations, Maxwell's equations,...), stiffness is a result of nonlinearity, initial conditions or boundary conditions and not of the existence of a large parameter $\frac{1}{\epsilon}$.
Molecular dynamics can also create widely separated time-scales from non-linear effects; we refer, for instance, to \cite{Yanao:09} and references therein.

\subsection{Generic stiff ODEs}\label{natfla1}
FLAVORS have a natural generalization to systems of the form
\begin{equation}\label{jkhgjgdshhh}
\dot{u}^{\alpha,\epsilon}=F(u^{\alpha,\epsilon},\alpha,\epsilon)
\end{equation}
where $u \mapsto F(u,\alpha,\epsilon)$ is Lipshitz continuous.

\begin{Condition}\label{lsdsddcddsdeehx}
Assume that:
\begin{enumerate}
\item[{\rm 1.}]  $\epsilon \mapsto F(u,\alpha,\epsilon)$ is uniformly continuous in the neighborhood of $0$.
\item[{\rm 2.}] There exists a diffeomorphism $\eta:=(\eta^x,\eta^y)$,
from $\R^d$ onto $\R^{d-p}\times \R^p$, independent from $\epsilon, \alpha$, with uniformly bounded $C^1, C^2$ derivatives, such that the process $(x_t^{\alpha},y_t^{\alpha})=\big(\eta^x(u_t^{\alpha,0}),\eta^y(u_t^{\alpha,0})\big)$ satisfies, for all $\alpha \geq 1$, the ODE
\begin{equation}\label{kfgfgdiuucdqcdciusesdejhd}
\dot{x}^\alpha=g(x^\alpha,y^{\alpha})\quad x^\alpha_0=x_0 ,
\end{equation}
where $g(x,y)$ is Lipschitz continuous in $x$ and $y$ on bounded sets.

\item[{\rm 3.}] There exists a family of probability measures $\mu(x,dy)$ on $\R^p$
   such that for all $x_0,y_0,T$ $\big((x_0,y_0):=\eta(u_0)\big)$ and $\varphi$ uniformly bounded and Lipschitz
    \begin{equation}\label{fgfhgfhttf}
    \Big|\frac{1}{T}\int_0^T \varphi(y^{\alpha}_s)\,ds-\int_{\R^p} \varphi(y)\mu(x_0,dy)\Big|\leq \chi\big(\|(x_0,y_0)\|\big) \big(E_1(T)+E_2(T \alpha^\nu)\big) \|\nabla \varphi\|_{L^\infty}
    \end{equation}
where $r\mapsto \chi(r)$ is bounded on compact sets and $E_2(r)\rightarrow 0$ as $r\rightarrow \infty$ and $E_1(r)\rightarrow 0$ as $r\rightarrow 0$.
\item[{\rm 4.}] For all $u_0$, $T>0$, the trajectories  $(u_{t}^{\alpha,0})_{0\leq t\leq T}$ are uniformly bounded in $\alpha \geq 1$.
\end{enumerate}
\end{Condition}
\begin{Remark}
Observe that slow variables are not kept frozen in equation \eref{fgfhgfhttf}.
The error on local invariant measures induced by the (slow) drift of $x^\alpha$ is controlled by $E_2$. More precisely, the convergence of the right hand side of \eref{fgfhgfhttf} towards zero requires the convergence of $T$  towards zero and (at the same time) the divergence of $T \alpha^\nu$ towards infinity.
\end{Remark}

Assume that we are given a mapping $\Phi_h^{\alpha,\epsilon}$ from $\R^d$ onto $\R^d$ approximating the flow of \eref{jkhgjgdshhh}. If the parameter $\alpha$ can be controlled then $\Phi_h^{\alpha,\epsilon}$ can be used as a black box for accelerating the computation of solutions of \eref{jkhgjgdshhh}.

\begin{Condition}\label{lsduisdsdsdeehx}
Assume that:
\begin{enumerate}
\item[{\rm 1.}] There exists a constant $h_0>0$ such that $\Phi^{\alpha,\epsilon}$ satisfies for all $h\leq h_0 \min(\frac{1}{\alpha^\nu},1) $, $0<\epsilon \leq 1 \leq \alpha$
\begin{equation}\label{lsfasshjhhddlfskjlkd}
\big|\Phi_h^{\alpha,\epsilon}(u)-u-h F(u,\alpha,\epsilon) \big|\leq C(u) h^2 (1+\alpha^{2\nu})
\end{equation}
where $C(u)$ is bounded on compact sets.
\item[{\rm 2.}] For all $u_0$, $T>0$, the discrete trajectories
$\Big(\big(\Phi^{0,\epsilon}_{\delta-\tau}\circ \Phi^{\frac{1}{\epsilon},\epsilon}_{\tau}\big)^k(u_0)\Big)_{0\leq k \leq T/\delta}$
are uniformly bounded in $0<\epsilon\leq 1$, $0< \delta \leq h_0$, $\tau \leq \min(h_0 \epsilon^\nu, \delta)$.
\end{enumerate}
\end{Condition}

FLAVOR can be defined as the algorithm given by the process
\begin{equation}\label{ksjhjhhjdskdjjhhuhuuwue}
\bar{u}_{t}=\big(\Phi^{0,\epsilon}_{\delta-\tau}\circ \Phi^{\frac{1}{\epsilon},\epsilon}_{\tau}\big)^k(u_0) \quad \text{for}\quad k\delta \leq t <(k+1)\delta
\end{equation}
The theorem below shows the accuracy of FLAVORS for
 $\delta \ll h_0$, $\tau \ll \epsilon^\nu$ and $\big(\frac{\tau}{\epsilon^\nu}\big)^2 \ll\delta \ll \frac{\tau}{\epsilon^\nu}$.

\begin{Theorem}\label{thmdssdr2sddsd}
Let $u_t^{\frac{1}{\epsilon},\epsilon}$ be the solution to \eref{jkhgjgdshhh} with $\alpha=1/\epsilon$ and $\bar{u}_t$ be defined by \eref{ksjhjhhjdskdjjhhuhuuwue}. Assume that Conditions \ref{lsdsddcddsdeehx} and \ref{lsduisdsdsdeehx} are satisfied
 then
 \begin{itemize}
\item $u_t^{\frac{1}{\epsilon},\epsilon}$ $F$-converges towards $\eta^{-1}*\big(\delta_{X_t}\otimes \mu(X_t,dy)\big)$ as $\epsilon \downarrow 0$ where $X_t$ is the solution to
\begin{equation}\label{skdljhdkdedehdjksdh}
\dot{X}_t=\int_{\R^p} g(X_t,y)\,\mu(X_t,dy)\quad \quad X_0=x_0 ,
\end{equation}
\item As $\epsilon \downarrow 0$, $\tau \epsilon^{-\nu} \downarrow 0$,  $\delta \frac{\epsilon^\nu}{\tau } \downarrow 0$, $\frac{\tau^2}{\epsilon^{2\nu} \delta} \downarrow 0$, $\bar{u}_t$ $F$-converges towards $\eta^{-1}*\big(\delta_{X_t}\otimes \mu(X_t,dy)\big)$ as $\epsilon \downarrow 0$ where $X_t$ is the solution of \eref{skdljhdkdedehdjksdh}.
\end{itemize}
\end{Theorem}
\begin{proof}
The proof of Theorem \ref{thmdssdr2sddsd} is similar to that of Theorem \ref{thm01} and \ref{thm04}. Only the idea of the proof will be given here.
The condition $\epsilon \ll 1$ is needed for the approximation of $u^{\alpha,\epsilon}$ by $u^{\alpha,0}$ and for the $F$-convergence of $u^{\frac{1}{\epsilon},0}$.
Since $y^\alpha_t=\eta^y(u^{\alpha,0}_t)$ the condition
  $\tau \ll \epsilon^\nu$ is used along with equation  \eref{lsfasshjhhddlfskjlkd} for the accuracy of $\Phi^{\frac{1}{\epsilon},\epsilon}_{\tau}$ in (locally) approximating  $y^\alpha_t$. The condition $\delta \ll \frac{\tau}{\epsilon^\nu}$ allows for the averaging of $g$  to take place prior to a significant change of $x^\alpha_t$; more precisely, it allows for $m\gg1$ iterations  of  $\Phi^{\frac{1}{\epsilon},\epsilon}_{\tau}$ prior to a significant change of $x^\alpha_t$.
  The condition $\big(\frac{\tau}{\epsilon^\nu}\big)^2\ll\delta$ is required in order to control the error accumulated  by $m$ iterations of $\Phi^{\frac{1}{\epsilon},\epsilon}_{\tau}$.
\end{proof}

\section{Deterministic mechanical systems: Hamiltonian equations}\label{jhhsghdjshdjg}

Since averaging with FLAVORS is obtained by flow composition, FLAVORS have an inherent extension to multiscale structure preserving integrators  for stiff Hamiltonian systems, i.e. ODEs of the form
\begin{equation}\label{ksdjjsdshkdjjksdhj}
\dot{p}=-\partial_q H(p,q) \quad \dot{q}=\partial_p H(p,q)
\end{equation}
where the Hamiltonian
\begin{equation}\label{kdshdkjhsjkdhjdhhhc}
H(q,p):=\frac{1}{2}p^T M^{-1} p+V(q)+\frac{1}{\epsilon} U(q)
\end{equation}
represents the total energy of a mechanical system with Euclidean phase space $\R^d \times \R^d$ or  a cotangent bundle $T^*\mathcal{M}$ of  a configuration manifold $\mathcal{M}$.

Structure preserving numerical methods for  Hamiltonian systems have been developed in the framework of geometric numerical integration \cite{Hairer:04, MR2132573} and variational integrators \cite{MaWe:01, LeMaOrWe2004b}. \emph{ The subject of geometric numerical integration deals with numerical integrators that preserve geometric properties of the flow of a differential equation, and it explains how structure preservation leads to an improved long-time behavior} \cite{MR2249159}. Variational integration theory derives integrators for mechanical
systems from discrete variational principles and are characterized by a  discrete Noether theorem. These  methods have excellent energy behavior over long integration runs because they are symplectic, i.e., by backward error analysis, they simulate a nearby mechanical system instead of
nearby differential equations. Furthermore, statistical properties of the dynamics such as Poincar\'{e} sections are well preserved even
with large time steps \cite{MR2496560}. Preservation of structures is especially important for long time simulations. Consider integrations of a harmonic oscillator, for example: no matter how small a time step is used, the amplitude given by Forward Euler / Backward Euler will increase / decrease unboundedly, whereas the amplitude given by Variational Euler (also known as symplectic Euler) will be oscillatory with a variance controlled by the step length.

These long term behaviors of structure-preserving numerical integrators motivated their extension to multiscale or stiff Hamiltonian systems. We refer to \cite{MR2275175} for a recent review on numerical integrators for highly oscillatory Hamiltonian systems. \emph{Symplectic integrators are natural for the integration of Hamiltonian systems since they
reproduce at the discrete level an important geometric property of the exact flow} \cite{LeBris:07}. For symplectic  integrators primarily for (but not limited to) stiff quadratic potentials, we refer to the Impulse Method, the Mollified Impulse Method, and their variations \cite{Grubmuller:91,Tuckerman:92,Skeel:99,Sanz-Serna:08}, which require an explicit form of the flow map of stiff process. In the context of variational integrators, by defining a discrete Lagrangian with an explicit trapezoidal approximation of the soft potential and a midpoint approximation for the fast potential, a symplectic (IMEX---IMplicit--EXplicit) scheme for stiff Hamiltonian systems has been proposed in \cite{Stern:09}. The resulting scheme is explicit for quadratic potentials and implicit for non quadratic stiff potentials.  We also refer to Le Bris and Legoll's (Hamilton-Jacobi derived) homogenization method \cite{LeBris:07}. Asynchronous Variational
Integrators \cite{LeMaOrWe2003}  provide a way to derive conservative symplectic integrators for PDEs where the solution advances non-uniformly in time; however, stiff potentials require a fine time step discretization over the whole time evolution.  In addition, multiple time-step methods \cite{MTS78} evaluate forces to different extends of accuracies by approximating less important forces via Taylor expansions, but it has issues on long time behavior, stability and accuracy, as described in Section 5 of \cite{MR1428715}.   Fixman froze the fastest bond oscillations in polymers to remove stiffness by adding a log term resemblant of entropy-based free energy to compensate \cite{Fixman:74}. This approach is successful in studying statistics of the system, but does not always reconstruct the correct dynamics \cite{Reich2000210, PeSkYa85, BoSc95}.

Several approaches to the homogenization of Hamiltonian systems (in analogy with classical homogenization \cite{BeLiPa78, JiKoOl91}) have been proposed. We refer to $\mathcal{M}$-convergence introduced in \cite{MR1490094, MR1436164}, to the two-scale expansion of  solutions of the Hamilton-Jacobi form of Newton's equations with stiff quadratic potentials \cite{LeBris:07} and to PDE methods in weak KAM theory \cite{MR2026176}.
We also refer to \cite{CaCha:09}, \cite{Iserles02} and \cite{MR810620}.

Obtaining explicit symplectic integrators for Hamiltonian systems with  non-quadratic stiff potentials is known to be an important and nontrivial problem.
By using Verlet/leap-frog macro-solvers, methods that are symplectic on slow variables (when those variables can be identified) have been proposed in the framework of HMM (the Heterogeneous Multiscale Method) in \cite{MR2161717, CaSer08}.
A ``reversible averaging'' method has been proposed in \cite{MR1843642} for mechanical systems with separated fast and slow variables.
More recently, a reversible multiscale integration method for mechanical systems was proposed in \cite{Ariel:09} in the context of HMM.
By tracking slow variables, \cite{Ariel:09} enforces reversibility in all variables as an optimization constraint at each coarse step when minimizing the distance between the effective drift obtained from the micro-solver (in the context of HMM) and the drift of the macro-solver.
 We are also refer to \cite{SerArTs09} for HMM symmetric methods for  mechanical systems with a stiff potentials of the form $\frac{1}{\epsilon}\sum_{j=1}^\nu g_j(q)^2$.

\subsection{FLAVORS for mechanical systems on manifolds}\label{subham}
Assume that we are given a first order accurate legacy integrator for \eref{ksdjjsdshkdjjksdhj} in which the parameter $1/\epsilon$ can be controlled, i.e. a
mapping $\Phi_h^{\alpha}$ acting on the phase space such that for $h\leq h_0 \min(1,\alpha^{-\frac{1}{2}})$
\begin{equation}\label{lsjfassddlfskjlkd}
\Big|\Phi_h^{\alpha}(q,p)-(q,p)-h \big(M^{-1}p,-V(q)-\alpha U(q)\big) \Big|\leq C h^2 (1+\alpha)
\end{equation}

Write $\Theta_\delta$, the FLAVOR discrete mapping approximating solutions of  \eref{ksdjjsdshkdjjksdhj} over time steps $\delta \gg \epsilon$, i.e.
\begin{equation}\label{jhjhgjhgjg87g87}
(q_{(n+1)\delta}, p_{(n+1)\delta}):=\Theta_\delta (q_{n\delta},p_{n\delta}) .
\end{equation}
FLAVOR can then be defined by
\begin{equation}\label{jhjhgjhgjg87g87A}
\Theta_\delta:=\Phi_{\delta -\tau}^{0}\circ \Phi_{\tau}^{\frac{1}{\epsilon}}
\end{equation}
 Theorem \ref{thmdssdr2sddsd} establishes the accuracy of this integrator under Conditions \ref{lsdsddcddsdeehx} and \ref{lsduisdsdsdeehx} provided that $\tau \ll\sqrt{\epsilon} \ll\delta$ and $\frac{\tau^2}{\epsilon}\ll\delta \ll \frac{\tau}{\sqrt{\epsilon}}$.

 \subsubsection{Structure preserving properties of FLAVORS}
We will now show that FLAVORS inherit the structure preserving properties of their legacy integrators.
\begin{Theorem}\label{jshgdsjhdggwuyege}
If for all $h,\epsilon>0$ $\Phi^\epsilon_h$ is symmetric under a group action,  then $\Theta_\delta$ is symmetric under the same group action.
\end{Theorem}
\begin{Theorem}\label{CorollarySymplectic}
If $\Phi^\alpha_h$ is symplectic on the co-tangent bundle $T^*\mathcal{M}$ of  a configuration manifold $\mathcal{M}$, then $\Theta_\delta$ defined by \eref{jhjhgjhgjg87g87A}  is symplectic on the co-tangent bundle $T^*\mathcal{M}$.
\end{Theorem}
Theorem \ref{jshgdsjhdggwuyege} and Theorem \ref{CorollarySymplectic} can be
resolved by noting that ``the overall method is symplectic - as a composition of symplectic transformations, and it is symmetric - as a
symmetric composition of symmetric steps'' (see Chapter XIII.1.3 of \cite{Hairer:04}).

Write
\begin{equation}
\Phi^*_h:= \big(\Phi_{-h}\big)^{-1}
\end{equation}
Let us recall the following definition corresponding to definition 1.4 of the Chapter V of \cite{Hairer:04}
\begin{Definition}
A numerical one-step method $\Phi_h$ is called  time-reversible if it satisfies $\Phi^*_h=\Phi_h$.
\end{Definition}

The following theorem, whose proof is straightforward, shows how to derive a ``symplectic and symmetric and time-reversible'' FLAVOR from a symplectic legacy integrator and its adjoint.
Since this derivation applies to manifolds, it also leads to structure-preserving FLAVORS for constrained mechanical systems.
\begin{Theorem}\label{sjhdsjdjhsgdjh}
If $\Phi^\alpha_h$ is symplectic on the co-tangent bundle $T^*\mathcal{M}$ of  a configuration manifold $\mathcal{M}$, then
\begin{equation}\label{kjhdwsjhdg}
\Theta_\delta:=\Phi_{\frac{\tau}{2}}^{\frac{1}{\epsilon},*}\circ \Phi_{\frac{\delta -\tau}{2}}^{0,*}\circ \Phi_{\frac{\delta -\tau}{2}}^{0}\circ \Phi_{\frac{\tau}{2}}^{\frac{1}{\epsilon}}
\end{equation}
 is symplectic and time-reversible on the co-tangent bundle $T^*\mathcal{M}$.
\end{Theorem}

\begin{Remark}
Observe  that (except for the first and last steps) iterating $\Theta_\delta$ defined by \eref{kjhdwsjhdg} is equivalent to
iterating
\begin{equation}\label{kjhdwsdehdg}
\Theta_\delta:= \Phi_{\frac{\delta -\tau}{2}}^{0,*}\circ \Phi_{\frac{\delta -\tau}{2}}^{0}\circ \Phi_{\frac{\tau}{2}}^{\frac{1}{\epsilon}}\circ \Phi_{\frac{\tau}{2}}^{\frac{1}{\epsilon},*}
\end{equation}
It follows that a symplectic, symmetric and reversible FLAVOR can be obtained in a nonintrusive way from a St\"{o}rmer/Verlet integrator for \eref{ksdjjsdshkdjjksdhj} \cite{MR2249159, MR1227985, Verlet1967}.
\end{Remark}

\subsubsection{An example of a symplectic FLAVOR} If the phase space is $\R^d\times \R^d$, then an example of symplectic FLAVOR is obtained from Theorem \ref{CorollarySymplectic} by choosing $\Phi^\alpha_h$ to be the symplectic Euler (also known as Variational Euler or VE for short) integrator defined by
\begin{equation}\label{lsjfaijissdsdedlfskjlkd}
\Phi_h^{\alpha}(q,p)= \begin{pmatrix}q\\p\end{pmatrix} +h \begin{pmatrix}M^{-1}\Big(p-h\big(V(q)+\alpha U(q)\big)\Big)\\-V(q)-\alpha U(q)\end{pmatrix}
\end{equation}
and letting $\Theta_\delta$ be defined by \eref{jhjhgjhgjg87g87A}.

\subsubsection{An example of a symplectic and time-reversible FLAVOR}
If the phase space is the Euclidean space $\R^d\times \R^d$, then an example of symplectic and time-reversible FLAVOR is obtained by
letting $\Theta_\delta$ be defined by Equation \eref{kjhdwsjhdg} of Theorem \ref{sjhdsjdjhsgdjh} by choosing $\Phi^\alpha_h$ to be the symplectic Euler integrator defined by \eref{lsjfaijissdsdedlfskjlkd} and
\begin{equation}\label{lsjfaijissddlfskjlkd}
\Phi_h^{\alpha,*}(q,p)= \begin{pmatrix}q\\p\end{pmatrix} +h \begin{pmatrix}M^{-1}p\\ -V(q+h M^{-1}p)-\alpha U(q+h M^{-1}p)\end{pmatrix}
\end{equation}

\subsubsection{An artificial FLAVOR}\label{gjsdgjhdhjhwg}
There is not a unique way of averaging the flows of \eref{kdshdkjhsjkdhjdhhhc}. We present below an alternative method based on the freezing and unfreezing of degrees of freedom associated with fast potentials. We have called this method ``artificial'' because it is intrusive. With this method, the discrete flow approximating solutions of \eref{ksdjjsdshkdjjksdhj} is given by \eref{jhjhgjhgjg87g87} with
\begin{equation}\label{jhjhgjhgjg87sdedg87A}
\Theta_\delta:=\theta^{tr}_{\delta-\tau} \circ \theta^\epsilon_{\tau} \circ \theta^V_{\delta}
\end{equation}
where $\theta^V_\delta$ is a symplectic map corresponding to the flow of $H^{slow}(q,p):=V(q)$, approximating the effects of the soft potential on momentum over the mesoscopic time step $\delta$ and defined by
\begin{equation}
    \theta^V_\delta \big(q,p\big) = \big(q,p-\delta \nabla V(q)\big).
\end{equation}
$\theta^\epsilon_\tau$ is a symplectic map approximating the flow of $H^{fast}(q,p):=\frac{1}{2}p^T M^{-1} p+\frac{1}{\epsilon}U(q)$ over a  microscopic time step $\tau$:
\begin{equation}
    \theta^\epsilon_\tau\big(q,p\big) = \big(q+\tau M^{-1}p,p-\frac{\tau}{\epsilon} \nabla U(q+tM^{-1}p)\big)
\end{equation}
$\theta^{tr}_{\delta-\tau}$ is a map approximating the flow of the Hamiltonian $H^{free}(q,p):=\frac{1}{2}p^T M^{-1} p$ under holonomic constraints  imposing the freezing of stiff variables.  Velocities along the direction of constraints have to be stored and set to be 0 before the constrained dynamics, i.e., frozen, and the stored velocities should be restored after the constrained dynamics, i.e., unfrozen; geometrically speaking, one projects to the constrained sub-symplectic manifold, runs the constrained dynamics, and lifts back to the original full space. Oftentimes, the exact solution to the constrained dynamics can be found (examples given in Subsections \ref{sjhgdkjhgsdjh}, \ref{sjhgdkjhgsdjh2}, \ref{kjkjdhkdjhhkdh}, \ref{FPUsec} and \ref{Prisec}).

When the exact solution to the constrained dynamics cannot be easily found, one may want to employ integrators for constrained dynamics such as SHAKE \cite{SHAKE} or RATTLE \cite{RATTLE} instead. This has to be done with caution, because symplecticity of the translational flow may be lost. The composition of projection onto the constrained manifold (freezing), evolution on the constrained manifold, and lifting from it to the unconstrained space (unfreezing) preserves symplecticity in the unconstrained space only if the evolution on the constrained manifold preserves the inherited symplectic form. A numerical integration preserves the discrete symplectic form on the constrained manifold, but not necessarily the projected continuous symplectic form.

\begin{Remark}
    This artificial FLAVOR is locally a perturbation of nonintrusive FLAVORS.  By splitting theory \cite{McQu02,Hairer:04},
    \begin{equation}
        \theta^{tr}_{\delta-\tau} \circ \theta^\epsilon_{\tau} \circ \theta^V_{\delta} \approx
        \theta^{tr}_{\delta-\tau} \circ \theta^V_{\delta-\tau} \circ \theta^\epsilon_{\tau} \circ \theta^V_{\tau} \approx
        \theta^{tr}_{\delta-\tau} \circ \theta^V_{\delta-\tau} \circ \Phi^{\frac{1}{\epsilon}}_\tau
    \end{equation}
    whereas $\Phi^0_{\delta-\tau} \circ \Phi^{\frac{1}{\epsilon}}_\tau \approx \theta^{free}_{\delta-\tau} \circ \theta^V_{\delta-\tau} \circ \Phi^{\frac{1}{\epsilon}}_\tau$, where $\theta^{free}$ is the flow of $H^{free}(q,p)$ under no constraint. The only difference is that constraints are treated in $\theta^{tr}$ but not in $\theta^{free}$.
\end{Remark}

\begin{Remark}
    This artificial FLAVOR can be formally regarded as $\Phi_{\delta-\tau}^{\infty} \circ \Phi_{\tau}^{\frac{1}{\epsilon}}$. In contrast natural FLAVOR is $\Phi_{\delta-\tau}^{0} \circ \Phi_{\tau}^{\frac{1}{\epsilon}}$.
\end{Remark}

The advantage of this artificial FLAVOR lies in the fact that only
$\tau \ll\sqrt{\epsilon} \ll\delta$ and $\delta \ll \frac{\tau}{\sqrt{\epsilon}}$ are required for its accuracy (and not $\frac{\tau^2}{\epsilon}\ll\delta$).  We also observe that, in general, artifical FLAVOR overperforms nonintrusive FLAVOR in FPU long time ($\mathcal{O}(\omega^2)$) simulations (we refer to Subsection \ref{FPUsec}).

\subsection{Variational derivation of FLAVORS}
FLAVORS based on variational legacy integrators \cite{MaWe:01} are variational too. Recall that discrete Lagrangian $L_d$ is an approximation of the integral of the continuous Lagrangian over one time step, and Discrete Euler-Lagrangian equation (DEL) corresponds to the critical point of the discrete action, which is a sum of the approximated integrals. The following diagram commutes:
\[
\xymatrix{
Singlescale~L_d
\ar[rrr]^{FLAVORization}
\ar[d]^{action~principle}
&&&
Multiscale~L_d
\ar[d]^{action~principle}
\\
Singlescale~DEL
\ar[rrr]^{FLAVORization}
&&&
Multiscale~DEL
}
\]

For example, recall Variational Euler (i.e. symplectic Euler) for system \eref{kdshdkjhsjkdhjdhhhc} with time step $h$
\begin{equation}
    \begin{cases}
        p_{k+1} &= p_k-h [\nabla V(q_k)+\frac{1}{\epsilon}\nabla U(q_k)] \\
        q_{k+1} &= q_k+h p_{k+1}
    \end{cases}
\end{equation}
can be obtained by applying variational principle to the following discrete Lagrangian
\begin{equation}
    {L_d}_h^{1/\epsilon}(q_k,q_{k+1})=h\left[ \frac{1}{2} \left(\frac{q_{k+1}-q_k}{h}\right)^2 -     \left(V(q_k)+\frac{1}{\epsilon}U(q_k)\right) \right] .
\end{equation}
Meanwhile, FLAVORized Variational Euler with smallstep $\tau$ and mesostep $\delta$
\begin{equation}
    \begin{cases}
        p'_k &=p_k-\tau [\nabla V(q_k)+\frac{1}{\epsilon}\nabla U(q_k)] \\
        q'_k &=q_k+\tau p'_k \\
        p_{k+1} &= p'_k-(\delta-\tau) \nabla V(q'_k) \\
        q_{k+1} &= q'_k+(\delta-\tau) p_{k+1}
    \end{cases}
\end{equation}
can be obtained by applying variational principle to the FLAVORized discrete Lagrangian
\begin{align}
    & {L_d}_\delta(q_k,q'_k,q_{k+1})= {L_d}_\tau^{1/\epsilon}(q_k,q'_k) + {L_d}_{\delta-\tau}^{0}(q'_k,q_{k+1}) \nonumber\\
    & \quad =\tau \left[ \frac{1}{2} \left(\frac{q'_k-q_k}{\tau}\right)^2 -     \left(V(q_k)+\frac{1}{\epsilon}U(q_k)\right) \right] + (\delta-\tau) \left[ \frac{1}{2} \left(\frac{q_{k+1}-q'_k}{\delta-\tau}\right)^2 - V(q'_k) \right]
\end{align}

FLAVORizations of other variational integrators such as Velocity Verlet follow similarly.

\section{SDEs}\label{sjhdjhsgdjhwghe}

\emph{Asymptotic problems for stochastic differential equations arose and were solved simultaneously with the very beginnings of the theory of such equations} \cite{MR1020057}. Here, we refer to the early work of Gikhman \cite{MR0073876}, Krylov \cite{KrBo1937, Kry39}, Bogolyubov \cite{MR0043001} and Papanicolaou-Kohler \cite{PaKo74}. We refer in particular to Skorokhod's detailed monograph \cite{MR1020057}.
As for ODEs,  effective equations for stiff SDEs can be obtained by averaging the instantaneous coefficients (drift and the diffusivity matrix squared) with respect to the fast components; we refer to Chapter II, Section 3 of \cite{MR1020057} for a detailed analysis including error bounds.
Numerical methods such as HMM \cite{MR2165382} and equation-free methods \cite{MR2174834} have been extended to SDEs based on this averaging principle. \emph{Implicit methods in general fail to capture the effective dynamics of the slow time scale because they cannot correctly capture non-Dirac invariant distributions} \cite{LiAbE:08} (we refer to non-Dirac invariant distribution as a measure of probability on the configuration space whose support is not limited to a single point). Another idea is to treat fast variables by conditioning; here, we refer to optimal prediction \cite{MR1915310, MR1750741, MR1619894} that has also been used for model reduction. We also refer to \cite{MR1447964, MR1704287, MR1834774, MR1835724, MR1876404,LiAbE:08, MR2385896}.

Since FLAVORS are obtained via flow averaging, they have a natural extension to SDEs developed in this section. As for ODEs, FLAVORS are directly applied to SDEs with mixed (hidden) slow and fast variables without prior (analytical or numerical) identification of slow variables. Furthermore, they can be implemented using a pre-existing scheme by turning on and off the stiff parameters.

For the sake of clarity, we will start the description of with the following SDE on $\R^d$:
\begin{equation}\label{fullsystemS}
d u^\epsilon_t=\big(G(u^\epsilon_t)+\frac{1}{\epsilon}F(u^\epsilon_t)\big)\,dt + \big(H(u^\epsilon_t)+\frac{1}{\sqrt{\epsilon}}K(u^\epsilon_t)\big)\,dW_t,\quad u^\epsilon_0=u_0
\end{equation}
where $(W_t)_{t\geq 0}$ is a $d$-dimensional Brownian Motion; $F$ and $G$ are  vector fields on $\R^d$; $H$ and $K$ are$d\times d$ matrix fields on $\R^d$. In Subsection \ref{natfla2}, we will consider the more general form \eref{jkhgsdejgdshhh}.

\begin{Condition}\label{lkjhaswkehljhsdjjkehx}
Assume that:
\begin{enumerate}
\item   $F, G,  H$ and $K$ are uniformly bounded and Lipschitz continuous.

\item There exists a diffeomorphism $\eta:=(\eta^x,\eta^y)$,
from $\R^d$ onto $\R^{d-p}\times \R^p$, independent of $\epsilon$, with uniformly bounded $C^1$, $C^2$ and $C^3$ derivatives, such that the process $(x^\epsilon_t,y^\epsilon_t)=(\eta^x(u^\epsilon_t),\eta^y(u^\epsilon_t))$ satisfies the SDE
\begin{equation}\label{kfgfgdsawssdeedejhd}
\begin{cases}
d x^\epsilon=g(x^\epsilon,y^\epsilon)\,dt + \sigma(x^\epsilon,y^\epsilon) dW_t ,\quad x^\epsilon_0=x_0\\
d y^\epsilon=\frac{1}{\epsilon}f(x^\epsilon,y^\epsilon)\,dt + \frac{1}{\sqrt{\epsilon}}Q(x^\epsilon,y^\epsilon) dW_t ,\quad y^\epsilon_0=y_0
\end{cases}
\end{equation}
where $g$ is $d-p$ dimensional vector field; $f$ a $p$-dimensional vector field; $\sigma$  is a $(d-p)\times d$-dimensional matrix field; $Q$ a $p\times d$-dimensional matrix field and $W_t$ a $d$-dimensional Brownian Motion.
\item  Let $Y_t$ be the solution to
\begin{equation}
dY_t=f(x_0,Y_t)\,dt+Q(x_0,Y_t)\, dW_t\quad \quad Y_0=y_0
\end{equation}
there exists a family of probability measures $\mu(x,dy)$ on $\R^p$ indexed by $x\in\R^{d-p}$ and a positive function $T \mapsto E(T)$ such that $\lim_{T\rightarrow \infty}E(T)=0$
 and such that for all $x_0,y_0,T$ and $\phi$ with uniformly bounded $C^r$ derivatives for $r\leq 3$,
    \begin{equation}\label{kshlkshe}
    \Big|\frac{1}{T}\int_0^T \E\big[\phi(Y_s)\big]-\int \phi(y)\mu(x_0,dy)\Big|\leq  \chi\big(\|(x_0,y_0)\|\big) E(T) \max_{r\leq 3}\|\phi\|_{C^r}
    \end{equation}
where $r\mapsto \chi(r) $ is  bounded on compact sets.
\item For all $u_0$, $T>0$,  $\sup_{0\leq t\leq T} \E\Big[\chi\big(\|u_{t}^\epsilon\|\big)\Big]$ is uniformly bounded in $\epsilon$.
\end{enumerate}
\end{Condition}
\begin{Remark}
As in the proof of Theorem \ref{thm01} the uniform regularity of $F$, $G$, $H$ and $K$ can be relaxed to local regularity by adding a control on the rate of escape of the process towards infinity. To simplify the presentation, we will use the global uniform regularity.
\end{Remark}

We will now extend the definition of two-scale flow convergence introduced in Subsection \ref{saskjagsjhgs} to stochastic processes.

\subsection{Two-scale flow convergence for SDEs}\label{subsde}
Let $\big(\xi^\epsilon(t,\omega)\big)_{t\in \R^+,\omega \in \Omega}$ be a sequence of stochastic processes on $\R^d$ (progressively measurable mappings from $\R^+ \times \Omega$ to $\R^d$) indexed by $\epsilon>0$. Let $(X_t)_{t\in \R^+}$ be a (progressively measurable) stochastic process on $\R^{d-p}$ ($p\geq 0$).
Let $x \mapsto \nu(x,dz)$ be a function from $\R^{d-p}$ into the space of probability measures on $\R^d$.

\begin{Definition}
We say that the process $\xi^\epsilon_t$ F-converges to $\nu(X_t,dz)$ as $\epsilon \downarrow 0$ and write $\xi^\epsilon_t \xrightarrow [\epsilon \rightarrow 0]{F} \nu(X_t,dz)$ if and only if for all function $\varphi$ bounded and uniformly Lipshitz-continuous on $\R^d$, and for all $t>0$
\begin{equation}
\lim_{h\rightarrow 0} \lim_{\epsilon \rightarrow 0} \frac{1}{h}\int_{t}^{t+h}\E\big[\varphi(\xi_s^\epsilon)\big]\,ds=
\E\big[\int_{\R^d} \varphi(z)\nu(X_t,dz)\big]
\end{equation}
\end{Definition}

\subsection{Non intrusive FLAVORS for SDEs}
Let $\omega$ be a random sample from a probability space $(\Omega,\mathcal{F},\P)$ and $\Phi^\alpha_h(,\omega)$
 a random mapping from $\R^d$ onto $\R^d$ approximating the flow of \eref{fullsystemS} for  $\alpha=1/\epsilon$.
If the parameter $\alpha$ can be controlled, then $\Phi^\alpha_h$ can be used as a black box  for accelerating the computation of solutions of \eref{fullsystemS} without prior identification of slow variables. Indeed, assume that there exists a constant $h_0>0$ and a normal random vector $\xi(\omega)$ such that for $h\leq h_0 \min(\frac{1}{\alpha},1)$
\begin{equation}\label{hgfjdshgfadsjhiiigdfhf}
\Bigg(\E\Big[\big|\Phi^{\alpha}_{h}(u,\omega)-u-h G(u)-\alpha h F(u)-\sqrt{h}H(u)\xi(\omega)-\sqrt{\alpha h}K(u)\xi(\omega)\big|^2\Big]\Bigg)^\frac{1}{2}\leq C h^\frac{3}{2}(1+\alpha)^\frac{3}{2}
\end{equation}
then  FLAVOR can be defined as the algorithm simulating the stochastic process
\begin{equation}\label{ksjahussahshdskdcrrfsdaue0}
\begin{cases}
\bar{u}_{0}=u_0\\
\bar{u}_{(k+1)\delta}= \Phi^0_{\delta-\tau}(.,\omega_k') \circ \Phi^\frac{1}{\epsilon}_{\tau}(\bar{u}_{k\delta},\omega_k)\\
\bar{u}_{t}=\bar{u}_{k\delta}\quad \text{for}\quad k\delta \leq t <(k+1)\delta
\end{cases} .
\end{equation}
where  $\omega_k,\omega_k'$ are i.i.d. samples from the probability space $(\Omega,\mathcal{F},\P)$, $\delta \leq h_0$ and $\tau \in (0,\delta)$ such that $\tau \leq \tau_0 \epsilon$.
Theorem \ref{thm04} establishes the asymptotic accuracy of FLAVOR for $\tau \ll \epsilon \ll\delta$ and
\begin{equation}
\big(\frac{\tau}{\epsilon}\big)^\frac{3}{2}\ll\delta \ll \frac{\tau}{\epsilon}.
\end{equation}

\subsection{Convergence theorem}

\begin{Theorem}\label{thm04}
Let $u^\epsilon$ be the solution to \eref{fullsystemS} and
$\bar{u}_t$ defined by \eref{ksjahussahshdskdcrrfsdaue0}. Assume that equation \eref{hgfjdshgfadsjhiiigdfhf} and Conditions \ref{lkjhaswkehljhsdjjkehx} are satisfied,
 then
 \begin{itemize}
\item $u^\epsilon_t$ $F$-converges towards $\eta^{-1}*\big(\delta_{X_t}\otimes \mu(X_t,dy)\big)$ as $\epsilon \downarrow 0$ where $X_t$ is the solution to
\begin{equation}\label{sklghgjhdsdsdkqhdjkh}
dX_t =\int g(X_t,y)\,\mu(X_t,dy)\,dt + \bar{\sigma}(X_t)\,dB_t \quad \quad X_0=x_0
\end{equation}
where $\bar{\sigma}$ is a $(d-p)\times (d-p)$ matrix field defined by
\begin{equation}\label{sklghgsdjhdkhdjkhqsw}
\bar{\sigma}\bar{\sigma}^T= \int \sigma \sigma^T(x,y)\,\mu(x,dy)
\end{equation}
and $B_t$ a $(d-p)$-dimensional Brownian Motion.
\item $\bar{u}_t$ $F$-converges towards $\eta^{-1}*\big(\delta_{X_t}\otimes \mu(X_t,dy)\big)$ as
 $\epsilon \downarrow 0$, $\tau \leq \delta$, $\frac{\tau}{\epsilon} \downarrow 0$, $\frac{\delta \epsilon }{\tau} \downarrow 0$ and $\big(\frac{\tau}{\epsilon}\big)^\frac{3}{2}\frac{1}{\delta} \downarrow 0$.
\end{itemize}
\end{Theorem}
The proof of convergence of SDEs of  type \eref{kfgfgdsawssdeedejhd} is classical, and
a comprehensive monograph can be found in  Chapter II of \cite{MR1020057}.
A proof of (mean squared) convergence of  HMM  applied to \eref{kfgfgdsawssdeedejhd} (separated slow and fast variables) with $\sigma=0$ has been obtained in
\cite{MR2165382}. A  proof of (mean squared) convergence of the Equation-Free Method applied to \eref{kfgfgdsawssdeedejhd} with $\sigma\not=0$ but independent of fast variables has been obtained in \cite{GiKeKup06}.
Theorem \ref{thm04} proves the convergence in distribution  of FLAVOR applied to  SDE \eref{fullsystemS} with hidden slow and fast processes.
One of the main difficulties of the proof of Theorem \ref{thm04} lies in the fact that we are not assuming that the noise on (hidden) slow variables is null or independent from fast variables. Without this assumption, $x^\epsilon_t$ converges only weakly towards $X_t$, the convergence of $u^\epsilon$ can only be weak and techniques for strong convergence can not be used. The proof of Theorem \ref{thm04} relies on a powerful result by Skorokhod (Theorem 1 of Chapter II of \cite{MR1020057}) stating that the convergence in distribution of a sequence of stochastic processes is implied by the convergence of their generators.
We refer to Subsection \ref{iuiuyuuyuyuytuty} of the appendix for the detailed proof of Theorem \ref{thm04}.

\subsection{Natural FLAVORS}
As for ODEs, it is not necessary to use legacy integrators to obtain FLAVORS for SDEs.
More precisely, Theorem \ref{thm04} remains valid if FLAVORS are defined to be algorithms simulating the discrete process
\begin{equation}\label{ksjhdskdcrrfdaue0}
\begin{cases}
\bar{u}_{0}=u_0\\
\bar{u}_{(k+1)\delta}= \theta^G_{\delta-\tau}(.,\omega_k') \circ \theta^\epsilon_{\tau}(\bar{u}_{k\delta},\omega_k)\\
\bar{u}_{t}=\bar{u}_{k\delta}\quad \text{for}\quad k\delta \leq t <(k+1)\delta
\end{cases}
\end{equation}
where $\omega_k,\omega_k'$ are i.i.d. samples from the probability space $(\Omega,\mathcal{F},\P)$ and
$\theta_\tau^{\epsilon}$ and $\theta^{G}_{\delta-\tau}$ are two random mappings  from $\R^d$ onto $\R^d$   satisfying  following conditions \ref{lkjhewlkehlksdsedsdaeehx}. More precisely,
 $\theta^\epsilon_{\tau}(.,\omega)$  approximates in distribution the flow of \eref{fullsystemS} over time steps  $\tau\ll\epsilon$.
  $\theta^G_{h}(.,\omega)$  approximates in distribution the flow of
  \begin{equation}\label{skjdsjjhdxsdhd}
d v^\epsilon_t=G(v^\epsilon_t) \,dt+H(v^\epsilon_t)\,dW_t
\end{equation}
  over time steps  $h\ll1$.

\begin{Condition}\label{lkjhewlkehlksdsedsdaeehx}
Assume that:
\begin{enumerate}
\item There exists $h_0, C>0$ and a $d$-dimensional centered Gaussian  vector $\xi(\omega)$ with identity covariance matrix such that for $h\leq h_0$,
\begin{equation}\label{hgfjhgfadsjhgdfhf}
\Bigg(\E\Big[\big|\theta^{G}_{h}(u,\omega)-u-h G(u)-\sqrt{h}H(u)\xi(\omega)\big|^2\Big]\Bigg)^\frac{1}{2}\leq C h^\frac{3}{2}
\end{equation}

\item There exists $\tau_0, C>0$ and a $d$-dimensional centered Gaussian  vector $\xi(\omega)$ with identity covariance matrix such that for $\frac{\tau}{\epsilon}\leq \tau_0$,
\begin{equation}\label{hgfjdshgfadsjhgdfhf}
\Bigg(\E\Big[\big|\theta^{\epsilon}_{\tau}(u,\omega)-u-\tau G(u)-\frac{\tau}{\epsilon}F(u)-\sqrt{\tau}H(u)\xi(\omega)-\sqrt{\frac{\tau}{\epsilon}}K(u)\xi(\omega)\big|^2\Big]\Bigg)^\frac{1}{2}\leq C \big(\frac{\tau}{\epsilon}\big)^\frac{3}{2}
\end{equation}
\item For all $u_0$, $T>0$,  $\sup_{0\leq n\leq T/\delta} \E\Big[\chi\big(\|\bar{u}_{n\delta}\|\big)\Big]$ is uniformly bounded in $\epsilon$, $0< \delta \leq h_0$, $\tau \leq \min(\tau_0 \epsilon, \delta)$,
    where $\bar{u}$ is defined by \eref{ksjhdskdcrrfdaue0}.
\end{enumerate}
\end{Condition}

\subsection{FLAVORS for generic stiff SDEs}\label{natfla2}
FLAVORS for stochastic systems have a natural generalization to SDEs on $\R^d$ of the form
\begin{equation}\label{jkhgsdejgdshhh}
du^{\alpha,\epsilon}=F(u^{\alpha,\epsilon},\alpha,\epsilon)\,dt
+K(u^{\alpha,\epsilon},\alpha,\epsilon)\,dW_t
\end{equation}
where $(W_t)_{t\geq 0}$ is a $d$-dimensional Brownian Motion, $F$ and $K$ are Lipshitz continuous in $u$.

\begin{Condition}\label{lsdsddcddsdsdseeehx}
Assume that:
\begin{enumerate}
\item  $\gamma \mapsto F(u,\alpha,\gamma)$ and $\gamma\mapsto K(u,\alpha,\gamma)$
are  uniformly continuous in the neighborhood of $0$.
\item There exists a diffeomorphism $\eta:=(\eta^x,\eta^y)$,
from $\R^d$ onto $\R^{d-p}\times \R^p$, independent from $\epsilon, \alpha$, with uniformly bounded $C^1$, $C^2$ and $C^3$ derivatives, and such that the stochastic process $(x_t^{\alpha},y_t^{\alpha})=(\eta^x(u_t^{\alpha,0}),\eta^y(u_t^{\alpha,0}))$ satisfies for all $\alpha \geq 1$ the SDE
\begin{equation}\label{kfgfgdiuucdqcdciusesdejhdS}
dx^\alpha=g(x^\alpha,y^{\alpha})\,dt+\sigma(x^\alpha,y^{\alpha}) \,dW_t\quad x^\alpha_0=x_0\\
\end{equation}
where $g$ is $d-p$ dimensional vector field, $\sigma$  is a $(d-p)\times d$-dimensional matrix field,  $g$ and $\sigma$ are uniformly bounded and  Lipschitz continuous in $x$ and $y$.

\item There exists a family of probability measures $\mu(x,dy)$ on $\R^p$
   such that for all $x_0,y_0,T$ $\big((x_0,y_0):=\eta(u_0)\big)$ and $\varphi$ with uniformly bounded $C^r$ derivatives for $r\leq 3$,
    \begin{equation}
    \Big|\frac{1}{T}\int_0^T \E\big[\varphi(y^{\alpha}_s)\big]\,ds-\int \varphi(y)\mu(x_0,dy)\Big|\leq \chi\big(\|(x_0,y_0)\|\big) \big(E_1(T)+E_2(T \alpha^\nu)\big) \max_{r\leq 3}\|\varphi\|_{C^r}
    \end{equation}
where $r\mapsto \chi(r)$ is bounded on compact sets and $E_2(r)\rightarrow 0$ as $r\rightarrow \infty$ and $E_1(r)\rightarrow 0$ as $r\rightarrow 0$.
\item For all $u_0$, $T>0$,  $\sup_{0\leq t\leq T} \E\Big[\chi\big(\|u_{t}^{\alpha,0}\|\big)\Big]$ is uniformly bounded in $\alpha\geq 1$.
\end{enumerate}
\end{Condition}
\begin{Remark}
As in the proof of Theorem \ref{thm01}, the uniform regularity of $g$ and $\sigma$ can be relaxed to local regularity by adding a control on the rate of escape of the process towards infinity. To simplify the presentation, we have use the global uniform regularity.
\end{Remark}

Let $\omega$ be a random sample from a probability space $(\Omega,\mathcal{F},\P)$ and
 $\Phi_h^{\alpha,\epsilon}(.,\omega)$  a random mapping from $\R^d$ onto $\R^d$ approximating in distribution the flow of \eref{jkhgsdejgdshhh} over time steps  $\tau\ll\epsilon$.
If the parameter $\alpha$ can be controlled, then $\Phi_h^{\alpha,\epsilon}$ can be used as a black box for accelerating the computation of solutions of \eref{jkhgsdejgdshhh}. The acceleration is obtained without prior identification of the slow variables.

\begin{Condition}\label{lsduisdsdsdjgsahaagseehx}
Assume that:
\begin{enumerate}
\item There exists $h_0, C>0$ and a $d$-dimensional centered Gaussian  vector $\xi(\omega)$ with identity covariance matrix such that for $h\leq h_0$, $0<\epsilon \leq 1 \leq \alpha$ and $h\leq h_0 \min(\frac{1}{\alpha^\nu},1) $
\begin{equation}\label{lsfasshjhhddlfssadkjlkd}
\Bigg(\E\Big[\big|\Phi_h^{\alpha,\epsilon}(u)-u-h F(u,\alpha,\epsilon)-\sqrt{h}\xi(\omega)K(u,\alpha,\epsilon) \big|^2 \Bigg)^\frac{1}{2}\leq C h^\frac{3}{2} (1+\alpha^{\frac{3\nu}{2}})
\end{equation}
\item For all $u_0$, $T>0$,  $\sup_{0\leq n\leq T/\delta} \E\Big[\chi\big(\|\bar{u}_{n\delta}\|\big)\Big]$ is uniformly bounded in $\epsilon$, $0< \delta \leq h_0$, $\tau \leq \min(h_0 \epsilon^\nu, \delta)$,
    where $\bar{u}$ is defined by \eref{ksjhdskdcrsdrfdaue0}.
\end{enumerate}
\end{Condition}

\paragraph{FLAVORS}
Let  $\delta \leq h_0$ and $\tau \in (0,\delta)$ such that $\tau \leq \tau_0 \epsilon^\nu$.
We define FLAVORS as the class of algorithms simulating the  stochastic process  $t\mapsto \bar{u}_{t}$ defined by
\begin{equation}\label{ksjhdskdcrsdrfdaue0}
\begin{cases}
\bar{u}_{0}=u_0\\
\bar{u}_{(k+1)\delta}= \Phi^{0,\epsilon}_{\delta-\tau}(.,\omega_k') \circ \Phi^{\frac{1}{\epsilon},\epsilon}_{\tau}(\bar{u}_{k\delta},\omega_k)\\
\bar{u}_{t}=\bar{u}_{k\delta}\quad \text{for}\quad k\delta \leq t <(k+1)\delta
\end{cases}
\end{equation}
where $\omega_k,\omega_k'$ are i.i.d. samples from the probability space $(\Omega,\mathcal{F},\P)$.
\begin{Remark}
$\omega_k$ simulates the randomness of the increment of the Brownian Motion between times $\delta k$ and $\delta k+\tau$. $\omega_k'$ simulates the randomness of the increment of the Brownian Motion between times  $\delta k+\tau$ and $\delta (k+1)$.
The independence of $\omega_k$ and $\omega_k'$ is reflection of the independence of the increments of a Brownian Motion.
\end{Remark}

The following theorem  shows that the flow averaging integrator is accurate with respect to $F$-convergence for $\tau \ll \epsilon^\nu \ll\delta$ and
\begin{equation}
\big(\frac{\tau}{\epsilon^\nu}\big)^\frac{3}{2}\ll\delta \ll \frac{\tau}{\epsilon^\nu} .
\end{equation}

\begin{Theorem}\label{thmdssdrgefd2sddsd}
Let $u_t^{\frac{1}{\epsilon},\epsilon}$ be the solution to \eref{jkhgsdejgdshhh} with $\alpha=1/\epsilon$ and $\bar{u}_t$ be defined by \eref{ksjhdskdcrsdrfdaue0}. Assume that Conditions \ref{lsdsddcddsdsdseeehx} and \ref{lsduisdsdsdjgsahaagseehx} are satisfied
 then
 \begin{itemize}
\item $u_t^{\frac{1}{\epsilon},\epsilon}$ $F$-converges towards $\eta^{-1}*\big(\delta_{X_t}\otimes \mu(X_t,dy)\big)$ as $\epsilon \downarrow 0$ where $X_t$ is the solution to
\begin{equation}\label{skdljhdksddedehdjksdh}
d X_t=\int g(X_t,y)\,\mu(X_t,dy)+ \bar{\sigma}(X_t)\,dB_t\quad \quad X_0=x_0
\end{equation}
where $\bar{\sigma}$ is a $(d-p)\times (d-p)$ matrix field defined by
\begin{equation}\label{cdew}
\bar{\sigma}\bar{\sigma}^T= \int \sigma \sigma^T(x,y)\,\mu(x,dy)
\end{equation}
and $B_t$ a $(d-p)$-dimensional Brownian Motion.

\item As $\epsilon \downarrow 0$, $\tau \epsilon^{-\nu} \downarrow 0$,  $\delta \frac{\epsilon^\nu}{\tau } \downarrow 0$, $\big(\frac{\tau}{\epsilon^\nu} \big)^\frac{3}{2}\frac{1}{\delta} \downarrow 0$, $\bar{u}_t$ $F$-converges towards $\eta^{-1}*\big(\delta_{X_t}\otimes \mu(X_t,dy)\big)$ as $\epsilon \downarrow 0$ where $X_t$ is the solution to \eref{skdljhdksddedehdjksdh}.
\end{itemize}
\end{Theorem}

\begin{proof}
The proof of Theorem \ref{thmdssdrgefd2sddsd} is similar to the proof of Theorem \ref{thm04}. The condition $\epsilon \ll 1$ is needed for the approximation of $u^{\alpha,\epsilon}$ by $u^{\alpha,0}$ and for the $F$-convergence of $u^{\frac{1}{\epsilon},0}$.
Since $y^\alpha_t=\eta^y(u^{\alpha,0}_t)$ the condition
  $\tau \ll \epsilon^\nu$ is used along with Equation  \eref{lsfasshjhhddlfssadkjlkd} for the accuracy of $\Phi^{\frac{1}{\epsilon},\epsilon}_{\tau}$ in (locally) approximating  $y^\alpha_t$. The condition $\delta \ll \frac{\tau}{\epsilon^\nu}$ allows for the averaging of $g$ and $\sigma$ to take place prior to a significant change of $x\alpha_t$; more precisely, it allows for $m\gg1$ iterations  of  $\Phi^{\frac{1}{\epsilon},\epsilon}_{\tau}$ prior to a significant change of $x\alpha_t$.
  The condition $\big(\frac{\tau}{\epsilon^\nu}\big)^\frac{3}{2}\ll\delta$ is required in order to control the error accumulated  by $m$ iterations of $\Phi^{\frac{1}{\epsilon},\epsilon}_{\tau}$.
\end{proof}

\section{Stochastic mechanical systems: Langevin equations}\label{kshskjhdklshdkjhjh}
Since the foundational work of Bismut \cite{Bi1981}, the field of stochastic geometric mechanics has grown in response to the demand for tools to analyze the structure of continuous and discrete mechanical systems with uncertainty \cite{Skeel2002, Ha2007,VaCi2006,CiLeVa2008, MiReTr2002, MiReTr2003, MiTr2004,  MR2408499, MR2385877, MR2491434, BoOw:09, BoVa:09}. Like their deterministic counterparts, these integrators are structure preserving in terms of statistical invariants.

In this section,  FLAVORS are developed to be structure preserving integrators for stiff stochastic mechanical systems, i.e., stiff Langevin equations of the form
\begin{equation}\label{jdhsjhgdjwghe}
\begin{cases}
dq=M^{-1}p\\
dp= -\nabla V(q)\,dt-\frac{1}{\epsilon}\nabla U(q)\,dt - c p\, dt+\sqrt{2 \beta^{-1}} c^\frac{1}{2} dW_t
\end{cases}
\end{equation}
and of the form
\begin{equation}\label{jsdkjdshdgjdhd}
\begin{cases}
dq=M^{-1}p\\
dp= -\nabla V(q)\,dt-\frac{1}{\epsilon}\nabla U(q)\,dt - \frac{c}{\epsilon} p\, dt+\sqrt{2 \beta^{-1}}\frac{ c^\frac{1}{2}}{\sqrt{\epsilon}} dW_t
\end{cases}
\end{equation}
where $c$ is a positive symmetric $d\times d$ matrix.
\begin{Remark}
 Provided that hidden fast variables remain locally ergodic, one can also consider Hamiltonians with a mixture of both slow and fast noise and friction. For the sake of clarity, we have restricted our presentation to \eref{jdhsjhgdjwghe} and \eref{jsdkjdshdgjdhd}.
\end{Remark}
Equations \eref{jdhsjhgdjwghe} and \eref{jsdkjdshdgjdhd} model
a mechanical system with Hamiltonian
\begin{equation}\label{kdshdkjhsjkdhfffjdhhhc}
H(q,p):=\frac{1}{2}p^T M^{-1} p+V(q)+\frac{1}{\epsilon} U(q) .
\end{equation}
The phase space is the Euclidean space  $\R^d \times \R^d$ or  a cotangent bundle $T^*\mathcal{M}$ of  a configuration manifold $\mathcal{M}$.
\begin{Remark}
If $c$ is not constant and $\mathcal{M}$ is not the usual $\R^d \times \R^d$ Euclidean space, one should use the Stratonovich integral instead of the It\^{o} integral.
\end{Remark}

\subsection{FLAVORS for stochastic mechanical systems on manifolds}\label{sublan}
As in Section \ref{jhhsghdjshdjg}, we assume that we are given a mapping $\Phi_h^{\alpha}$ acting on the phase space such that for $h\leq h_0 \min(1,\alpha^{-\frac{1}{2}})$
\begin{equation}\label{lssdkjjfdassddlfskjdjlkd}
\Big|\Phi_h^{\alpha}(q,p)-(q,p)-h \big(M^{-1}p,-V(q)-\alpha U(q)\big) \Big|\leq C h^2 (1+\alpha)
\end{equation}
Next, consider the following Ornstein-Uhlenbeck equations:
\begin{equation}\label{jkdjfghffdkjheg}
dp=-\alpha c p\,dt + \sqrt{\alpha}\sqrt{2 \beta^{-1}} c^\frac{1}{2} dW_t
\end{equation}
The stochastic flow of \eref{jkdjfghffdkjheg} is defined by the following stochastic evolution map:
\begin{equation}\label{jkdjfghffdkjhegSe}
\Psi_{t_1,t_2}^\alpha(q,p)=\Big(q,e^{-c\alpha(t_2-t_1)}p+\sqrt{2 \beta^{-1} \alpha} c^\frac{1}{2}\int_{t_1}^{t_2} e^{-c\alpha(t_2-s)} dW_s \Big)
\end{equation}
Let  $\delta \leq h_0$ and $\tau \in (0,\delta)$ such that $\tau \leq \tau_0/ \sqrt{\alpha}$. FLAVOR for \eref{jdhsjhgdjwghe} can then be defined by
\begin{equation}\label{kdfsjhdskdcrrfdaue0}
\begin{cases}
(\bar{q}_{0},\bar{p}_0)=(q_0,p_0)\\
(\bar{q}_{(k+1)\delta},\bar{p}_{(k+1)\delta})=
\Phi^{0}_{\delta-\tau}\circ \Psi_{k \delta+\tau,(k+1)\delta}^{1}\circ
\Phi^{\frac{1}{\epsilon}}_{\tau}\circ \Psi_{k\delta, k \delta+\tau}^{1}(\bar{q}_{k\delta},\bar{p}_{k\delta})
\end{cases}
\end{equation}
and FLAVOR for \eref{jsdkjdshdgjdhd} can be defined by
\begin{equation}\label{kdfsjhdiiuskdcrrfdauijhe0}
\begin{cases}
(\bar{q}_{0},\bar{p}_0)=(q_0,p_0)\\
(\bar{q}_{(k+1)\delta},\bar{p}_{(k+1)\delta})=
\Phi^{0}_{\delta-\tau}\circ
\Phi^{\frac{1}{\epsilon}}_{\tau}\circ \Psi_{k\delta, k \delta+\tau}^{\frac{1}{\epsilon}}(\bar{q}_{k\delta},\bar{p}_{k\delta})
\end{cases}
\end{equation}

 Theorem \ref{thmdssdrgefd2sddsd} establishes the accuracy of these integrators under Conditions \ref{lsdsddcddsdsdseeehx} and \ref{lsduisdsdsdjgsahaagseehx} provided that
$\tau \ll\sqrt{\epsilon} \ll\delta$ and $\big(\frac{\tau}{\sqrt{\epsilon}}\big)^\frac{3}{2}\ll\delta \ll \frac{\tau}{\sqrt{\epsilon}}$.

\subsection{Structure Preserving properties of FLAVORS for stochastic mechanical systems on manifolds}
First, observe that if $\Phi^{\alpha}_h$ and $\Psi^{\frac{1}{\epsilon}}_h$ are symmetric under a group action for all $\epsilon>0$, then the resulting FLAVOR, as a
symmetric composition of symmetric steps, is symmetric under the same  group action (see comment below Theorem \ref{sjhdsjdjhsgdjh}).

 Similarly, the following theorem shows that FLAVORS inherits  structure-preserving properties from  those associated with $\Phi^\alpha_h$ (the component approximating the Hamiltonian part of the flow).
\begin{Theorem}\label{sjhdsjdjhsgdhshhjh}
$\quad$
\begin{itemize}
\item If $\Phi^\alpha_h$ is symplectic, then the FLAVORS defined by \eref{kdfsjhdskdcrrfdaue0} and \eref{kdfsjhdiiuskdcrrfdauijhe0}
are quasi-symplectic as defined in Conditions RL1 and RL2 of \cite{MiTr2003} (it degenerates to a symplectic method if friction is set equal to zero and the Jacobian of the flow map is independent of $(q,p)$).
\item If in addition
 $c$ is isotropic then FLAVOR defined by \eref{kdfsjhdskdcrrfdaue0}  is conformally  symplectic, i.e., it preserves the precise symplectic area
change associated to the flow of inertial Langevin processes \cite{McPe2001}.
\end{itemize}
\end{Theorem}
\begin{proof}
Those properties are a consequence of the fact that FLAVORS are splitting schemes. The quasi-symplecticity and symplectic conformallity of GLA has been obtained in a similar way in \cite{BoOw:09}.
\end{proof}

\subsubsection{Example of quasi-symplectic FLAVORS}\label{qsfla} An example of quasi-symplectic FLAVOR can be obtained by choosing $\Phi^\alpha_h$ to be the symplectic Euler integrator defined by \eref{lsjfaijissdsdedlfskjlkd}. This integrator is also conformally symplectic if $c$ is isotropic and friction is slow.

\subsubsection{Example of quasi-symplectic and time-reversible FLAVORS}
Defining $\Phi^\alpha_h$ by \eref{lsjfaijissdsdedlfskjlkd} and $\Phi_h^{\alpha,*}$ by \eref{lsjfaijissddlfskjlkd}, an example of quasi-symplectic and time-reversible FLAVOR can be obtained by using the  symmetric Strang splitting:

\begin{equation}\label{kdfsjhdsskdcrrfdaue0}
(\bar{q}_{(k+1)\delta},\bar{p}_{(k+1)\delta})=\Psi_{k\delta+\frac{\delta}{2}, (k+1) \delta}^{1} \circ
\Phi_{\frac{\tau}{2}}^{\frac{1}{\epsilon},*}\circ \Phi_{\frac{\delta -\tau}{2}}^{0,*}\circ \Phi_{\frac{\delta -\tau}{2}}^{0}\circ \Phi_{\frac{\tau}{2}}^{\frac{1}{\epsilon}}
\circ \Psi_{k\delta, k \delta+\frac{\delta}{2}}^{1}(q,p)
\end{equation}
for \eref{jdhsjhgdjwghe} and
\begin{equation}\label{kdfsjdhdiiuskdcrrfdauijhe0}
(\bar{q}_{(k+1)\delta},\bar{p}_{(k+1)\delta})=
\Psi_{(k+1)\delta-\frac{\tau}{2}, (k+1) \delta}^{\frac{1}{\epsilon}} \circ
\Phi_{\frac{\tau}{2}}^{\frac{1}{\epsilon},*}\circ \Phi_{\frac{\delta -\tau}{2}}^{0,*}\circ \Phi_{\frac{\delta -\tau}{2}}^{0}\circ \Phi_{\frac{\tau}{2}}^{\frac{1}{\epsilon}}
\circ \Psi_{k\delta, k \delta+\frac{\tau}{2}}^{\frac{1}{\epsilon}}(q,p)
\end{equation}
for \eref{jsdkjdshdgjdhd}. This integrator is also conformally symplectic if $c$ is isotropic and friction is slow.

\subsubsection{Example of Boltzmann-Gibbs reversible  Metropolis-adjusted FLAVORS}
Since the probability density of $\Psi_{t_1,t_2}$ can  be explicitly computed, it follows that the probability densities of \eref{kdfsjhdsskdcrrfdaue0} and \eref{kdfsjdhdiiuskdcrrfdauijhe0} can be explicitly computed, and these algorithms can be metropolized and made reversible with respect to the Gibbs distribution as it has been shown in \cite{BoVa:09} for the Geometric Langevin Algorithm introduced in \cite{BoOw:09}. This metropolization leads to stochastically stable (and ergodic if the noise  applied on momentum is not degenerate) algorithms. We refer to \cite{BoVa:09} for details. Observe that if the proposed move is rejected, the momentum has to be flipped and the acceptance probability involves a momentum flip. It is proven in \cite{BoVa:09} that GLA \cite{BoOw:09} remains strongly accurate after a metropolization  involving local momentum flips. Whether this preservation of accuracy over trajectories transfers in a weak sense (in distributions) to FLAVORS remains to be investigated.

\section{Numerical analysis of FLAVOR based on Variational Euler}
\label{kjkjdhkdjhhkdhlalala}

\begin{figure}
\centering
\subfigure[Nonintrusive FLAVOR]{
\includegraphics[width=0.45\textwidth]{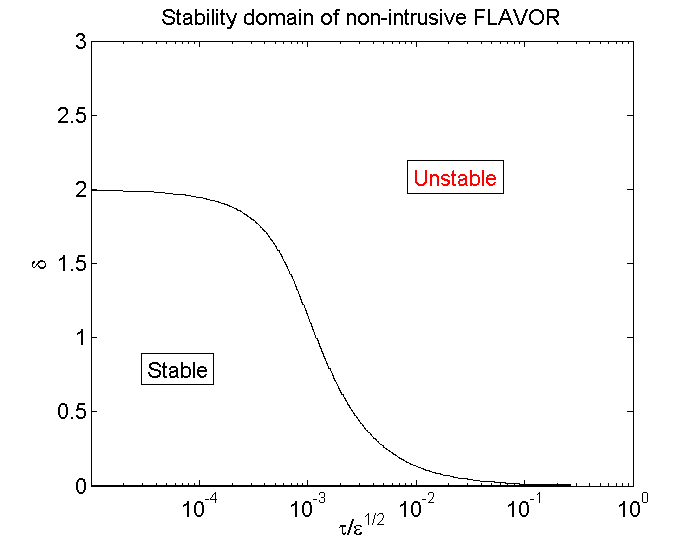}
\label{error4nonintrusive}
}
~
\subfigure[Artificial FLAVOR]{
\includegraphics[width=0.45\textwidth]{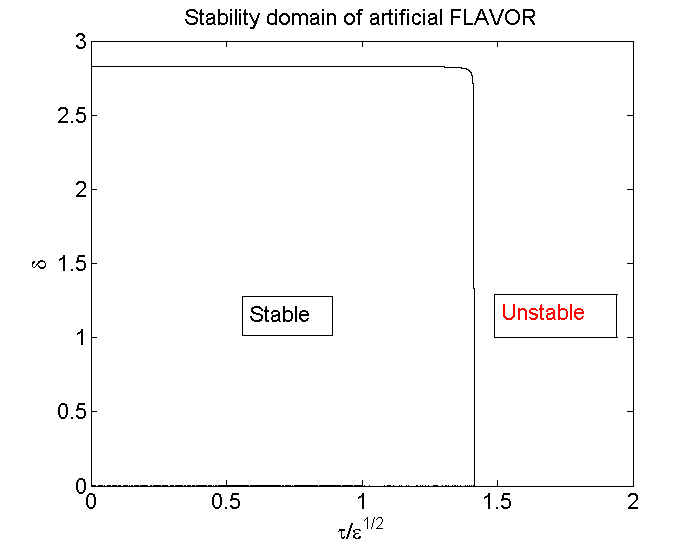}
\label{error4artificial}
}
\caption{\footnotesize Stability domain of non-intrusive and artificial FLAVOR applied to \eref{erroranalysis_example2} as a function of $\delta$ and $\tau/\epsilon$. $\omega=1/\sqrt{\epsilon}=1000$.}
\end{figure}

\subsection{Stability}\label{kjkjdhkdjhhkdhlala}
Consider the following linear Hamiltonian system
\begin{equation}
    H(x,y,p_x,p_y)=\frac{1}{2}p_x^2+\frac{1}{2}p_y^2+\frac{1}{2}x^2+\frac{\omega^2}{2}(y-x)^2
    \label{erroranalysis_example2}
\end{equation}
 with $\omega\gg 1$. Here $\frac{x+y}{2}$ is the slow variable and $y-x$ is the fast variable.

It can be shown that, when applied to \eref{erroranalysis_example2}, Symplectic  Euler \eref{lsjfaijissdsdedlfskjlkd} is stable if and only if $h \leq \sqrt{2}/\omega$.
 Write $\Theta_{\delta,\tau}$ the non-intrusive FLAVOR \eref{jhjhgjhgjg87g87A} obtained by using Symplectic Euler \eref{lsjfaijissdsdedlfskjlkd} as a Legacy integrator. Write $\Theta^a_{\delta,\tau}$
the artificial FLAVOR described in Subsection \ref{gjsdgjhdhjhwg}.

 \begin{Theorem}\label{FLAVOR_stability_sqrt2}
 The non-intrusive FLAVOR $\Theta_{\delta,\tau}$ with $1/\sqrt{\tau}\gg\omega\gg1$ is stable if and only if $\delta \in (0,2)$.\\
 The artificial FLAVOR $\Theta^a_{\delta,\tau}$ with $1/\tau\gg \omega\gg 1$ is stable if and only if   $\delta \in (0,2\sqrt{2})$.
 \end{Theorem}
 \begin{proof}
 The numerical scheme associated with $\Theta_{\delta,\tau}$ can be written as
\begin{eqnarray}\label{hsgjhdgjhsgd}
    & \begin{bmatrix} y_{n+1} \\ x_{n+1} \\ (p_y)_{n+1} \\ (p_x)_{n+1} \end{bmatrix} = T \begin{bmatrix} y_n \\ x_n \\ (p_y)_n \\ (p_x)_n \end{bmatrix}
\end{eqnarray}
with
\[
T=\begin{bmatrix} 1 & 0 & \delta-\tau & 0 \\ 0 & 1 & 0 & \delta-\tau \\ 0 & 0 & 1 & 0 \\ 0 & 0 & 0 & 1 \end{bmatrix}
    \begin{bmatrix} 1 & 0 & 0 & 0 \\ 0 & 1 & 0 & 0 \\ \tau-\delta & 0 & 1 & 0 \\ 0 & 0 & 0 & 1 \end{bmatrix}
    \begin{bmatrix} 1 & 0 & \tau & 0 \\ 0 & 1 & 0 & \tau \\ 0 & 0 & 1 & 0 \\ 0 & 0 & 0 & 1 \end{bmatrix}
    \begin{bmatrix} 1 & 0 & 0 & 0 \\ 0 & 1 & 0 & 0 \\ -\tau(\omega^2+1) & \tau\omega^2 & 1 & 0 \\ \tau\omega^2 & -\tau\omega^2 & 0 & 1 \end{bmatrix} \nonumber\\
\]
The characteristic polynomial of $T$ is
\begin{align}
    &\lambda^4+ (-4+\delta^2-\delta^2\tau^2+2\delta\tau^3-\tau^4+2\delta\tau\omega^2-\delta^2\tau^2\omega^2+2\delta\tau^3\omega^2-\tau^4\omega^2)\lambda^3+(6-2\delta^2 \nonumber\\
    &+2\delta^2\tau^2-4\delta\tau^3+2\tau^4-4\delta\tau\omega^2+\delta^3\tau\omega^2+2\delta^2\tau^2\omega^2-4\delta\tau^3\omega^2-\delta^3\tau^3\omega^2+2\tau^4\omega^2\nonumber\\ &+2\delta^2\tau^4\omega^2-\delta\tau^5\omega^2) \lambda^2 \nonumber\\
    &+(-4+\delta^2-\delta^2\tau^2+2\delta\tau^3-\tau^4+2\delta\tau\omega^2-\delta^2\tau^2\omega^2+2\delta\tau^3\omega^2-\tau^4\omega^2)\lambda+1
\end{align}
Since $\omega \gg 1$, $\tau \ll 1/\omega^2$, as long as $\delta \lesssim 1$ roots to the above polynomial are close to roots to the asymptotic polynomial
\begin{equation}
    \lambda^4+(\delta^2-4)\lambda^3+(6-2\delta^2)\lambda^2+(\delta^2-4)\lambda+1
\end{equation}
which can be shown to be $1$ with multiplicity 2 and $\frac{1}{2}(2-\delta^2 \pm \delta\sqrt{\delta^2-4})$.
It is easy to see that all roots are complex numbers with moduli less or equal to one if and only if $|\delta|\leq 2$.

The numerical scheme associated with $\Theta^a_{\delta,\tau}$ can be written as in \eref{hsgjhdgjhsgd} with
\begin{eqnarray}
 T = \begin{bmatrix} 1 & 0 & \frac{\delta-\tau}{2} & \frac{\delta-\tau}{2} \\ 0 & 1 & \frac{\delta-\tau}{2} & \frac{\delta-\tau}{2} \\ 0 & 0 & 1 & 0 \\ 0 & 0 & 0 & 1 \end{bmatrix}
    \begin{bmatrix} 1 & 0 & 0 & 0 \\ 0 & 1 & 0 & 0 \\ -\tau\omega^2 & \tau\omega^2 & 1 & 0 \\ \tau\omega^2 & -\tau\omega^2 & 0 & 1 \end{bmatrix}
    \begin{bmatrix} 1 & 0 & \tau & 0 \\ 0 & 1 & 0 & \tau \\ 0 & 0 & 1 & 0 \\ 0 & 0 & 0 & 1 \end{bmatrix}
    \begin{bmatrix} 1 & 0 & 0 & 0 \\ 0 & 1 & 0 & 0 \\ -\delta & 0 & 1 & 0 \\ 0 & 0 & 0 & 1 \end{bmatrix}
\end{eqnarray}
The characteristic polynomial of $T$ is
\begin{equation}
    2\lambda^4+(4\omega^2\tau^2+\tau\delta+\delta^2-8)\lambda^3+(12-2\delta^2-2\delta\tau-8\tau^2\omega^2+2\delta^2\tau^2\omega^2)\lambda^2+(4\omega^2\tau^2+\tau\delta+\delta^2-8)\lambda+2
\end{equation}
Since $\omega \gg 1$, $\tau \ll 1/\omega$, as long as $\delta \lesssim 1$ roots to the above polynomial are close to roots to the asymptotic polynomial
\begin{equation}
    2\lambda^4+(\delta^2-8)\lambda^3+(12-2\delta^2)\lambda^2+(\delta^2-8)\lambda+1
\end{equation}
which can be shown to be $1$ with multiplicity 2 and $\frac{1}{4}(4-\delta^2 \pm \delta\sqrt{\delta^2-8})$. All roots are complex numbers with moduli less or equal to one if and only if $|\delta|\leq 2\sqrt{2}$
\end{proof}

Figures \ref{error4nonintrusive} and \ref{error4artificial} illustrate the domain of stability of nonintrusive FLAVOR  (based on symplectic Euler \eref{jhjhgjhgjg87g87A} and \eref{lsjfaijissdsdedlfskjlkd}) and artificial FLAVOR \eref{jhjhgjhgjg87sdedg87A} applied to the flow of \eref{erroranalysis_example2}, i.e. values of $\delta$ and $\tau/\epsilon$ ensuring stable numerical integrations. We observe that artificial FLAVOR has a much larger stability domain than nonintrusive FLAVOR. Specifically, for nonintrusive FLAVOR and large values of $\delta$, $\tau = o(\sqrt{\epsilon})$ is not enough and one needs $\tau = o(\epsilon)$ for a stable integration, whereas artificial FLAVOR only requires $\tau = \sqrt{2\epsilon}$, a minimum requirement for a stable symplectic Euler integration of the fast dynamics.

Notice that there is no resonance behavior in terms of stability; everything below the two curves is stable and everything outside is not stable (plots not shown).

\begin{figure}
\centering
\subfigure[Error of nonintrusive FLAVOR as a function of $\delta$ and $\tau/\sqrt{\epsilon}$. Notice that not all pairs of step lengths lead to stable integrations.]{
\includegraphics[width=0.45\textwidth]{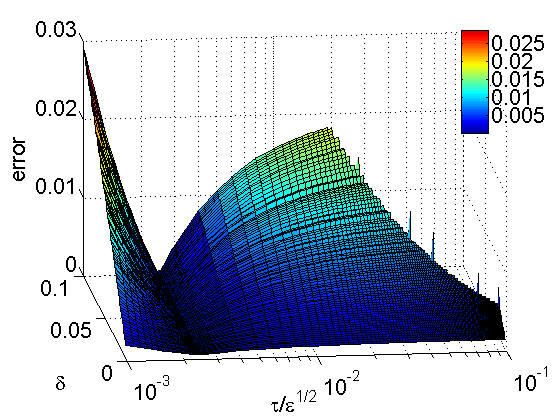}
\label{error1nonintrusive}
}
~
\subfigure[Error of artificial FLAVOR as a function of $\delta$ and $\tau/\sqrt{\epsilon}$]{
\includegraphics[width=0.45\textwidth]{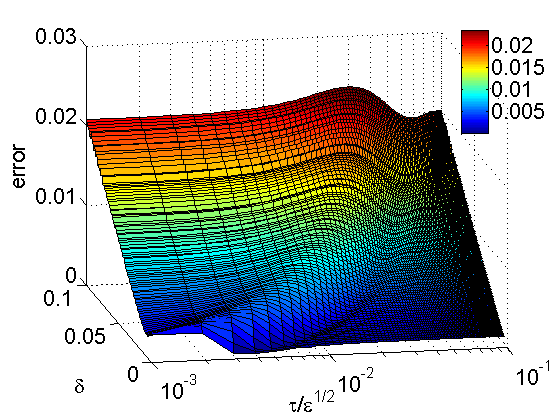}
\label{error1artificial}
}

\subfigure[Optimal $\tau/\sqrt{\epsilon}$ and error of nonintrusive FLAVOR as functions of $\delta$]{
\includegraphics[width=0.45\textwidth]{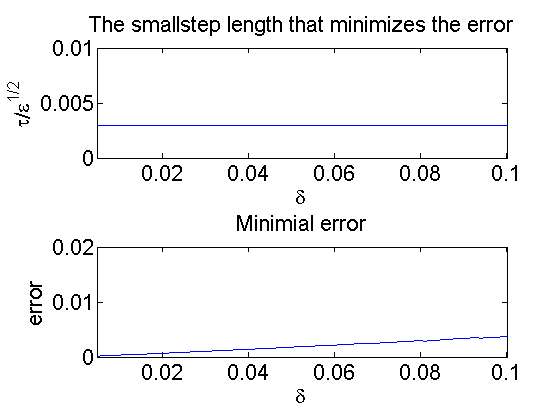}
\label{error2nonintrusive}
}
~
\subfigure[Optimal $\tau/\sqrt{\epsilon}$ and error of artificial FLAVOR as functions of $\delta$]{
\includegraphics[width=0.45\textwidth]{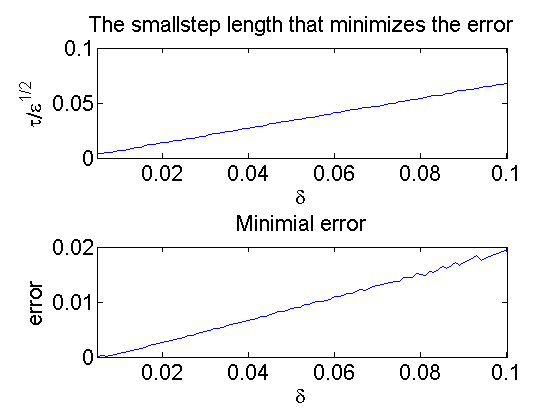}
\label{error2artificial}
}

\subfigure[Error dependence on $\tau/\sqrt{\epsilon}$ for a given $\delta$: nonintrusive FLAVOR]{
\includegraphics[width=0.45\textwidth]{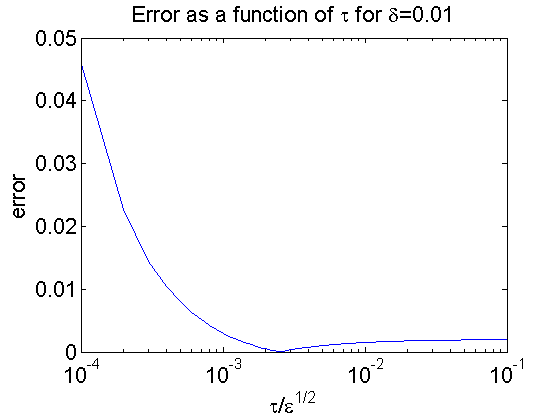}
\label{error3nonintrusive}
}
~
\subfigure[Error dependence on $\tau/\sqrt{\epsilon}$ for a given $\delta$: artificial FLAVOR]{
\includegraphics[width=0.45\textwidth]{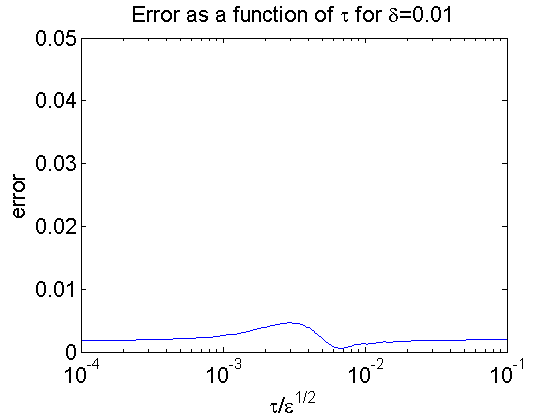}
\label{error3artificial}
}
\caption{\footnotesize Error analysis of \eqref{erroranalysis_example2}. Parameters are $\omega=\sqrt{\epsilon}=10^3$, $x(0)=0.8$ and $y(0)=x(0)+1.1/\omega$.}
\end{figure}

\subsection{Error analysis}\label{sjhgdkjhgsdjh2}
The flow of \eref{erroranalysis_example2} has been explicitly computed and compared with solutions obtained from nonintrusive FLAVOR based on symplectic Euler (\eref{jhjhgjhgjg87g87A} and \eref{lsjfaijissdsdedlfskjlkd}) and with artificial FLAVOR \eref{jhjhgjhgjg87sdedg87A}.

The total simulation time is $T=10$, and absolute errors on the slow variable have been computed with respect to the Euclidean norm of the difference in positions between analytical and numerical solutions. Stability is investigated using the same technique used in Subsection \ref{kjkjdhkdjhhkdhlala}.
Figures \ref{error1nonintrusive} and \ref{error1artificial} illustrate  errors as functions of mesostep $\delta$ and renormalized small step $\tau/\epsilon$. Observe that given $\delta$  errors are minimized at specific values of $\tau/\epsilon$ for both integrators, but the accuracy of nonintrusive FLAVOR is less sensitive to $\tau/\epsilon$. Figures \ref{error2nonintrusive} and \ref{error2artificial} plot the optimal value of $\tau/\epsilon$ as a function of $\delta$ and the associated to error.
Observe also that for nonintrusive FLAVOR the dependence of the optimal value of $\tau/\epsilon$   on $\delta$ is weak, whereas for artificial FLAVOR the optimal value of $\tau/\epsilon$ roughly scales linearly with $\delta$. Figure \ref{error3nonintrusive} and \ref{error3artificial} describe how error changes with smallstep $\tau$ for mesostep $\delta$ fixed. Figure \ref{error3nonintrusive} can be viewed in correspondence with the condition $\delta<<\tau/\epsilon$ required for accuracy. This requirement, however, is just a sufficient condition to obtain an error bound, as we can see in Figure \ref{error3artificial}. There the weak dependence of the error on $\tau/\epsilon$ for a fixed $\delta$ shows that one
does not have to choose the microstep with too much care or optimize the integrator with respect to its value, if artificial FLAVOR is used. As a matter of fact, all the numerical experiments illustrated in this paper (except for Figures \ref{error2nonintrusive} and \ref{error2artificial}) have been performed without any tuning of the value $\tau/\epsilon$. We have simply used the rule of thumb $\delta \sim \gamma \frac{\tau}{\epsilon}$  where $\gamma$ is a small parameter ($0.1$ for instance).

Therefore, it appears that the benefits of  artificial FLAVORS lie in their superior accuracy and stability.

Notice that there is no resonant value of $\delta$ or $\tau$.

\begin{figure} [ht]
\begin{tabular}{c}
{\resizebox{\imsizebig}{!}{\includegraphics{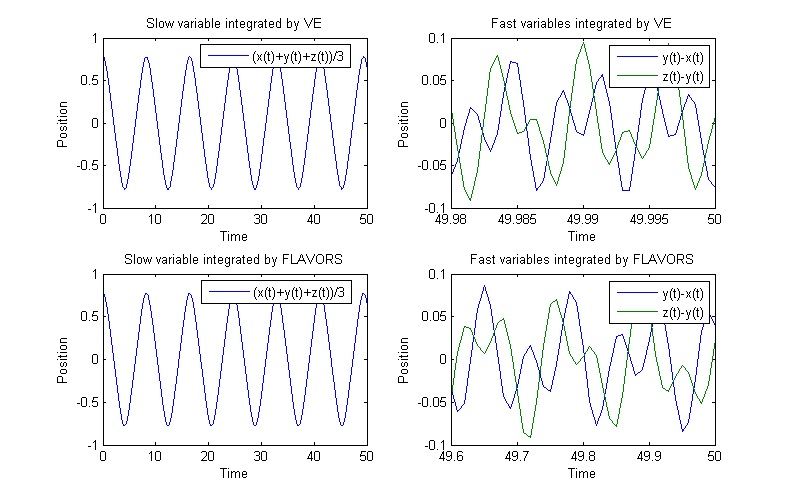}}}
\end{tabular}
\caption{\footnotesize Comparison between trajectories integrated by Variational Euler  and FLAVOR (defined by \eref{jhjhgjhgjg87g87A} and \eref{lsjfaijissdsdedlfskjlkd}).
FLAVOR uses mesostep $\delta=0.01$  and  microstep $\tau=0.0005$, and Symplectic Euler uses time step $\tau=0.0005$. Time axes in the right column are zoomed in (by different ratios) to illustrate the fact that fast variables are captured in the sense of measure.
FLAVOR accelerated the computation by roughly 20x ( $\delta=20\tau$).
In this experiment $\epsilon=10^{-6}$, $\omega_1=1.1$, $\omega_2=0.97$, $x(0)=0.8$, $y(0)=0.811$, $z(0)=0.721$, $p_x(0)=0$, $p_y(0)=0$ and $p_z(0)=0$. Simulation time $T=50$.
}
\label{example2result}
\end{figure}

\begin{figure} [h]
\centering
\subfigure[Asymptotically linear error dependence on $\delta=1/M$]{
\includegraphics[width=0.29\textwidth]{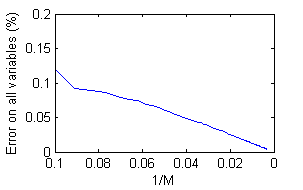}
\label{example2error_M}
}
\quad
\subfigure[Asymptotically linear error dependence on total simulation time $T$]{
\includegraphics[width=0.29\textwidth]{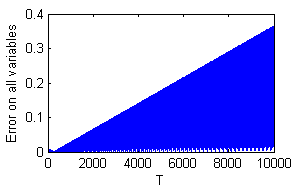}
\label{example2error_T}
}
\quad
\subfigure[Asymptotically independent of the scaling factor $\omega$]{
\includegraphics[width=0.29\textwidth]{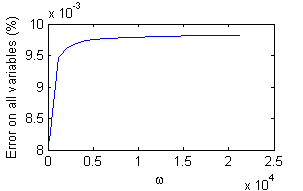}
\label{example2error_w}
}
\caption{\footnotesize Error dependence on parameters in a FLAVOR simulation of \eqref{erroranalysis_example}}
\end{figure}

\subsection{Numerical error analysis for nonlinear systems}\label{sjhgdkjhgsdjh}
In this subsection, we will consider the nonlinear Hamiltonian system
\begin{equation}
    H(x,y,z,p_x,p_y,p_z)=\frac{1}{2}p_x^2+\frac{1}{2}p_y^2+\frac{1}{2}p_z^2+x^4+\epsilon^{-1}\frac{\omega_1}{2}(y-x)^2
    +\epsilon^{-1}\frac{\omega_2}{2}(z-y)^2
    \label{erroranalysis_example}
\end{equation}
Thus, the potential is $U=\frac{\omega_1}{2}(y-x)^2+\frac{\omega_2}{2}(z-y)^2$ and $V=x^4$. Here $\frac{x+y+z}{3}$ acts as a slow degree of freedom and $y-x$ and $z-y$ act as fast degrees of freedom.

Figure \ref{example2result} illustrates $t \mapsto \frac{x(t)+y(t)+z(t)}{3}$ (slow variable, convergent strongly) and $t \mapsto (y(t)-x(t),z(t)-y(t))$ (fast variables, convergent in measure) computed with symplectic Euler and with the induced symplectic FLAVOR \eref{jhjhgjhgjg87g87A}).
Define $q:=(x,y,z)$.
To illustrate the $F$-convergence property of FLAVOR, we fix $H=1$, vary the mesostep $\delta=H/M$ by changing $M$ and show the Euclidean norm error of the difference between  $\frac{1}{M}\sum_{i=0}^{M-1} q(T-ih/M)$ computed with FLAVOR and computed with symplectic Euler in Figure \ref{example2error_M}. Notice that without an averaging over time length $h$, the error will be no longer monotonically but oscillatorily decreasing as $\delta$ changes (plots not shown), because fast variables are captured only in the sense of measure. As shown in Figure \ref{example2error_M} the error scales linearly with $\frac{1}{M}$ for $M$ not too small, and therefore the global error is a linear function of the mesostep $\delta$ and the method is first order convergent. Figure \ref{example2error_T} shows that the error in general grows linearly with the total simulation time, and this linear growth of the error has been observed  for a simulation time larger than $\omega$ ($\epsilon^{-1/2}$). Figure \ref{example2error_w} shows that the error does not depend on $\omega$ ( $\epsilon^{-1/2}$) for a fixed $\delta$, as long as  $\epsilon$ is not too large (i.e. $\omega$ not too small); the error is instead controlled by $M$. This is  not caused by reaching the limit of machine accuracy, it is a characteristic of the method: the plateau for large $\omega$ corresponds to the complete scale separation regime of FLAVOR as a multiscale method.

Notice that there is no resonant value of $\delta$ in the sense of convergence.

The fact that the error scales linearly with total simulation time is a much stronger (numerical) result than our (theoretical) error analysis for FLAVORS (in which the error is bounded by a term growing exponentially with the total simulation time). We conjecture that the linear growth of the error is a consequence of the fact that FLAVOR is symplectic and is only true for a subclass of systems, possibly integrable systems. A rigorous analysis of the effects of the structure preservation of FLAVORS on long term behavior remains to be done.

\begin{figure} [h]
\begin{tabular}{c}
{\resizebox{\imsizebig}{!}{\includegraphics{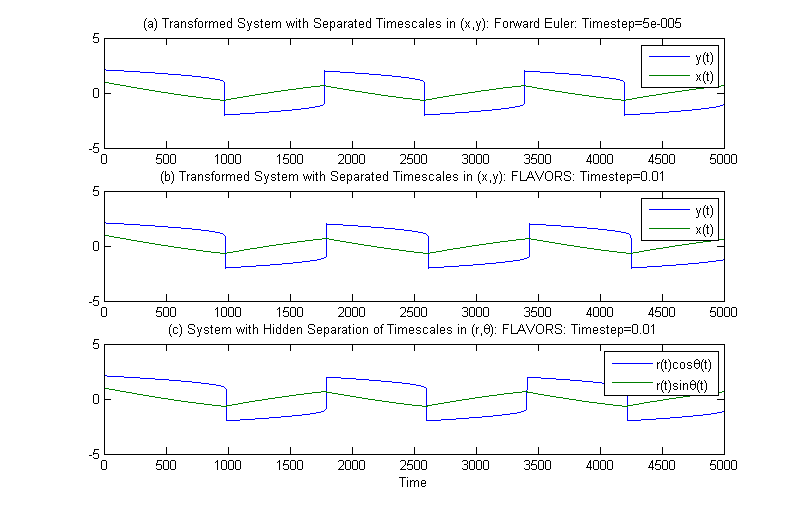}}}
\end{tabular}
\caption{\footnotesize Over a timespan of $5/\epsilon$ (a) Direct Forward Euler simulation of \eref{eq2} with time steps resolving the fast time scale (b) (nonintrusive \eref{ksjhjhhjdskdjjhhuhuuwue}) FLAVOR simulation of \eref{eq2} (c) Polar to cartesian image of the (nonintrusive \eref{ksjhjhhjdskdjjhhuhuuwue}) FLAVOR simulation of \eref{eq1} with hidden slow and fast variables. Forward Euler uses time step $h=0.05 \epsilon=0.00005$. The two FLAVORS simulations use $\delta=0.01$ and $\tau=0.00005$. Parameters are $\frac{1}{\epsilon}=1000$, $x(0)=1$, $y(0)=1$}
\label{vanderPol}
\end{figure}

\section{Numerical experiments}
\subsection{Hidden  Van der Pol oscillator (ODE)}
Consider the following system ODEs
\begin{equation}\label{eq1}
    \begin{cases}
        \dot{r}=\frac{1}{\epsilon} (r\cos\theta+r\sin\theta-\frac{1}{3}r^3\cos^3\theta)\cos\theta - \epsilon\, r\cos\theta\sin\theta \\
        \dot{\theta}=-\epsilon \,cos^2\theta-\frac{1}{\epsilon}(\cos\theta+\sin\theta-\frac{1}{3}r^2\cos^3\theta)\sin\theta
    \end{cases}
\end{equation}
where $\epsilon \ll1$.
Taking the transformation from polar coordinates to Cartesian coordinates by $[x,y]=[r\sin\theta,r\cos\theta]$ as the local diffeomorphism, we obtained the hidden system:
\begin{equation}
    \begin{cases}\label{eq2}
        \dot{x}=-\epsilon y \\
        \dot{y}=\frac{1}{\epsilon}(x+y-\frac{1}{3}y^3)
    \end{cases}
\end{equation}
Taking the second time derivative of $y$,  the system can also be written as the $2^{nd}$-order ODE:
\begin{equation}
    \ddot{y}+y=\frac{1}{\epsilon} (1-y^2)\dot{y} .
\end{equation}
The latter is a classical  Van der Pol oscillator \cite{Verhulst:96}. Nonintrusive FLAVOR as defined by \eref{ksjhjhhjdskdjjhhuhuuwue} can be directly applied to \eref{eq1} (with hidden slow and fast processes) by turning on and off the stiff parameter $\frac{1}{\epsilon}$. More precisely, defining  $\Phi^{\epsilon,\alpha}(r,\theta)$ by
\begin{equation}\label{eqff1}
\Phi^{\alpha, \epsilon}_h(r,\theta):=\begin{pmatrix}r\\\theta \end{pmatrix}+ \alpha h \begin{pmatrix}(r\cos\theta+r\sin\theta-\frac{1}{3}r^3\cos^3\theta)\cos\theta\\-  (\cos\theta+\sin\theta-\frac{1}{3}r^2\cos^3\theta)\sin\theta\end{pmatrix}-\epsilon h \begin{pmatrix}r\cos\theta\sin\theta\\  cos^2\theta\end{pmatrix}
\end{equation}
FLAVOR is defined by \eref{ksjhjhhjdskdjjhhuhuuwue} with $\bar{u}:=(\bar{r},\bar{\theta})$, i.e.,
\begin{equation}\label{ksjhjhhjdsehhuhuuwue}
(\bar{r}_t,\bar{\theta}_t)=\big(\Phi^{0,\epsilon}_{\delta-\tau}\circ \Phi^{\frac{1}{\epsilon},\epsilon}_{\tau}\big)^k(r_0,\theta_0) \quad \text{for}\quad k\delta \leq t <(k+1)\delta .
\end{equation}
We refer to Figure \ref{vanderPol} for a comparison of integrations by Forward Euler, used as a benchmark, and FLAVORS. FLAVORS gives trajectories close to Forward Euler  and correctly captures the $\mathcal{O}(\frac{1}{\epsilon})$ period \cite{Verhulst:96} of the relaxation oscillation. Moreover, a 200x acceleration is achieved using FLAVOR.

\begin{figure} [h]
\begin{tabular}{c}
{\resizebox{\imsizebig}{!}{\includegraphics{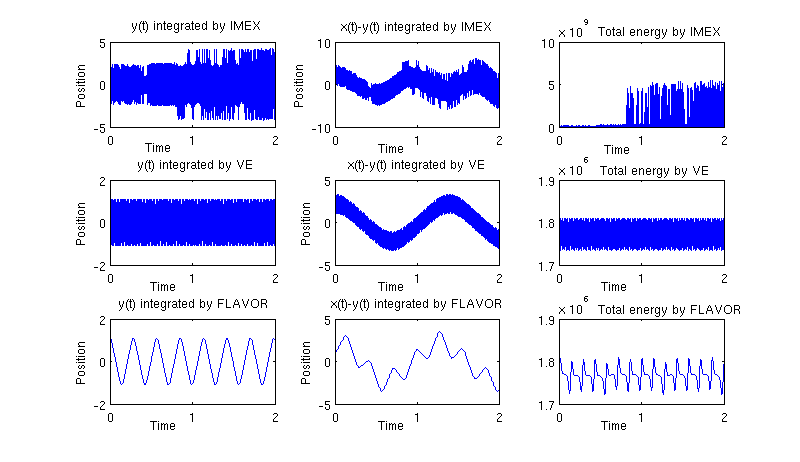}}}
\end{tabular}
\caption{\footnotesize
In this experiment, $\epsilon=10^{-6}$, $y(0)=1.1$, $x(0)=2.2$, $p_y(0)=0$ and $p_x(0)=0$. Simulation time $T=2$. FLAVOR (defined by \eref{jhjhgjhgjg87g87A} and \eref{lsjfaijissdsdedlfskjlkd}) uses mesostep $\delta=10^{-3}$  and microstep $\tau=10^{-5}$, Variational Euler uses small time step $\tau=10^{-5}$, and IMEX uses mesostep $\delta=10^{-3}$.
Since the fast potential  is nonlinear, IMEX is an implicit method and nonlinear equations have to be solved at every step, and  IMEX turns out to be slower than Variational Euler. FLAVOR is strongly accurate with respect to slow variables and accurate in the sense of measures with respect to fast variables. Comparing to Symplectic Euler, FLAVOR accelerated the computation by roughly 100x.
}
\label{example1result}
\end{figure}

\subsection{Hamiltonian system with nonlinear stiff and soft potentials}\label{kjkjdhkdjhhkdh}
In this subsection, we will apply the Symplectic Euler FLAVOR defined by \eref{jhjhgjhgjg87g87A} and \eref{lsjfaijissdsdedlfskjlkd}
to the mechanical system whose Hamiltonian is
\begin{equation}
    H(y,x,p_y,p_x):=\frac{1}{2}p_y^2+\frac{1}{2}p_x^2+\epsilon^{-1}y^6+(x-y)^4
\end{equation}
Here, stiff potential $\epsilon^{-1}U=\epsilon^{-1}y^6$ and soft potential $V=(x-y)^4$ are both nonlinear.

Figure \ref{example1result} illustrates $t\mapsto y(t)$ (dominated by a fast process), $t\mapsto x(t)-y(t)$ (a slow process modulated by a fast process), and $t\mapsto H(t)$ computed with: Symplectic Euler, the induced symplectic FLAVOR (\eref{jhjhgjhgjg87g87A} and \eref{lsjfaijissdsdedlfskjlkd}), and IMEX \cite{Stern:09}. Notice that $x-y$ is not a purely slow variable but contains some fast component, and therefore the FLAVOR integration of it contains a modulation of local oscillations, which could be interpreted as that fast component slowed down by FLAVOR. It's not easy to find a purely slow variable or a purely fast variable in the form of \eref{kfgfgdiuuiusedejhd} for this example, but the integrated trajectory for such a slow variable will not contain these slowed-down local oscillations.

\begin{figure} [h]
\begin{tabular}{c}
{\resizebox{\imsizebig}{!}{\includegraphics{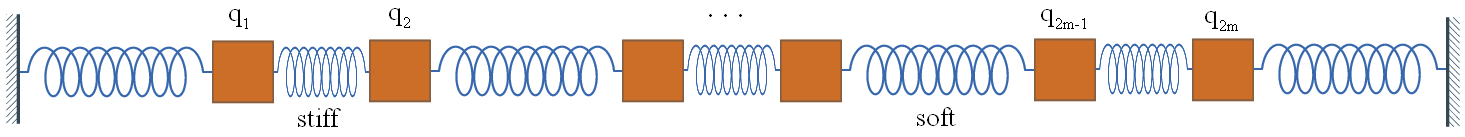}}}
\end{tabular}
\caption{\footnotesize Fermi-Pasta-Ulam problem \cite{FPU:55} -- 1D chain of alternatively connected harmonic stiff and non-harmonic soft springs}
\label{FPUfigure}
\end{figure}

\begin{figure}
\centering
\subfigure[By Variational Euler with  small time step $\tau'=5\times 10^{-5}=0.05/\omega$. 38 periods in Subplot2 with zoomed-in time axis ($\sim$380 in total over the whole simulation span).]{
\includegraphics[width=0.4\textwidth]{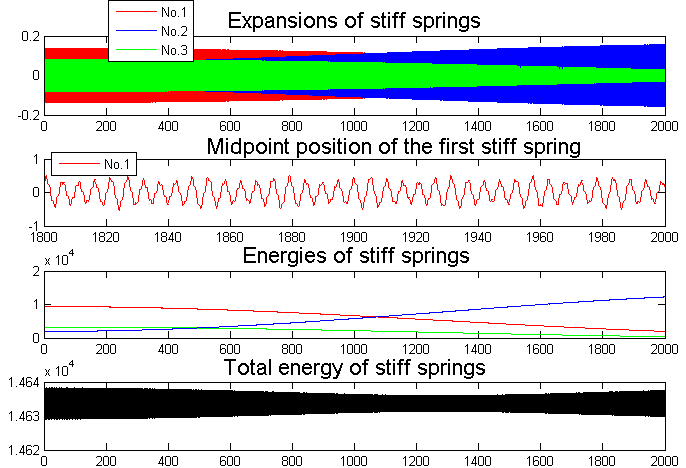}
\label{FPU_VE}
}
\subfigure[By artificial FLAVOR \eref{jhjhgjhgjg87sdedg87A} with  mesostep  $\delta=0.002$ and  microstep $\tau=10^{-4}=0.1/\omega$. 38 periods in Subplot2 with zoomed-in time axis ($\sim$380 in total over the whole simulation span).]{
\includegraphics[width=0.4\textwidth]{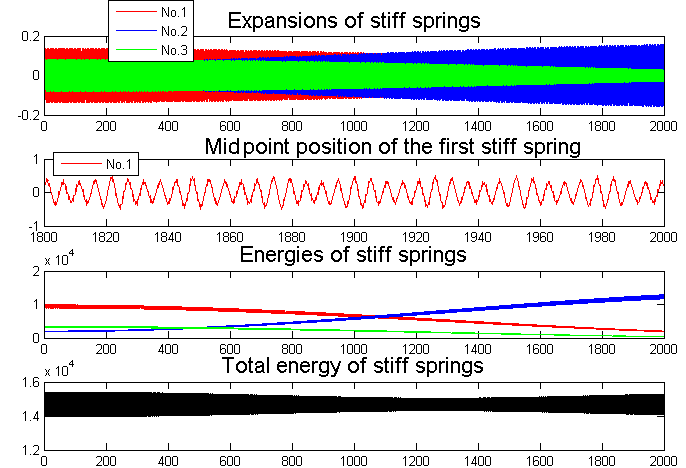}
\label{FPU_FLAVORS}
}
\caption{\footnotesize Simulations of the FPU problem over $T=2\omega$. Subplot2 of both figures have zoomed-in time axes so that whether phase lag or any other distortion of trajectory exists could be closely investigated. In this experiment $m=3$, $\omega=10^3$, $x(0)=[0.4642,-0.4202,0.0344,0.1371,0.0626,0.0810]$ is randomly chosen and $y(0)=[0,0,0,0,0,0]$.}
\label{FPU_short}
\end{figure}

\subsection{Fermi-Pasta-Ulam problem}\label{FPUsec}
In this subsection, we will consider the Fermi-Pasta-Ulam (FPU) Problem \cite{FPU:55} illustrated by Figure \ref{FPUfigure} and associated with the Hamiltonian
\begin{equation}
    H(q,p):=\frac{1}{2} \sum_{i=1}^m (p_{2i-1}^2+p_{2i}^2)+\frac{\omega^2}{4} \sum_{i=1}^m (q_{2i}-q_{2i-1})^2+ \sum_{i=0}^m (q_{2i+1}-q_{2i})^4.
\end{equation}
The FPU problem is a well known benchmark problem \cite{FPUcomp07, Hairer:04} for multiscale integrators  because it exhibits different behaviors over widely separated timescales. The stiff springs nearly behave  like harmonic oscillators with period $\sim \mathcal{O}(\omega^{-1})$. Then, the centers of masses linked by stiff springs (i.e., the midpoints of stiff springs) change over a timescale $\mathcal{O}(1)$. The third timescale, $\mathcal{O}(\omega)$, is associated with the rate of energy exchange between stiff springs. The fourth timescale, $\mathcal{O}(\omega^{2})$, corresponds to the synchronization of energy exchange between stiff springs. On the other hand, the total energy of the stiff springs behaves almost like a constant. This wide separation of timescales can be seen in Figure \ref{FPU_short}, \ref{FPU_long}, and \ref{FPU_quadratic_exact_long}, where four subplots address different scales: Subplot1 shows the fast variables $(q_{2i}-q_{2i-1})/\sqrt{2}$; Subplot2 shows one of the slow variables $(q_{2}+q_{1})/\sqrt{2}$; Subplot3 shows the energy transfer pattern among stiff springs, which is even slower; Subplot4 shows the near-constant total energy of three stiff springs. All four subplots are time-series. A comprehensive survey on FPU problem, including discussions on timescales and numerical recipes, can be found in \cite{Hairer:04}.

Figures \ref{FPU_VE} and \ref{FPU_FLAVORS} compare symplectic Euler (with time steps fine enough to resolve FPU over the involved long time scale) and with the  artificial FLAVOR \eref{jhjhgjhgjg87sdedg87A}. On a timescale $\mathcal{O}(\omega)$ ($\omega \gg 1$), FLAVOR  captured slow variable's periodic behavior with the correct period and phase, as well as the slower process of energy transfer.  At the same time, FLAVOR accelerated the computation by roughly 40x (since $\delta=40\tau'$).

It is not worrisome that artificial FLAVOR produces stiff spring energy trajectories with rapid local oscillations, which exhibit in both thicker individual energy curves and total energy with larger variance. In fact, these local oscillations do not seem to affect the global transfer pattern nor its period and are caused by the numerical error asociated with microstep $\tau$. This can be inferred by using the artificial FLAVOR introduced in Subsection \ref{gjsdgjhdhjhwg} with  $\theta^\epsilon_\tau$ corresponding to the exact flow of $H^{fast}$ (rather than its Variational Euler approximation: this specific artificial Euler resembles the Impulse Method, but the Impulse Method will yield unbounded trajectories if one runs even longer time simulations, whereas FLAVORS do not seem to have an error growing exponentially with the total simulation time). As illustrated in Figure \ref{FPU_FLAVORS_exactfast}, exact flow helps to obtain thin energy curves of stiff springs with no rapid local oscillations as well as a total energy with a variance smaller than that given by fine Variational Euler (Figure \ref{FPU_VE}), with the transfer pattern similar to Figure \ref{FPU_FLAVORS}.

Now, we reach further to $\mathcal{O}(\omega^2)$ total integration time to investigate different integrators' performances in capturing the synchronized energy exchange pattern (Figure \ref{FPU_long}).

\begin{figure}[t]
\centering
\subfigure[By Velocity Verlet with small time step $h=10^{-5}$.]{
\includegraphics[ width=0.47\textwidth]{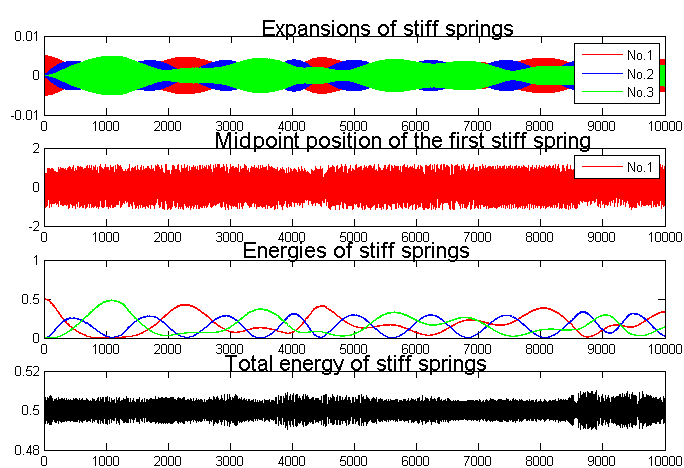}
\label{FPU_VV_long}
}
~
\subfigure[By artificial FLAVOR \eref{jhjhgjhgjg87sdedg87A} with mesostep $\delta=0.002$ and microstep $\tau=0.0005=0.1/\omega$.]{
\includegraphics[width=0.47\textwidth]{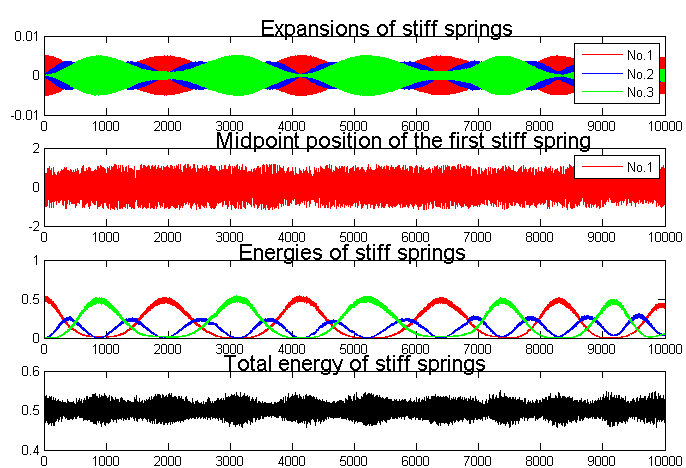}
\label{FPU_FLAVORS_long}
}

\subfigure[By IMEX with mesostep $\delta=0.002$.]{
\includegraphics[width=0.47\textwidth]{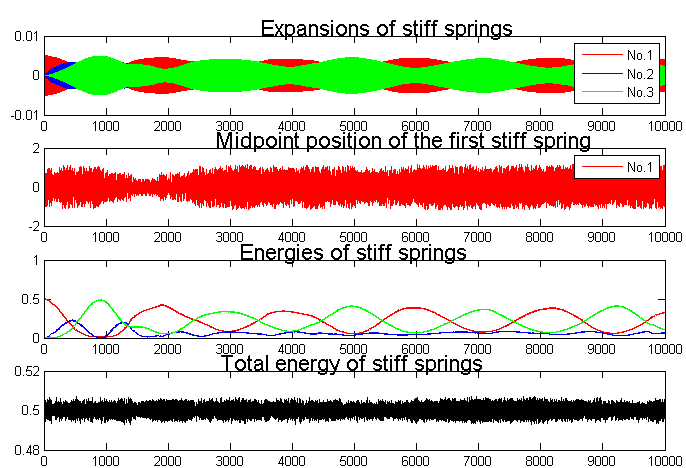}
\label{FPU_IMEX_long}
}
~
\subfigure[By Impulse Method with mesostep $\delta=0.002$.]{
\includegraphics[width=0.47\textwidth]{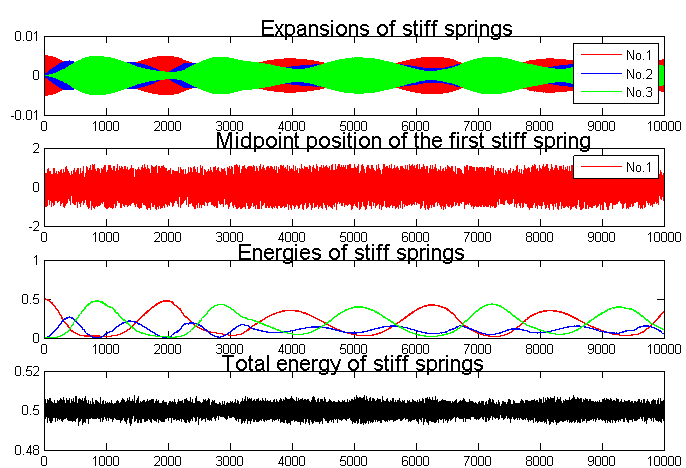}
\label{FPU_Impulse_long}
}
\caption{\footnotesize Simulations of FPU problem over $T=\frac{1}{4}\omega^2$. Initial conditions are $x(0)=[1,0,0,1/\omega,0,0]$ and $y(0)=[0,0,0,0,0,0]$ so that energy starts concentrated on the leftmost soft and stiff springs, propagates to the right, bounces back, and oscillates among springs. We chose a smaller $\omega=200$ because with a larger $\omega$ it would take weeks to run Velocity Verlet  on a laptop.}
\label{FPU_long}
\end{figure}

There is a significant difference among stiff spring energy transfer patterns produced by Velocity Verlet, FLAVOR, IMEX and the Impulse Method.
Here, there is no analytic solution or provably accurate method for comparison. FLAVOR is the only method that shows periodic behavior on the long time scale and convergence tests show that FLAVOR's trajectories remain stable under small variations of step sizes.
Notice that mathematically it can be shown that the dynamical system admits periodic orbits (for example, by Poincar\'{e}-Bendixson theorem). Furthermore, it is physically intuitive that the three stiff springs should alternatively obtain their maximal and minimal energies, and these maximal energies should be of fixed values.  In addition, if we change the slow potential to be quadratic the system is still very similar to non-harmonic FPU but now analytically solvable. There, the energy exchanging pattern (Figure \ref{FPU_quadratic_exact_long}) resembles FLAVORS' result of the non-harmonic system but not the other integrators'. Notice that if run on the modified quadratic FPU problem however, FLAVORS, Velocity-Verlet, IMEX and the Impulse Method all obtain perfect results (plots omitted). These are numerical evidences supporting FLAVOR on the $\mathcal{O}(\omega^2)$ timescale.

\begin{figure}[ht]
\begin{minipage}{0.5\linewidth}
\centering
\includegraphics[width=1.0\textwidth]{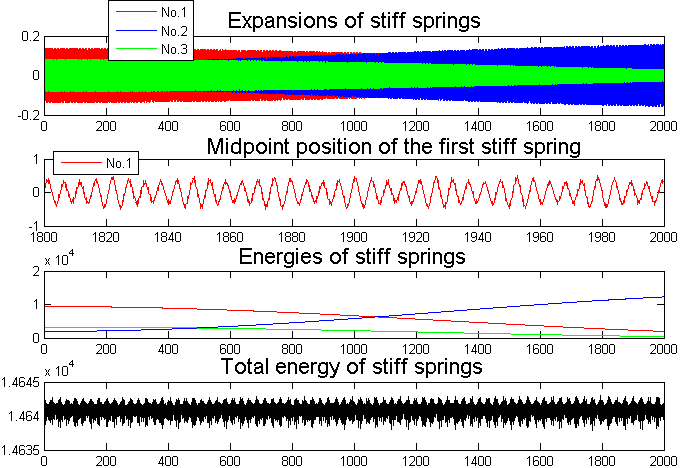}
\caption{\footnotesize By artificial FLAVORS (Subsection \ref{gjsdgjhdhjhwg}) based on exact fast flow with mesostep $\delta=0.002$ and  microstep $\tau=10^{-4}$. Less oscillatory stiff spring energies. 38 periods in Subplot2 with zoomed-in time axis ($\sim$380 in total over the whole simulation span). }
\label{FPU_FLAVORS_exactfast}
\end{minipage}
\hspace{0.2cm}
\begin{minipage}{0.5\linewidth}
\centering
\includegraphics[width=1.0\textwidth]{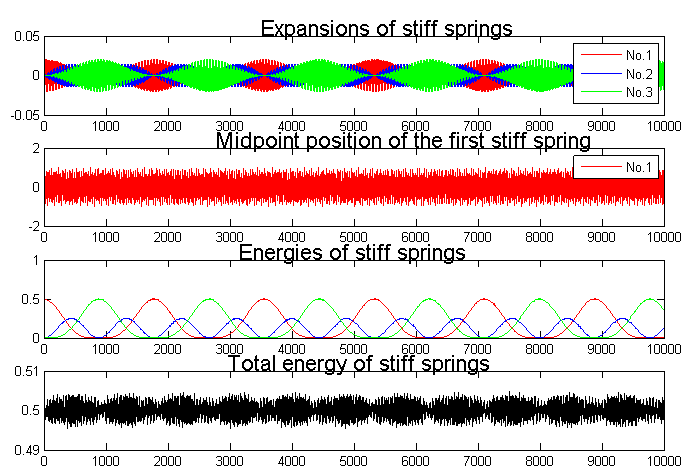}
\caption{\footnotesize Harmonic FPU, $T=50\omega$, exact solution}
\label{FPU_quadratic_exact_long}
\end{minipage}
\end{figure}

It is worth discussing why Velocity-Verlet with a time step much smaller than the characteristic length of the fast scale ($\mathcal{O}(1/\omega)$)is still not satisfactory. Being a second order method, it has an error bound of $\mathcal{O}(e^T h^2)$. On the other hand, backward error analysis guarantees that the energy of the integrated trajectory oscillates around the true conserved energy, hence eliminating the possibility of exponential growth of the numerical solution. Nevertheless, at this moment there is no result known to the authors to link these two analytical results to guarantee long term accuracy on the stiff springs' energies. The energy exchange among stiff springs is in fact an delicate phenomenon, and a slight distortion in stiff spring lengths could easily disrupt its period or even its periodicity.

These numerical observations seem to indicate that symplectic FLAVORS may have special long time properties. Specifically, although we could not quantify the error here because there is no benchmark to compare to when the total simulation time is $\mathcal{O}(\omega^2)$, the long term behavior seems to indicate an error growing much slower than exponentially (please refer to Remark \ref{exponentialError} for a discussion on exponential error bounds and Figure \ref{example2error_T} for another example of conjectured linear error growth). A rigorous investigation on FLAVORS' long time behavior remains to be done.
\begin{figure}
\centering
\subfigure[FLAVOR]{
\includegraphics[width=0.496\textwidth]{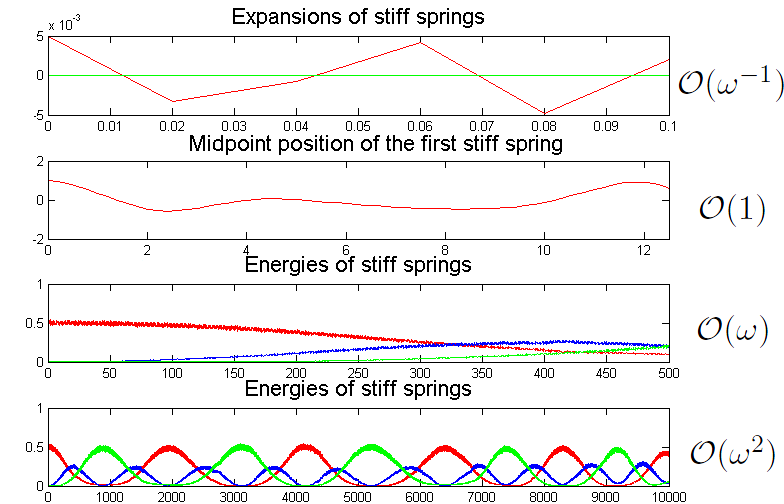}
}
~
\subfigure[Velocity Verlet]{
\includegraphics[width=0.434\textwidth]{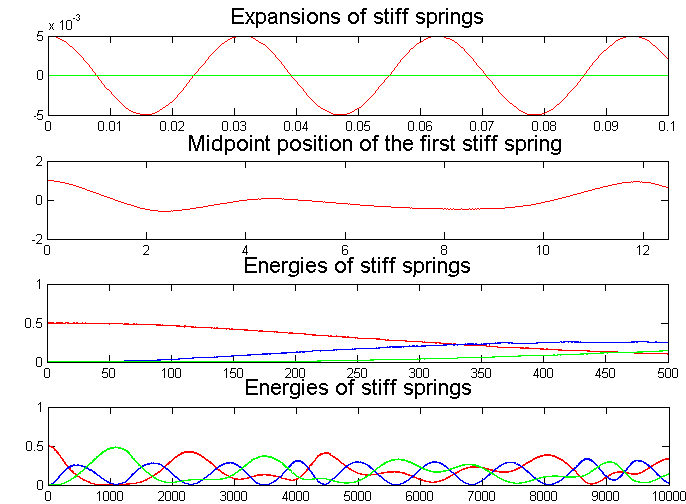}
}
\caption{\footnotesize Quantities of interest in integrations of FPU over different timescales. FLAVOR \eref{jhjhgjhgjg87sdedg87A}  captures the fastest timescale in the sense of measure, while Velocity Verlet cannot accurately capture the slowest ($\mathcal{O}(\omega^2)$) timescale despite of the small time step it uses. Here FLAVOR is 200 times faster than Velocity Verlet. All parameters are the same as in Figure \ref{FPU_VV_long} and \ref{FPU_FLAVORS_long}, e.g. $\omega=200$, $\delta=0.002$, $\tau=0.0005$ and $h=10^{-5}$.}
\label{FPU_timescales}
\end{figure}

Figure \ref{FPU_timescales} summarizes FLAVOR's performance on various timescales in a comparison to Velocity Verlet.

Notice that there are many sophisticated methods designed for integrating the FPU problem (see \cite{Hairer:04} for a review), as well as general multiscale methods that can be applied to the FPU problem. HMM as one state-of-art method in the latter category, together with  identification of slow variables \cite{Ariel:08} captured the energy transfer between stiff springs over a time span of the order of $\omega$. Simulations shown here are over a time span of the order of $\mathcal{O}(\omega^2)$.

\subsubsection{On resonances}
Multiscale in time integrators are usually plagued by two kinds of resonances.

The first kind, called Takens resonance \cite{Takens:80}, is related to the fact
there are no closed equations on slow variables \cite{MR1436164}.
FLAVORS  avoid Takens resonance because, thanks to $F$-convergence, the information on the local invariant measure of fast variables is not lost. Observe that
the FPU problem has Takens resonance (the eigenfrequencies of the strong potential are identical).  Nevertheless, FLAVORS still capture the solution trajectories given any large value of $\omega$  with mesostep $\delta \gg 1/\omega$ independent of $\omega$.

The second kind \cite{CalvoSena09} is related to instabilities created by interactions between parameters $\epsilon$, $\tau$ and $\delta$. For instance, if $\epsilon^{-1}=\omega^2$ resonance might happen at $\omega\delta$ or $\omega\tau$ equal to multiples of $\pi/2$. The analysis provided in Section \ref{kjkjdhkdjhhkdhlalala} shows that such unstable interaction does not occur, either in the sense of stability or in terms of peaks of error function. This can be intuitively understood by observing that FLAVORS never approximate $\cos(\delta \omega)$, while on the other hand, it does approximate $\cos(\tau \omega)$ whose resonance frequency $\tau=2\pi/\omega$ is ruled off by the requirement of $\tau \ll \epsilon$ for nonintrusive FLAVOR and $\tau \ll \sqrt{\epsilon}$ for artificial FLAVOR.

\subsection{Nonlinear 2D primitive molecular dynamics}\label{Prisec}

Now consider a two-dimensional, two degrees of freedom example in which a point mass is linked through a spring to a massless fixed hinge at the origin. While the spring as well as the point mass are allowed to rotate around the hinge (the spring remains straight), the more the spring-mass tilts away from its equilibrium angle the more restorative force it will experience. This example is a simplified version of prevailing molecular dynamics models, in which bond lengths and angles between neighboring bonds are both spring-like; other potential energy terms are ignored.

Denote by $x$ and $y$ the Euclidean coordinates of the mass, and $p_x$, $p_y$ corresponding momentums. Also introduce polar coordinates $(r,\theta)$, with $x=r\cos\theta$ and $y=r\sin\theta$. Then the Hamiltonian reads

\begin{align}\label{sdiudyiuyyur}
    H &=\frac{1}{2}p_x^2+\frac{1}{2}p_y^2+\frac{1}{2}\omega^2(r-r_0)^2+(\cos\theta)^2 \nonumber\\
      &=\frac{1}{2}p_x^2+\frac{1}{2}p_y^2+\frac{1}{2}\omega^2(\sqrt{x^2+y^2}-r_0)^2+\frac{x^2}{x^2+y^2}
\end{align}
where $r_0$ is equilibrium bond length parameter and $\omega$ is large number denoting bond oscillation frequency.

\begin{Remark} This seemingly trivial example is not easy to integrate.
\begin{enumerate}
\item  If the system is viewed in Euclidean coordinates $(x,y,p_x,p_y)$ it is completely nonlinear with a nonpolynomial potential, and hence the Impulse Method or its variations \cite{Grubmuller:91,Tuckerman:92,Skeel:99,Sanz-Serna:08}, or IMEX \cite{Stern:09}, or the homogenization method introduced in \cite{LeBris:07} cannot be applied using a mesostep.

\item If the Hamiltonian is rewritten  in generalized coordinates $(r,\theta,p_r,p_\theta)$,  $H=\frac{1}{2}p_r^2+\frac{1}{2}\frac{p_\theta^2}{r^2}+\frac{1}{2}\omega^2(r-r_0)^2+\frac{1}{2}\cos(\theta)^2$, a fast quadratic potential can be identified.

    However, the mass matrix $\begin{bmatrix} 1 & 0 \\ 0 & r^2 \end{bmatrix}$ is not constant, but rapidly oscillating, and hence methods that work for quasi-quadratic fast potentials (i.e. ``harmonic oscillator'' with a slowly changing frequency) (\cite{LeBris:07} for example) cannot be  applied.

\end{enumerate}
\end{Remark}

\begin{figure} [h]
\begin{tabular}{c}
{\resizebox{\imsizebig}{!}{\includegraphics{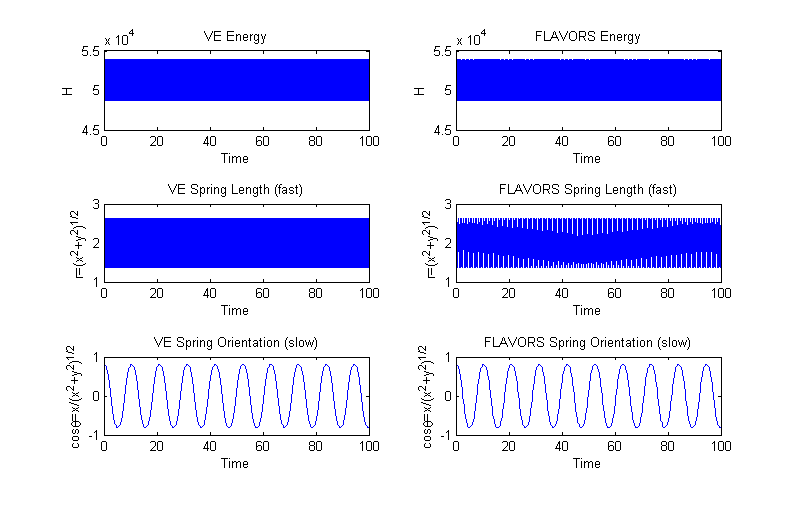}}}
\end{tabular}
\caption{\footnotesize Simulation of \eqref{sdiudyiuyyur}. Symplectic Euler uses  small time step $\tau=0.0002$ and the induced symplectic FLAVOR (\eref{jhjhgjhgjg87g87A} and \eref{lsjfaijissdsdedlfskjlkd}) uses  mesostep $\delta=0.01$ and  microstep $\tau=0.0002$. In this simulation $\omega=500$, $x(0)=1.1$, $y(0)=0.8$, $p_x(0)=0$, $p_y(0)=0$ and simulation time $T=100$.}
\label{example4result}
\end{figure}

Figure \ref{example4result}  compares symplectic Euler  with the induced symplectic FLAVOR (\eref{jhjhgjhgjg87g87A} and \eref{lsjfaijissdsdedlfskjlkd}) applied to \eref{sdiudyiuyyur} in Euclidean coordinates.

FLAVOR  reproduced the slow $\theta$ trajectory  while accelerating the simulation time by roughly 50x (since $\delta=50\tau$). It can also be seen from both  energy fluctuations and the trajectory of the fast variable that the fast process' amplitude is well captured although its period has been lengthened.

\begin{figure} [h]
\centering
\includegraphics[width=0.5\textwidth]{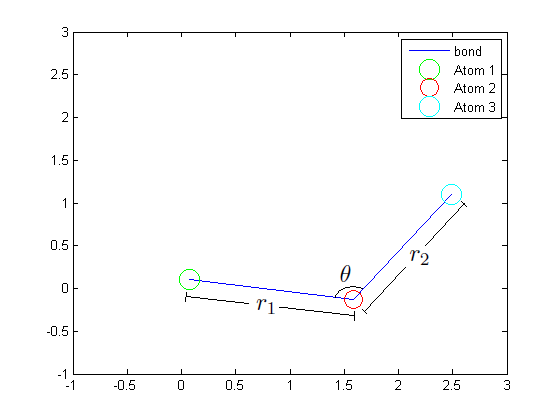}
\caption{\footnotesize One exemplary configuration of a propane molecule}
\label{propane_model}
\end{figure}

\subsection{Nonlinear 2D molecular clipper}\label{expro}

We now consider a united-atom representation of a three atom polymer with two bonds (e.g. propane or water molecule). This is a simplified version of several prevailing molecular dynamics force fields (for example, CHARMM \cite{CHARMM}, AMBER \cite{AMBER}, or a simpler example of butane \cite{butane1,butane2}). Using conservation of momentum,
we fix the coordinate system in the 2D plane defined by the three atoms.
 Introduce both Cartesian coordinates $(x_1,y_1,x_2,y_2,x_3,y_3)$, as well as generalized coordinates $r_1=\sqrt{(x_2-x_1)^2+(y_2-y_1)^2}$ and $r_2=\sqrt{(x_3-x_2)^2+(y_3-y_2)^2}$ for bond lengths and $\theta$ for the angle between the two bonds (Figure \ref{propane_model}). The kinetic energy is
\begin{equation}
    K.E.=\frac{1}{2}m_1 (\dot{x}_1^2+\dot{y}_1^2)+\frac{1}{2}m_2 (\dot{x}_2^2+\dot{y}_2^2)+\frac{1}{2}m_3 (\dot{x}_3^2+\dot{y}_3^2)
\end{equation}
where $m_1$,$m_2$,and $m_3$ denote the masses of the atoms.

The potential energy consists of a bond term and a bond angle term, both of which are of harmonic oscillator type:
\begin{align}
    P.E. &= V_{bond}+V_{angle} \\
    V_{bond} &= \frac{1}{2}K_r [(r_1-r_0)^2+(r_2-r_0)^2] \\
    V_{angle} &= \frac{1}{2} K_\theta (\cos(\theta)-\cos(\theta_0))^2
\end{align}

Notice that the system is in fact fully nonlinear: if written in generalized coordinates, the kinetic energy will correspond to a nonlinear and position dependent mass matrix, whereas in Cartesian coordinates, both terms in the potential energy are non-polynomial functions in positions.

In the case of propane, $m_1=15\mu,m_2=14\mu,m_3=15\mu$ where $\mu=1.67\cdot 10^{-27}kg$, $r_0=1.53 \AA$, $K_r=83.7kcal/(mol \AA^2)$, $\theta_0=109.5^{\circ}$ and $K_\theta=43.1kcal/mol$ \cite{butane1}.

The propane system is characterized by a separation of timescales to some extent: bond stretching and bond-angle bending are characterized by $10^{14}$ and $10^{13}$ Hz vibrational frequencies respectively \cite{ZhSc:93}. For investigation on FLAVORS, we use unitless parameters and exaggerate the timescale separation by setting $K_r$ to be $8370$ and $K_\theta$ to be $4.31$. We also let $\mu=1$ without loss of generality for arithmetic considerations.

In this system, the bond potential is the fast potential and the bond-angle potential is the slow one. It is well known that using a large time step at the timescale corresponding to the bond-angle potential by freezing bond lengths produces biased results, and many physics based methods have been proposed to remedy this difficulty (for example by the approach of Fixman \cite{Fixman:74}; also see a review in \cite{ZhSc:93}). On the other hand, few multiscale methods work for this fully nonlinear system.

\begin{figure} [h]
\begin{tabular}{c}
{\resizebox{\imsizebig}{!}{\includegraphics{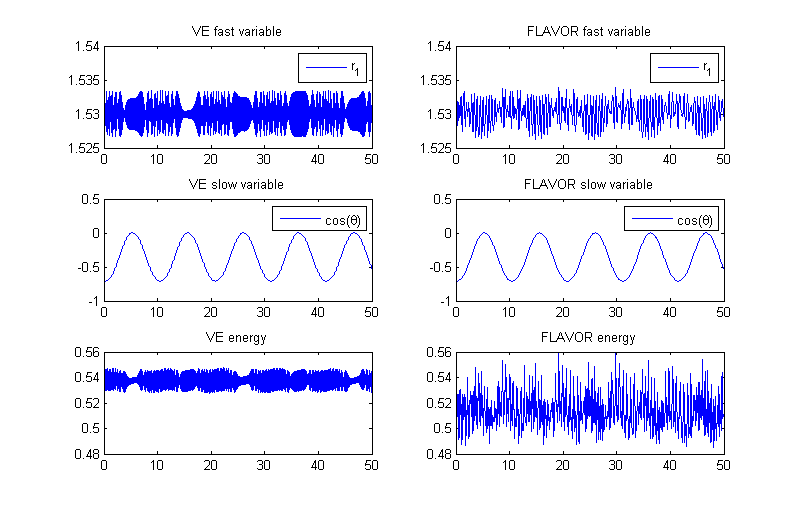}}}
\end{tabular}
\caption{\footnotesize Simulations of exaggerated propane molecule (Subsection \ref{expro}). Symplectic Euler uses $h=0.01$ and the induced symplectic FLAVOR (\eref{jhjhgjhgjg87g87A} and \eref{lsjfaijissdsdedlfskjlkd}) parameters are $\delta=0.1$ and $\tau=0.01$. Initial conditions are $[x_1,y_1,x_2,y_2,x_3,y_3]=[0,  0,  1.533,  0,  2.6136, 1.0826]$ and $[m_1\dot{x}_1,m_1\dot{y}_1,m_2\dot{x}_2,m_2\dot{y}_2,m_3\dot{x}_3,m_3\dot{y}_3]=[-0.4326,-1.6656,0.1253,0.2877,-1.1465,1.1909]$.}
\label{propane_modelresult}
\end{figure}

Figure \ref{propane_modelresult}  compares symplectic Euler  with the induced symplectic FLAVOR (\eref{jhjhgjhgjg87g87A} and \eref{lsjfaijissdsdedlfskjlkd}) applied in Euclidean coordinates. 10x acceleration is achieved. A simulation movie is also available at \href{http://www.cds.caltech.edu/~mtao/Propane.avi}{\url{http://www.cds.caltech.edu/~mtao/Propane.avi}} and \href{http://www.acm.caltech.edu/~owhadi/}{\url{http://www.acm.caltech.edu/~owhadi/}}.

\subsection{Forced nonautonomous mechanical system: Kapitza's inverted pendulum}
As the famous Kapitza's inverted pendulum shows \cite{Kapitza:65} (for recent references see \cite{Ariel:09} for numerical integration and \cite{SernaHammer} for generalization to the stochastic setting), the up position of a single pendulum can be stabilized if the pivot of the pendulum experiences external forcing in the form of vertical oscillation. Specifically, if the position of the pivot is given by $y=sin(\omega t)$, the system is governed by
\begin{equation}
    l\ddot{\theta}=[g+\omega^2 \sin(2\pi \omega t)]\sin \theta
\end{equation}
where $\theta$ denotes the clockwise angle of the pendulum from the positive $y$ direction, $l$ is the length of the pendulum and $g$ is the gravitational constant. In this case, the rapid vibration causes the pendulum to oscillate slowly around the positive $y$ direction with a $\mathcal{O}(1)$ frequency.

\begin{figure} [h]
\begin{tabular}{c}
{\resizebox{\imsizebig}{!}{\includegraphics{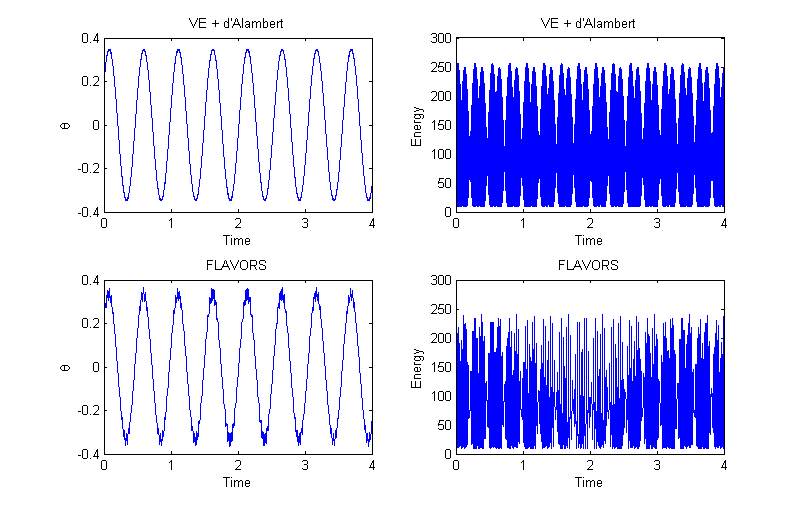}}}
\end{tabular}
\caption{Simulations of the inverted pendulum. The integration by Variational Euler + d'Alembert principle uses time step $h=0.2/\omega/\sqrt{l}\approx 0.000067$, while FLAVOR (defined by \eref{jakjshdhgshjd}) uses $\delta=0.002$ and $\tau=0.2/\omega/\sqrt{l}$. Also, $g=9.8$, $l=9$, $\theta(0)=0.2$, $\dot{\theta}(0)=0$ and $\omega=1000$}
\label{InvertedPendulum}
\end{figure}

A single scale integration of this system could be done by Variational Euler with discrete d'Alembert principle for external forces \cite{MaWe:01}
\begin{equation}
    \begin{cases}
        f_i=\omega^2 \sin(2\pi \omega i h) \\
        p_{i+1}=p_i+h[g+f_i] \sin \theta_i \\
        \theta_{i+1}=\theta_i+hp_{i+1}/l
    \end{cases}
\end{equation}
where the time step length $h$ has to be smaller than $\mathcal{O}(1/\omega)$.

 FLAVOR is given by
\begin{equation}\label{jakjshdhgshjd}
\begin{cases}
q_{n\delta+\tau}=q_{n\delta}+\tau p_{n\delta}/l \\
p_{n\delta+\tau}=p_{n\delta} +\tau g \sin(q_{n\delta+\tau})+ \omega^2 \sin(2\pi\omega n\tau)\\
q_{(n+1)\delta}=q_{n\delta+\tau}+(\delta-\tau) p_{n\delta+\tau}/l \\
p_{(n+1)\delta}=p_{n\delta+\tau} +(\delta-\tau)  g \sin(q_{(n+1)\delta})
\end{cases}
\end{equation}

Observe that the time dependent force is synchronized on the $\tau$ time scale instead of the $\delta$ time scale, specifically $\omega^2 \sin(2\pi\omega n\tau)$ instead of $\omega^2 \sin(2\pi\omega n\delta)$ in \eref{jakjshdhgshjd}

Numerical results are illustrated in Figure \ref{InvertedPendulum} (also available as a movie at \href{http://www.cds.caltech.edu/~mtao/InvertedPendulum.avi}{\url{http://www.cds.caltech.edu/~mtao/InvertedPendulum.avi}} and \href{http://www.acm.caltech.edu/~owhadi/}{\url{http://www.acm.caltech.edu/~owhadi/}}). Notice in this example that $\theta$, being the only degree of freedom, contains a combination of slow and fast dynamics. FLAVOR could only capture the fast dynamics in the sense of measures, and this is why dents appear
as modulation on
 the slow oscillation of $\theta$. On the other hand, although this forced system does not admit a conserved energy, the value of the Hamiltonian should oscillate periodically due to the periodic external driving force. While a non mechanics based method such as Forward Euler often produces an unbounded growth or a decrease in the energy, FLAVORS do not have this drawback.

\begin{Remark}
Consider the case of a rapid potential of the form $\Omega^2(q_1) q_2^2/\epsilon^2$ (where $q_1$ is the slow an $q_2$ the fast variable). In the limit of a vanishing $\epsilon$, it is known that the term contributes to the effective Hamiltonian with a contribution $V(q_1)$ (the so called Fixman term). On may have the intuition that FLAVOR would only be consistent with a term of the form $\gamma V(q_1)$, where $0<\gamma<1$ is only a fraction of one, because the rapid force is only accounted for over a time $\tau<\delta$. This intuition is not correct because the effect of FLAVOR is not to account the rapid force for over a time $\tau<\delta$ but to slow down the rapid force by a fraction $\tau/\epsilon$. This effect can also be seen in the algorithm \eref{jakjshdhgshjd} where the force term $\omega^2 \sin(2\pi\omega n\tau)$ has been slowed down by a factor $\tau/\delta$ (the Kapitza's inverted pendulum illustrates a similar phenomenon where rapid oscillations contribute a stabilizing term to the effective Hamiltonian, nevertheless FLAVORS remain accurate).
\end{Remark}

\begin{figure} [h]
\begin{tabular}{c}
{\resizebox{\imsizebig}{!}{\includegraphics{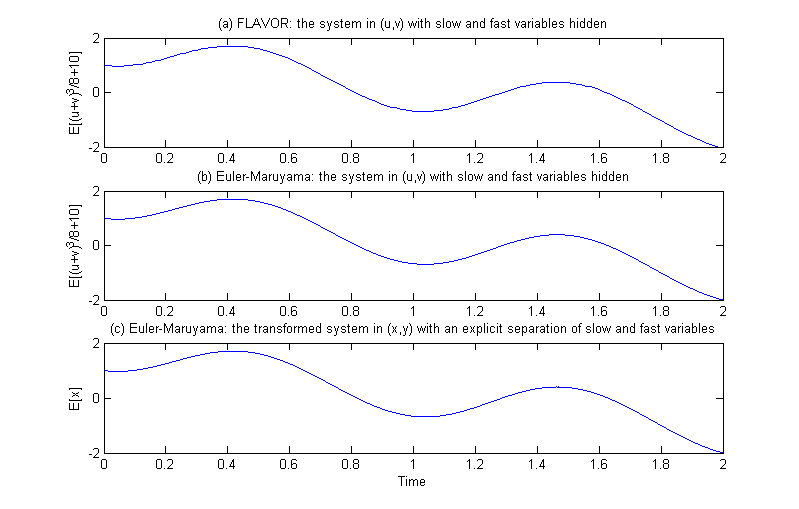}}}
\end{tabular}
\caption{\footnotesize (a) Integration of \eqref{general_SDE_mixedup} by nonintrusive FLAVOR \eref{ksjahussahshdskdcrrfsdaue0} using mesostep step $\delta=0.01$ (b) Integration of \eqref{general_SDE_mixedup} by Euler-Maruyama using fine time step $h=10^{-4}$ (c) Integration of \eqref{general_SDE_separated} by Euler-Maruyama using the same small step $h=10^{-4}$. Expectations of the slow variable (whether or not hidden) are obtained by empirically averaging over an ensemble of 100 independent sample trajectories. $\epsilon=10^{-4}$, $x(0)=1+\epsilon$, $y(0)=1$, $T=2$ (the expectation of the real solution will blow up around $T=3$). We have chosen $c=10$ so that the transformation is a diffeomorphism.}
\label{example_general_SDE_mixedup}
\end{figure}

\subsection{Nonautonomous SDE system with hidden slow variables}

Consider the following artificial nonautonomous SDE system
\begin{equation}
    \begin{cases}
        du=\frac{4}{3(u+v)^2}\left(-\frac{1}{2}\left(\frac{v-u}{2}\right)^2+5\sin(2\pi t)\right)dt-\frac{1}{\epsilon}\left(\left(\frac{u+v}{2}\right)^3+c-\frac{v-u}{2}\right)dt-\sqrt{\frac{2}{\epsilon}}dW_t \\
        dv=\frac{4}{3(u+v)^2}\left(-\frac{1}{2}\left(\frac{v-u}{2}\right)^2+5\sin(2\pi t)\right)dt+\frac{1}{\epsilon}\left(\left(\frac{u+v}{2}\right)^3+c-\frac{v-u}{2}\right)dt+\sqrt{\frac{2}{\epsilon}}dW_t
    \end{cases}
    \label{general_SDE_mixedup}
\end{equation}
where $c$ is a positive constant and the two $dW_t$ terms refer to the same Brownian Motion.
The system \eref{general_SDE_mixedup} can be converted via the local diffeomorphism
\begin{equation}
    \begin{cases}
        u=(x-c)^{1/3}-y \\
        v=(x-c)^{1/3}+y
    \end{cases} ,
\end{equation}
into the following hidden system separating slow and fast variables
\begin{equation}
    \begin{cases}
        dx=-\frac{1}{2}y^2 dt+5\sin(2\pi t)dW_t \\
        dy=\frac{1}{\epsilon} (x-y)dt + \sqrt{\frac{2}{\epsilon}}dW_t
    \end{cases} .
    \label{general_SDE_separated}
\end{equation}
Nonintrusive FLAVOR \eref{ksjahussahshdskdcrrfsdaue0} can be directly applied to \eqref{general_SDE_mixedup} using a time step $\delta\gg \epsilon$ without prior identification of the slow and fast variables, i.e., without prior identification of the slow variable $x$ or of the system \eref{general_SDE_separated}.
The expected values of solutions of \eref{general_SDE_mixedup} integrated by FLAVORS with mesostep $\delta$ and Euler-Maruyama with a small time step $\tau$ are presented in Figure \ref{example_general_SDE_mixedup}. FLAVOR has accelerated the computation by 100x.

\begin{figure} [h]
\begin{tabular}{c}
{\resizebox{\imsizebig}{!}{\includegraphics{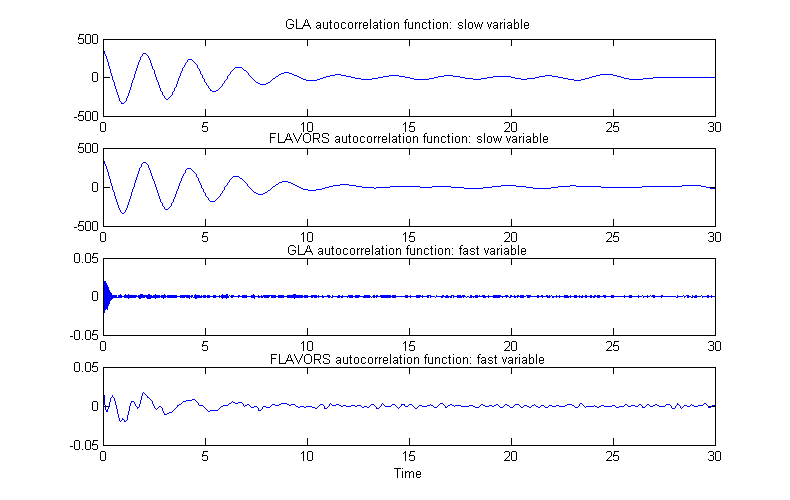}}}
\end{tabular}
\caption{SDE \eqref{langevin_slow_nonlinear}: autocorrelation functions of $\mathbb{E}[y(t)y(0)]$ (dominantly fast) and of $\mathbb{E}[(x(t)-y(t))(x(0)-y(0))]$ (dominantly slow), empirically obtained by GLA and FLAVORS.}
\label{langevin_slow_nonlinear_auto}
\end{figure}

\begin{figure} [ht]
\centering
\subfigure[$\mathbb{E}\big(x(t)-y(t)\big)$]{
\includegraphics[width=0.46\textwidth]{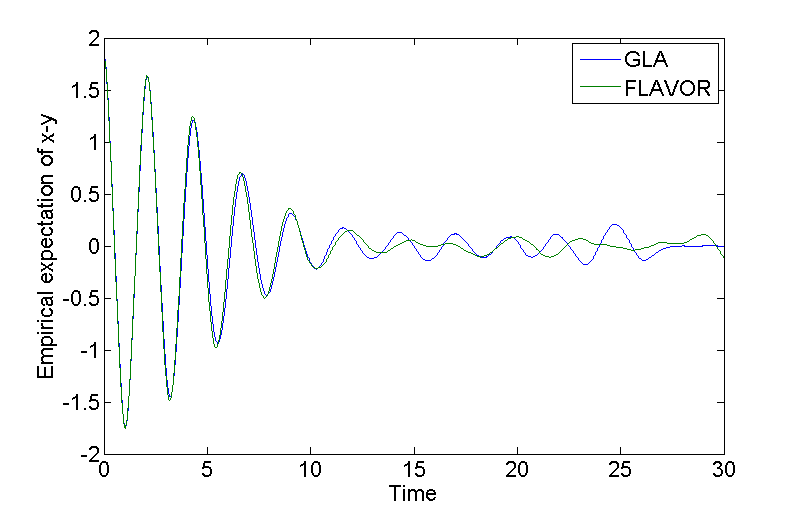}
\label{langevin_slow_nonlinear_moment1}
}
~
\subfigure[$\mathbb{E}\big((x(t)-y(t))^2\big)$]{
\includegraphics[width=0.46\textwidth]{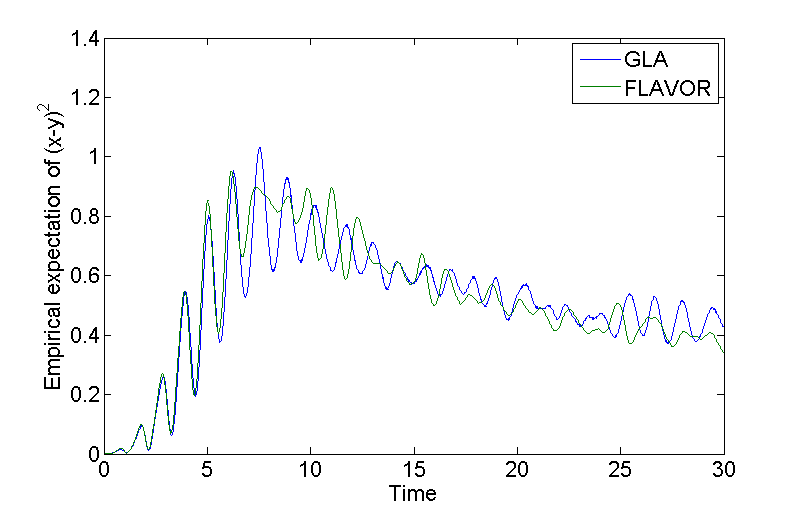}
\label{langevin_slow_nonlinear_moment2}
}
\caption{SDE \eqref{langevin_slow_nonlinear}: Empirical moments obtained from simulations of ensembles of \eqref{langevin_slow_nonlinear} with GLA and quasi-symplectic FLAVOR (Subsection \ref{qsfla})}
\end{figure}

\subsection{Langevin equations with slow noise and friction}

In this subsection, we consider the one dimensional, two degrees of freedom system   modeled by the SDEs (now both springs are quartic rather than harmonic):
\begin{equation}
    \begin{cases}
        dy=p_y dt \\
        dx=p_x dt \\
        dp_y=-\epsilon^{-1} y^3 dt - 4(y-x)^3 dt - cp_y dt + \sigma dW^1_t \\
        dp_x=-4(x-y)^3 dt - cp_x dt + \sigma dW^2_t
    \end{cases} .
    \label{langevin_slow_nonlinear}
\end{equation}

We compare several autocorrelation functions and time-dependent moments of this stochastic process integrated by quasi-symplectic FLAVOR (\eref{kdfsjhdskdcrrfdaue0} and \eref{lsjfaijissdsdedlfskjlkd}) and Geometric Langevin Algorithm (GLA) \cite{BoOw:09}. FLAVOR and GLA gave results in agreement (Figure \ref{langevin_slow_nonlinear_auto},
\ref{langevin_slow_nonlinear_moment1}, \ref{langevin_slow_nonlinear_moment2}). Since GLA is weakly-convergent and Boltzmann-Gibbs preserving, this is numerical evidence that quasi-symplectic FLAVOR is also.

Expectations are empirically calculated by averaging over an ensemble of 100 sample trajectories with $T=30$, $\epsilon=10^{-8}$, $\tau=0.001$, $\delta=0.01$. $y(0)=2.1/\omega$ (with $\omega:=1/\sqrt{\epsilon}$), $x(0)=y(0)+1.8$, $c=0.1$ and $\sigma=0.5$. GLA uses time step $h=0.001$. Noise and friction are slow here in the sense that they are not of the order $\mathcal{O}(\omega)$ or larger.

As shown in the plots, in the regime dominated by deterministic dynamics (roughly from $t=0$ to $t=8$) various moments calculated empirically by FLAVORS and GLA are in agreement, indicating that the same rate of convergence towards the Boltzmann-Gibbs distribution is obtained. And in that regime, autocorrelation functions of the slow variables agree, serving as numerical evidence that FLAVORS is weakly converging towards the SDE solution, whereas autocorrelation functions of the fast variables agree only in the sense of measures (after time averaging over a mesoscopic ($\mathcal{o}(1)$) time span).
The fluctuations between  FLAVORS and GLA for large time are an effect of the finite number of samples ($100$) used to compute sample averages.

Recall that if the noise is applied to slow variables, FLAVORS do not converge strongly but only in the sense of distributions.

\subsection{Langevin equations with fast noise and friction}
\begin{figure} [h]
\begin{tabular}{c}
{\resizebox{\imsizebig}{!}{\includegraphics{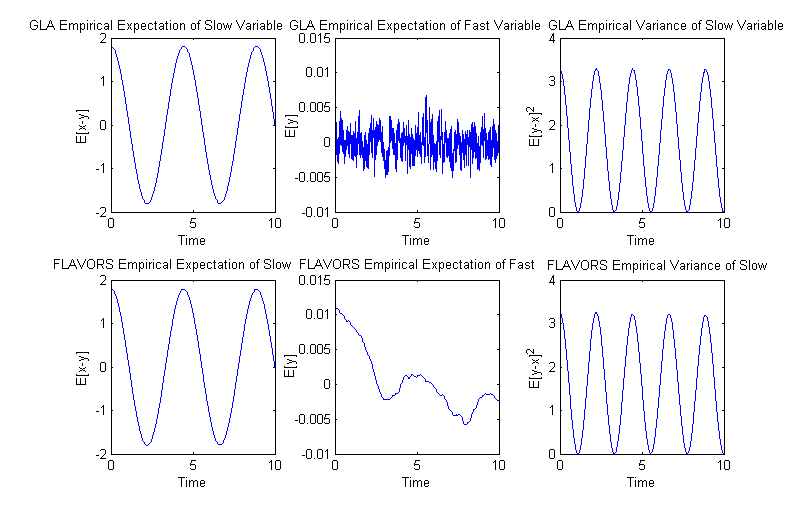}}}
\end{tabular}
\caption{\footnotesize $\mathbb{E}[x(t)-y(t)]$, $\mathbb{E}[y(t)]$, and $\mathbb{E}[x(t)-y(t)]^2$ obtained by GLA and quasi-symplectic FLAVOR (Subsection \ref{qsfla}). Expectations are empirically calculated by averaging over an ensemble of 50 sample trajectories with $T=10$, $\omega=100$, $\tau=10^{-4}$, $\delta=0.01$. $y(0)=1.1/\omega$, $x(0)=y(0)+1.8$, $c=0.1$ and $\sigma=1$. GLA uses  time step $h=10^{-4}$. }
\label{langevin_fast}
\end{figure}

Consider a system with the same configuration as above. The difference is that the soft spring oscillates at a frequency nonlinearly dependent on the stiff spring's length, and the left mass experiences strong friction and noise while the right mass does not.
The Hamiltonian is
\begin{equation}
    H(y,x,p_y,p_x)=\frac{1}{2}p_y^2+\frac{1}{2}p_x^2/2+\frac{1}{4}\omega^4 y^4+e^y (x-y)^2 ,
\end{equation}
and the governing SDEs are:
\begin{equation}
    \begin{cases}
        dy=p_y dt\\
        dx=p_x dt\\
        dp_y=-\omega^4 y^3 dt - (2+y-x)(y-x) e^y dt - \omega^2 cp_y dt + \omega \sigma dW^t \\
        dp_x=-2(x-y)e^y dt
    \end{cases} .
\end{equation}
In this system, the deterministic dynamics and the effects of noise and friction both involve  a $\mathcal{O}(1/\omega^2)$ timescale. We have implemented the fast noise and friction version of FLAVORS (\eref{kdfsjhdiiuskdcrrfdauijhe0} and \eref{lsjfaijissdsdedlfskjlkd}).

In Figure \ref{langevin_fast}, we have
 plotted the first and second moments of the dominantly slow variable $x(t)-y(t)$ as well as the first moment of the dominantly fast variable $y(t)$ as functions of time. Moments of the dominantly slow variable integrated by quasi-symplectic FLAVOR (Subsection \ref{qsfla}) and GLA \cite{BoOw:09} concur, numerically suggesting weak convergence and preservation of Boltzmann-Gibbs. 100x computational acceleration is achieved.

\section{Appendix}
\subsection{Proof of theorems \ref{thm01} and \ref{thm01b}}\label{subap1}

Define the process $t\mapsto (\bar{x}_{t},\bar{y}_{t})$ by
\begin{equation}\label{ksjdededhdskdjjhwue}
(\bar{x}_{t},\bar{y}_{t}):=\eta(\bar{u}_t) .
\end{equation}
It follows from the regularity of $\eta$ that it is sufficient to prove the
$F$-convergence of $(\bar{x}_{t},\bar{y}_{t})$ towards $\delta_{X_t}\otimes \mu(X_t,dy)$. Moreover, it is also sufficient to prove the following inequalities \eref{jkdsgkdjshgdsdkjghe11}, \eref{lkslkcdsdehkhsdsddedj311} in order to obtain inequalities \eref{jkdsgkdjshgdsdkjghe} and \eref{lkslkcdhkhsdedj3}
\begin{equation}\label{jkdsgkdjshgdsdkjghe11}
\begin{split}
|x_{t}^\epsilon-\bar{x}_{t}|\leq C e^{C t} \psi_1(u_0,\epsilon,\delta,\tau)
\end{split}
\end{equation}
and
\begin{equation}\label{lkslkcdsdehkhsdsddedj311}
\begin{split}
\Big|\frac{1}{T}\int_{t}^{t+T}\varphi(\bar{x}_{s},\bar{y}_s)\,ds-\int_{\R^p} \varphi(X_t,y)\mu(X_t,dy)\Big|\leq \psi_2(u_0,\epsilon,\delta,\tau,T,t) (\|\varphi\|_{L^\infty}+\|\nabla \varphi\|_{L^\infty} )
\end{split}
\end{equation}
Now define $\psi^\epsilon_\tau$ by
\begin{equation}
\psi^\epsilon_\tau(x,y):= \eta\circ \theta^{\epsilon}_\tau\circ \eta^{-1}(x,y)
\end{equation}
Define $\psi^g_h$ by
\begin{equation}
\psi^g_h(x,y):= \eta\circ \theta^{G}_h\circ \eta^{-1}(x,y)
\end{equation}
\begin{Proposition}\label{kdlhkjhwjd2d}
The vector fields $f$ and $g$ associated with the system of Equations \eref{kfgfgdiuuiusedejhd} are  Lipschitz continuous. We also have
\begin{equation}\label{ksjhdskdxddxjjhwue}
(\bar{x}_{t},\bar{y}_{t})=\big(\psi^g_{\delta-\tau}\circ \psi^\epsilon_{\tau}\big)^k(x_0,y_0)\quad \text{for}\quad k\delta \leq t <(k+1)\delta .
\end{equation}
Moreover, there exists $C>0$ such that for $h\leq h_0$ and $\frac{\tau}{\epsilon}\leq \tau_0$ we have
\begin{equation}\label{lsfde3eldeiuuddiuskjlkd}
\big|\psi_\tau^{\epsilon}(x,y)-(x,y)-\tau \big(g(x,y),0\big)-\frac{\tau}{\epsilon}\big(0,f(x,y)\big) \big|\leq C \big(\frac{\tau}{\epsilon}\big)^2
\end{equation}
and
\begin{equation}\label{lsfddeiuskjddflkd}
\big|\psi_h^{g}(x,y)-(x,y)-h \big(g(x,y),0\big) \big|\leq C h^2 .
\end{equation}
Furthermore, given $x_0,y_0$, the trajectories of $(x^\epsilon_{t},y^\epsilon_{t})$ and $(\bar{x}_{t},\bar{y}_{t})$ are uniformly bounded in $\epsilon$, $\delta \leq h_0$, $\tau \leq \min(\tau_0 \epsilon, \delta)$.
\end{Proposition}
\begin{proof}
Since $(x,y)=\eta(u)$, we have
\begin{equation}
\dot{x}=(G+\frac{1}{\epsilon}F)\nabla \eta^x \circ \eta^{-1}(x,y)
\end{equation}
\begin{equation}
\dot{y}=(G+\frac{1}{\epsilon}F)\nabla \eta^y \circ \eta^{-1}(x,y) .
\end{equation}
Hence, we deduce from Equation \eref{kfgfgdiuuiusedejhd} of Condition \ref{lsdsddsdeehxA1} that
\begin{equation}
g(x,y)=G \nabla \eta^x \circ \eta^{-1}(x,y)
\end{equation}
\begin{equation}
f(x,y)=F \nabla \eta^y \circ \eta^{-1}(x,y) .
\end{equation}
We deduce the regularity of $f$ and $g$ from the regularity of $G$, $F$ and $\eta$.
Equation \eref{ksjhdskdxddxjjhwue} is a direct consequence of  the definition of $\psi^\epsilon_\tau$ and $\psi^g_h$ and Equation \eref{ksjhdskdjjhwue} (we write $(x_0,y_0):=\eta(u_0)$).
Observe that Equation \eref{kfgfgdiuuiusedejhd} of Condition \ref{lsdsddsdeehxA1} also requires that
\begin{equation}\label{dklhwshd23}
F\nabla \eta^x=0 \quad G\nabla \eta^y=0 .
\end{equation}
Now observe that
\begin{equation}
\begin{split}
\psi^\epsilon_\tau(x,y)-(x,y)&-\big(g(x,y),0\big)\tau-\big(0,f(x,y)\big)\frac{\tau}{\epsilon}=
\\&\big(\eta\circ \theta^{\epsilon}_\tau-\eta-\tau \big(G \nabla \eta^x,0\big)-\frac{\tau}{\epsilon}
\big(0,F \nabla \eta^y\big)  \big)\circ \eta^{-1}(x,y) .
\end{split}
\end{equation}
Using \eref{dklhwshd23}, \eref{lsfddlfskjlkd}, Taylor expansion and the regularity of $\eta$ we obtain \eref{lsfde3eldeiuuddiuskjlkd}.
Similarly
\begin{equation}
\psi^g_h(x,y)-(x,y)-h\big(g(x,y),0\big):= \Big(\eta\circ \theta^{G}_h-\eta(x,y)-h\big(G \nabla \eta^x,0\big)\Big)\circ \eta^{-1}(x,y) .
\end{equation}
Using \eref{dklhwshd23}, \eref{hgfjhgdfsjhgfhf}, Taylor expansion and the regularity of $\eta$ we obtain \eref{lsfddeiuskjddflkd}.
The uniform bound (depending on $x_0,y_0$)  on the trajectories of $(x^\epsilon_{t},y^\epsilon_{t})$ and $(\bar{x}_{t},\bar{y}_{t})$ is a consequence of the uniform bound (given $u_0$) on the trajectories of $u^\epsilon_t$ and $\bar{u}_t$.
\end{proof}
It follows from Proposition \ref{kdlhkjhwjd2d} that it is sufficient to prove theorems \ref{thm01} and \ref{thm01b} in the situation where $\eta$ is the identity diffeomorphism. More precisely, the $F$-convergence of $\bar{u}_t$ is a consequence of the $F$-convergence of $(\bar{x}_t,\bar{y}_t)$ and the regularity of $\eta$. Furthermore, from the uniform bound (depending on $(x_0,y_0)$)  on the trajectories of $(x^\epsilon_{t},y^\epsilon_{t})$ and $(\bar{x}_{t},\bar{y}_{t})$ we deduce that $g$ and $f$ are uniformly bounded and Lipshitz continuous (in $\epsilon$, $\delta \leq h_0$, $\tau \leq \min(\tau_0 \epsilon, \delta)$) over those trajectories.

Define
\[\bar{g}:=\int g(x,y)\,\mu(x,dy)\]
where $\mu$ is the family of measures introduced in Condition \ref{lsdsddsdeehxA2}.
Let us prove the following lemma.
\begin{Lemma}\label{lkhdkjwhdkhe}
\begin{equation}
\begin{split}
|x_{n\delta}^\epsilon-\bar{x}_{n\delta}|\leq& Ce^{Cn\delta} \Big( \delta+\big(\frac{\tau}{\epsilon}\big)^2 \frac{1}{\delta}
+\sup_{1\leq l \leq n} |J(l)|\Big)
\end{split}
\end{equation}
with $J(k)=J_1(k)+J_2(k)$,
\begin{equation}
J_1(k):=\sum_{n=0}^{k-1} \big(\int_{n\delta}^{(n+1)\delta} g(x_{n\delta}^\epsilon,y_s^\epsilon)\,ds- \delta \bar{g}(x_{n\delta}^\epsilon\big)
\end{equation}
and
\begin{equation}
J_2(k):=\sum_{n=0}^{k-1} \delta \big(\bar{g}(\bar{x}_{n \delta})-  g(\bar{x}_{n\delta}, \bar{y}_{n\delta})\big)
\end{equation}
\end{Lemma}
\begin{proof}
Observe that
\begin{equation}
x_{(n+1)\delta}^\epsilon=x_{n\delta}^\epsilon+\int_{n\delta}^{(n+1)\delta} g(x_{n \delta}^\epsilon,y_s^\epsilon)\,ds+\int_{n \delta}^{(n+1) \delta} (g(x_{s}^\epsilon,y_s^\epsilon)-g(x_{n \delta}^\epsilon,y_s^\epsilon))
\,ds
\end{equation}
Hence,
\begin{equation}
\begin{split}
x_{(n+1)\delta}^\epsilon-\bar{x}_{(n+1)\delta}=x_{n\delta}^\epsilon-\bar{x}_{n \delta}+I_1+I_2(n)+I_3+I_4(n)+I_5
\end{split}
\end{equation}
with
\begin{equation}
I_1:=\int_{n \delta}^{(n+1) \delta} (g(x_{s}^\epsilon,y_s^\epsilon)-g(x_{n \delta}^\epsilon,y_s^\epsilon)) ds
\end{equation}
\begin{equation}
I_2(n):=\int_{n \delta}^{(n+1) \delta} g(x_{n \delta}^\epsilon,y_s^\epsilon)\,ds- \delta \bar{g}(x_{n \delta}^\epsilon)
\end{equation}
\begin{equation}
I_3:=\delta \big(\bar{g}(x_{n \delta}^\epsilon)-\bar{g}(\bar{x}_{n \delta})\big)
\end{equation}
\begin{equation}
I_4(n):=\delta \big(\bar{g}(\bar{x}_{n \delta})-  g(\bar{x}_{n\delta}, \bar{y}_{n\delta})\big)
\end{equation}
\begin{equation}
I_5:=\delta g(\bar{x}_{n\delta}, \bar{y}_{n\delta})-(\bar{x}_{(n+1)\delta}-\bar{x}_{n \delta})
\end{equation}
Now observe that
\begin{equation}
|I_1|\leq \|\nabla_x g\|_{L^\infty} \delta^2
\end{equation}
and
\begin{equation}
|I_3|\leq \delta \|\nabla_x g\|_{L^\infty} |x_{n\delta }^\epsilon-\bar{x}_{n\delta}| .
\end{equation}
Using \eref{lsfde3eldeiuuddiuskjlkd} and \eref{lsfddeiuskjddflkd} we obtain that
\begin{equation}
|I_5|\leq C \Big(\delta^2+\big(\frac{\tau}{\epsilon}\big)^2\Big)
\end{equation}
Combining the previous equations, we have obtained that
\begin{equation}\label{ineq1}
\begin{split}
x_{(n+1)\delta}^\epsilon-\bar{x}_{(n+1)\delta}\leq x_{n\delta}^\epsilon-\bar{x}_{n \delta}+C \Big(\delta^2+\big(\frac{\tau}{\epsilon}\big)^2\Big) + C \delta |x_{n\delta}^\epsilon-\bar{x}_{n \delta}|+(I_2+I_4)(n)
\end{split}
\end{equation}
and
\begin{equation}\label{ineq2}
\begin{split}
x_{(n+1)\delta}^\epsilon-\bar{x}_{(n+1)\delta}\geq x_{n\delta}^\epsilon-\bar{x}_{n \delta}-C \Big(\delta^2+\big(\frac{\tau}{\epsilon}\big)^2\Big) - C \delta |x_{n\delta}^\epsilon-\bar{x}_{n \delta}|+(I_2+I_4)(n)
\end{split}
\end{equation}
Write
\begin{equation}
\begin{split}
J(n):=\sum_{k=0}^{n-1}(I_2+I_4)(k)
\end{split}
\end{equation}
Summing the first $n$ inequalities \eref{ineq1} and \eref{ineq2}
\begin{equation}
\begin{split}
x_{n\delta}^\epsilon-\bar{x}_{n\delta}\leq C \Big(\delta^2+\big(\frac{\tau}{\epsilon}\big)^2\Big) n + C \delta \sum_{k=0}^{n-1}|x_{k\delta}^\epsilon-\bar{x}_{k \delta}|+J(n)
\end{split}
\end{equation}
\begin{equation}
\begin{split}
x_{n\delta}^\epsilon-\bar{x}_{n\delta}\geq - C \Big(\delta^2+\big(\frac{\tau}{\epsilon}\big)^2\Big) n - C \delta \sum_{k=0}^{n-1}|x_{k\delta}^\epsilon-\bar{x}_{k \delta}|+J(n)
\end{split}
\end{equation}
Hence
\begin{equation}
\begin{split}
|x_{n\delta}^\epsilon-\bar{x}_{n\delta}|\leq C \Big(\delta^2+\big(\frac{\tau}{\epsilon}\big)^2\Big) n + C \delta \sum_{k=0}^{n-1}|x_{k\delta}^\epsilon-\bar{x}_{k \delta}|+|J(n)|
\end{split}
\end{equation}
And we obtain by induction
\begin{equation}\label{klsdkdlshdj}
\begin{split}
|x_{n\delta}^\epsilon-\bar{x}_{n\delta}|\leq& C \Big(\delta^2+\big(\frac{\tau}{\epsilon}\big)^2\Big)\big(n+ C \delta\sum_{k=1}^n (n-k)(1+C \delta)^{k-1}\big)\\
& |J(n)|+C \delta \sum_{l=2}^n (1+ C \delta )^{l-2}|J(n-l+1)|
\end{split}
\end{equation}
Equation \eref{klsdkdlshdj} concludes the proof of  Lemma \ref{lkhdkjwhdkhe}.
\end{proof}
We now need to control $J_1(k)$ and $J_2(k)$. First, let us prove the following lemma.

\begin{Lemma}\label{wgehedghg3d}
For $N\in \N^*$ we have
\begin{equation}
|J_1(k)|\leq (\delta k) C \Big(\delta e^{C \frac{\delta}{N\epsilon}}+ E\big(\frac{\delta}{N \epsilon}\big)\Big)
\end{equation}
\end{Lemma}

\begin{proof}
Define $\hat{y}^\epsilon_t$ such that $\hat{y}^\epsilon_t=y^\epsilon_t$ for $t=(n+j/N) \delta$, $j\in \N^*$, and
\begin{equation}
\frac{d \hat{y}^\epsilon_t}{dt}=\frac{1}{\epsilon}f(x_{n\delta}^\epsilon,\hat{y}^\epsilon_t)\quad \text{for}\quad (n+j/N) \delta \leq t<(n+(j+1)/N) \delta .
\end{equation}
Using the regularity of $f$ and $g$ we obtain that
\begin{equation}\label{skldlhskhjd}
|\hat{y}^\epsilon_t-y^\epsilon_t|\leq C \delta e^{C \frac{\delta}{N\epsilon}} .
\end{equation}
First, observe that
\begin{equation}
\begin{split}
\frac{1}{\delta}\int_{n\delta}^{(n+1)\delta} g(x_{n\delta}^\epsilon,y_s^\epsilon)\,ds-\bar{g}(x_{n\delta}^\epsilon)=K_1+K_2
\end{split}
\end{equation}
with
\begin{equation}
\begin{split}
K_1:=\frac{1}{\delta}\sum_{j=0}^{N-1} \int_{(n+j/N)\delta}^{(n+(j+1)/N)\delta} \big(g(x_{n\delta}^\epsilon,y_s^\epsilon)-g(x_{n\delta}^\epsilon,\hat{y}_s^\epsilon)\big)\,ds
\end{split}
\end{equation}
and
\begin{equation}
\begin{split}
K_2:=\frac{1}{N}\sum_{j=0}^{N-1} \Big(\frac{N}{\delta}\int_{(n+j/N)\delta}^{(n+(j+1)/N)\delta} g(x_{n\delta}^\epsilon,\hat{y}_s^\epsilon)\,ds- \bar{g}(x_{n\delta}^\epsilon)\Big) .
\end{split}
\end{equation}
We have

\begin{equation}
\begin{split}
|K_1|\leq \|\nabla_y g\|_{L^\infty}\frac{1}{N}\sum_{j=0}^{N-1} \sup_{(n+j/N)\delta\leq s \leq (n+(j+1)/N)\delta} |y_s^\epsilon-\hat{y}_s^\epsilon| .
\end{split}
\end{equation}

Hence, we obtain from \eref{skldlhskhjd} that
\begin{equation}
\begin{split}
|K_1|\leq C \delta e^{C \frac{\delta}{N\epsilon}}
\end{split}
\end{equation}
Moreover, we obtain from conditions \ref{lsdsddsdeehxA2} and \ref{lsdsddsdeehxA3} that
\begin{equation}
\begin{split}
|K_2|\leq C E\big(\frac{\delta}{N \epsilon}\big)
\end{split}
\end{equation}
This concludes the proof of Lemma \ref{wgehedghg3d}.
\end{proof}

\begin{Lemma}\label{gjkwhgheegd}
We have for $m\in \N^*$
\begin{equation}
\big|J_2(k)\big| \leq C \delta k \Big(m \delta + E(\frac{m \tau}{\epsilon})+\big(\frac{\tau}{\epsilon}+ m \delta+m (\frac{\tau}{\epsilon})^2  \big)e^{C \frac{m\tau}{\epsilon}} \Big) .
\end{equation}
\end{Lemma}
\begin{proof}
Let $m\in \N^*$. Define $(\tilde{x}_s,\tilde{y}_{s})$ such that for $j\in \N^*$, $n\in \N^*$,

\begin{equation}
\begin{cases}
\frac{d\tilde{x}_{s }}{dt}=g(\tilde{x}_s,\tilde{y}_{s}) \quad \text{for}\quad jm \delta  \leq s < (j+1)m\delta\\
\frac{d\tilde{y}_{s }}{dt}=\frac{1}{\epsilon}f(\tilde{x}_s,\tilde{y}_{s})\quad \text{for}\quad n\delta  \leq s < n\delta+\tau\\
\tilde{y}_{s }=\tilde{y}_{n\delta+\tau} \quad \text{for}\quad n\delta +\tau \leq s < (n+1)\delta \\
\tilde{y}_{(n+1)\delta }=\tilde{y}_{n\delta+\tau} \quad \text{for}\quad n+1\not=jm \\
(\tilde{x}_{jm },\tilde{y}_{jm })=(\bar{x}_{jm \delta},\bar{y}_{jm \delta})
\end{cases} .
\end{equation}

Define $\tilde{y}_{s}^a$ by
\begin{equation}\label{kdgskdhghgsd}
\begin{cases}
\frac{d \tilde{y}_{t}^a}{dt}=\frac{1}{\epsilon}f(\bar{x}_{jm \delta},\tilde{y}_{t}^a) \quad \text{for}\quad jm\tau  \leq t < (j+1)m\tau\\
\tilde{y}_{jm\tau}^a=\bar{y}_{jm \delta}
\end{cases} ,
\end{equation}
and define $\tilde{x}_{n}^a$ by
\begin{equation}\label{sdhgskjdgh}
\tilde{x}_{n}^a=\bar{x}_{jm \delta}  \quad \text{for}\quad  jm \leq n < (j+1)m .
\end{equation}
Observe that
\begin{equation}
J_2(k) =K_3+K_4+K_5+K_6+K_7
\end{equation}
with
\begin{equation}
K_3:=\sum_{n=0}^{k-1}\big(\int_{n\delta}^{(n+1)\delta} g(\tilde{x}_{s}, \tilde{y}_{s})\,ds-  \delta  g(\bar{x}_{n\delta}, \bar{y}_{n\delta})\big) ,
\end{equation}
\begin{equation}\label{kdsdsgd}
K_4:=\sum_{n=0}^{k-1} \delta \big(\frac{1}{\tau}\int_{n\tau}^{(n+1)\tau} g(\tilde{x}_{n}^a, \tilde{y}_{s}^a)\,ds-  \frac{1}{\delta}\int_{n\delta}^{(n+1)\delta} g(\tilde{x}_{s}, \tilde{y}_{s})\,ds \big) ,
\end{equation}
\begin{equation}
K_5:=\frac{\delta}{\tau}\sum_{n=0}^{k-1}  \big(\tau \bar{g}(\tilde{x}_{n}^a)
-\int_{n\tau}^{(n+1)\tau} g(\tilde{x}_{n}^a, \tilde{y}_{s}^a)\,ds\big) ,
\end{equation}
\begin{equation}
K_6:=\delta \sum_{n=0}^{k-1}\big(  \bar{g}(\bar{x}_{n \delta})- \bar{g}(\tilde{x}_{n }^a)\big) .
\end{equation}
Using the regularity of $g$ we obtain
\begin{equation}
|K_6|\leq \delta k C  \delta m .
\end{equation}
Arranging the right hand side of \eref{kdsdsgd} into groups of $m$ terms corresponding to the intervals of \eref{kdgskdhghgsd} we obtain, from  Condition \ref{lsdsddsdeehxA2} and  Condition \ref{lsdsddsdeehxA3}, that
\begin{equation}
|K_5|\leq C k \delta E(\frac{m\tau}{\epsilon}) .
\end{equation}
Using \eref{sdhgskjdgh} and the regularity of $f$ and $g$ we obtain the following  inequality
\begin{equation}\label{kdgsksddsdgsd}
|\tilde{y}_{\frac{\delta}{\tau}t}^a-\tilde{y}_t|\leq C m\delta e^{C\frac{m\tau}{\epsilon}} .
\end{equation}
It follows that
\begin{equation}\label{kdssddsdsgd}
|K_4|\leq C \delta k m\delta e^{C\frac{m\tau}{\epsilon}} .
\end{equation}
Similarly, using \eref{lsfde3eldeiuuddiuskjlkd} and \eref{lsfddeiuskjddflkd}, we obtain the following inequalities
\begin{equation}\label{hgfjhgfsjhgsdedfhfo1}
|\tilde{y}_{n\delta}-\bar{y}_{n\delta}|\leq C (\frac{\tau}{\epsilon}+m\delta+m (\frac{\tau}{\epsilon})^2)\frac{m\tau}{\epsilon} e^{C \frac{m\tau}{\epsilon}} ,
\end{equation}
\begin{equation}\label{hgfjhgfjhgaswsdedfhfo1}
|\tilde{x}_{n\delta}-\bar{x}_{n\delta}|\leq C m\big(\delta+(\frac{\tau}{\epsilon})^2\big) .
\end{equation}
It follows that
\begin{equation}
|K_3|\leq C \delta k \big(\frac{\tau}{\epsilon}+m\delta+m(\frac{\tau}{\epsilon})^2\big) e^{C \frac{m\tau}{\epsilon}} .
\end{equation}
This concludes the proof of Lemma \ref{gjkwhgheegd}.
\end{proof}

Combining Lemma \ref{lkhdkjwhdkhe}, \ref{wgehedghg3d} and \ref{gjkwhgheegd} we have obtained that
\begin{equation}
\begin{split}
|x_{n\delta}^\epsilon-\bar{x}_{n\delta}|\leq & C  e^{C \delta n}
\Big(\delta+\big(\frac{\tau}{\epsilon}\big)^2 \frac{1}{\delta}+\delta  e^{C \frac{\delta}{N\epsilon}}+ E\big(\frac{\delta}{N \epsilon}\big)+ E(\frac{m \tau}{\epsilon})\\&+\big(\frac{\tau}{\epsilon}+ m \delta+m(\frac{\tau}{\epsilon})^2  \big)e^{C \frac{m\tau}{\epsilon}}
\Big)
\end{split}
\end{equation}
Choosing $N$ such that $ e^{C \frac{\delta}{N\epsilon}} \sim \delta^{-\frac{1}{2}}$ (observe that we need $\epsilon\leq \delta/(-C\ln \delta)$) and $m$ such that
$\frac{m\tau}{\epsilon} e^{C \frac{m\tau}{\epsilon}}\sim \big(\frac{\delta \epsilon}{\tau}+\frac{\tau}{\epsilon}\big)^{-\frac{1}{2}}$ we obtain
for $\frac{\delta \epsilon}{\tau}+\frac{\tau}{\epsilon}\leq 1$
 that
\begin{equation}
\begin{split}
|x_{n\delta}^\epsilon-\bar{x}_{n\delta}|\leq &C e^{C \delta n}
\Bigg(\sqrt{\delta}+\big(\frac{\tau}{\epsilon}\big)^2 \frac{1}{\delta}+E\big(\frac{1}{C}\ln \frac{1}{\delta}\big)\\&+\big(\frac{\delta \epsilon}{\tau}\big)^{\frac{1}{2}}+\big(\frac{\tau}{\epsilon}\big)^{\frac{1}{2}}+ E\Big(\frac{1}{C}\ln \Big(\big(\frac{\delta \epsilon}{\tau}+\frac{\tau}{\epsilon}\big)^{-1}\Big)\Big)
\Bigg)
\end{split}
\end{equation}
This concludes the proof of  inequality \eref{jkdsgkdjshgdsdkjghe11}. The proof of \eref{lkslkcdsdehkhsdsddedj311} is similar and is also a consequence of \eref{jkdsgkdjshgdsdkjghe11}.

\subsection{Proof of Theorem \ref{thm04}}\label{iuiuyuuyuyuytuty}
Define the process $t\mapsto (\bar{x}_{t},\bar{y}_{t})$ by
\begin{equation}\label{ksjdededhdskzzdjjhwue}
(\bar{x}_{t},\bar{y}_{t}):=\eta(\bar{u}_t) .
\end{equation}
It follows from the regularity of $\eta$ that it is sufficient to prove the
$F$-convergence of $(\bar{x}_{t},\bar{y}_{t})$ towards $\delta_{X_t}\otimes \mu(X_t,dy)$.
Now define $\psi^\epsilon_\tau$ by
\begin{equation}
\psi^\epsilon_\tau(x,y,\omega):= \eta\circ \theta^{\epsilon}_\tau(.,\omega)\circ \eta^{-1}(x,y) ,
\end{equation}
Define $\psi^g_h$ by
\begin{equation}
\psi^g_h(x,y,\omega):= \eta\circ \theta^{G}_h(.,\omega)\circ \eta^{-1}(x,y) .
\end{equation}

\begin{Proposition}\label{kdlhkjhwjdzzd2d}
The vector fields $f$, $g$ and matrix fields $\sigma$, $Q$ associated with the system of Equations \eref{kfgfgdsawssdeedejhd} are uniformly bounded and Lipschitz continuous. We also have
\begin{equation}\label{ksjhdskdcrrzzfdaue0}
\begin{cases}
(\bar{x}_{0},\bar{y}_{0})=\eta(u_0)\\
(\bar{x}_{(k+1)\delta},\bar{y}_{(k+1)\delta})= \psi^g_{\delta-\tau}(.,\omega_k') \circ \psi^\epsilon_{\tau}\big((\bar{x}_{k\delta},\bar{y}_{k\delta}),\omega_k\big)\\
(\bar{x}_{t},\bar{y}_{t})=(\bar{x}_{k\delta},\bar{y}_{k\delta})\quad \text{for}\quad k\delta \leq t <(k+1)\delta
\end{cases}
\end{equation}
where $\omega_k,\omega_k'$ are i.i.d. samples from the probability space $(\Omega,\mathcal{F},\P)$.
Moreover there exists $C>0$ and and  $d$-dimensional centered Gaussian  vectors $\xi'(\omega)$, $\xi''(\omega)$ with identity covariance matrices such that for $h\leq h_0$ and $\frac{\tau}{\epsilon}\leq \tau_0$ we have

\begin{equation}\label{hgfzzjhgfxxxadsjhgdfhf}
\Bigg(\E\Big[\big|\psi^{g}_{h}(x,y,\omega)-(x,y)-h \big(g(x,y),0\big)-\sqrt{h}\big(\sigma(x,y) \xi'(\omega),0\big)\big|^2\Big]\Bigg)^\frac{1}{2}\leq C h^\frac{3}{2} ,
\end{equation}

\begin{equation}\label{hgfjdshzzgfadsjhgdfhf}
\begin{split}
\Bigg(\E\Big[\big|\psi^{\epsilon}_{\tau}(x,y,\omega)-(x,y)-\tau \big(g(x,y),0\big)&-\frac{\tau}{\epsilon}\big(0,f(x,y)\big)-\sqrt{\tau}
\big(\sigma(x,y)\xi''(\omega),0\big)\\&-\sqrt{\frac{\tau}{\epsilon}}
\big(0,Q(x,y)\xi''(\omega)\big)\big|^2\Big]\Bigg)^\frac{1}{2}\leq C \big(\frac{\tau}{\epsilon}\big)^\frac{3}{2} .
\end{split}
\end{equation}
\end{Proposition}

\begin{proof}
Since $(x,y)=\eta(u)$, we obtain from \eref{fullsystemS} and It\^{o}'s formula
\begin{equation}
\begin{split}
dx=&\big((G+\frac{1}{\epsilon}F)\nabla \eta^x \circ \eta^{-1}(x,y)\big)\,dt
+\big(\nabla \eta^x (H+\frac{1}{\sqrt{\epsilon}}K) \big)\circ \eta^{-1}(x,y)\,dW_t\\
&+\frac{1}{2}\sum_{ij}\partial_i \partial_j \eta^x \big((H+\frac{1}{\sqrt{\epsilon}}K)(H+\frac{1}{\sqrt{\epsilon}}K)^T\big)_{ij}\,dt
s\end{split}
\end{equation}
\begin{equation}
\begin{split}
dy=&\big((G+\frac{1}{\epsilon}F)\nabla \eta^y \circ \eta^{-1}(x,y)\big)\,dt
+\big(\nabla \eta^y (H+\frac{1}{\sqrt{\epsilon}}K) \big)\circ \eta^{-1}(x,y)\,dW_t\\
&+\Big(\frac{1}{2}\sum_{ij}\partial_i \partial_j \eta^y \big((H+\frac{1}{\sqrt{\epsilon}}K)(H+\frac{1}{\sqrt{\epsilon}}K)^T\big)_{ij}\Big)\circ \eta^{-1}\,dt .
\end{split}
\end{equation}
Hence we deduce from Equation \eref{kfgfgdsawssdeedejhd} of Condition \ref{lkjhaswkehljhsdjjkehx} that
\begin{equation}
g(x,y)=\Big(G \nabla \eta^x +
\frac{1}{2}\sum_{ij}\partial_i \partial_j \eta^x (H H^T)_{ij} \Big)\circ \eta^{-1}(x,y)
\end{equation}
\begin{equation}
\sigma(x,y)=\big(\nabla \eta^x H\big) \circ \eta^{-1}(x,y)
\end{equation}
\begin{equation}
f(x,y)=\Big(F \nabla \eta^y +
\frac{1}{2}\sum_{ij}\partial_i \partial_j \eta^y (K K^T)_{ij} \Big)\circ \eta^{-1}(x,y)
\end{equation}
\begin{equation}
Q(x,y)=\big(\nabla \eta^y K\big) \circ \eta^{-1}(x,y) .
\end{equation}
\begin{Remark}
Observe that Equation \eref{kfgfgdsawssdeedejhd} of Condition \ref{lkjhaswkehljhsdjjkehx} requires that
\begin{equation}\label{jhdssdfddkdh}
F\nabla \eta^x =0 \quad G\nabla \eta^y ,
\end{equation}
\begin{equation}\label{jhdsjddfkdh}
\sum_{ij}\partial_i \partial_j \eta^x \big(K  K^T\big)_{ij}=0 ,
\end{equation}
\begin{equation}\label{jhdsjrfrfkdh}
\sum_{ij}\partial_i \partial_j \eta^y \big(H H^T\big)_{ij}=0 ,
\end{equation}
\begin{equation}\label{jhdsjkdh}
\sum_{ij}\partial_i \partial_j \eta^x \big(K H^T+H K^T\big)_{ij}=0 ,
\end{equation}
and
\begin{equation}\label{wieuyie}
\sum_{ij}\partial_i \partial_j \eta^y \big(K H^T+H K^T\big)_{ij}=0 .
\end{equation}
Equations \eref{jhdsjkdh} and \eref{wieuyie} are satisfied if $K H^T$ is skew-symmetric. One particular case could be, of course, $K H^T=0$, which translates into the fact that for all $u$ the ranges of $H(u)$ and $K(u)$ are orthogonal, i.e., the noise with amplitude $1/\sqrt{\epsilon}$ is applied to degrees of freedom orthogonal to those with $\mathcal{O}(1)$ noise.
\end{Remark}
We deduce the regularity of $f$, $g$, $\sigma$ and $Q$ from the regularity of $G$, $F$, $H$, $K$ and $\eta$.
Equation \eref{ksjhdskdxddxjjhwue} is a direct consequence of  the definition of $\psi^\epsilon_\tau$ and $\psi^g_h$ and Equation \eref{ksjhdskdcrrzzfdaue0}.
Now observe that
\begin{equation}\label{hgfjdshxzzgfadsjhgdfhf}
\begin{split}
&\psi^{\epsilon}_{\tau}(x,y,\omega)-(x,y)-\tau \big(g(x,y),0\big)-\frac{\tau}{\epsilon}\big(0,f(x,y)\big)
-\sqrt{\tau}\big(\sigma(x,y)\xi'(\omega),0\big)
\\&-\sqrt{\frac{\tau}{\epsilon}}
\big(0,Q(x,y)\xi'(\omega)\big)=
\Big(\eta\circ \theta^{\epsilon}_\tau-\eta-\tau \big(G \nabla \eta^x+
\frac{1}{2}\sum_{ij}\partial_i \partial_j \eta^x (H H^T)_{ij},0\big)\\&-\frac{\tau}{\epsilon}
\big(0,F \nabla \eta^y+
\frac{1}{2}\sum_{ij}\partial_i \partial_j \eta^y (K K^T)_{ij} \big)
-\sqrt{\tau}\big(\nabla \eta^x H\xi'(\omega),0\big)\\& -\sqrt{\frac{\tau}{\epsilon}}
\big(0,\nabla \eta^y K\xi'(\omega)\big)
\Big)\circ \eta^{-1}(x,y) .
\end{split}
\end{equation}
Using Equations \eref{jhdssdfddkdh}, \eref{jhdsjddfkdh}, \eref{jhdsjrfrfkdh},
\eref{jhdsjkdh} and \eref{wieuyie}, the Taylor-Ito expansion of $\eta\circ \theta^{\epsilon}_\tau$, the regularity of $\eta$, and Setting $\xi'$ equal to $\xi$ defined in Equation \eref{hgfjdshgfadsjhgdfhf} we obtain Equation \eref{hgfjdshzzgfadsjhgdfhf}.
The proof of Equation \eref{hgfzzjhgfxxxadsjhgdfhf} is similar.
\end{proof}
It follows from Proposition \ref{kdlhkjhwjdzzd2d} that it is sufficient to prove Theorem \ref{thm04} in the situation where $\eta$ is the identity diffeomorphism.  More precisely the $F$-convergence of $\bar{u}_t$ is a consequence of the $F$-convergence of $(\bar{x}_t,\bar{y}_t)$ and the regularity of $\eta$.

Let $x\mapsto \varphi(x)$ be a function with continuous and bounded derivatives up to order $3$.
Let us prove the following lemma.

\begin{Lemma}\label{jshdjhsdjhgdjh} We have
\begin{equation}\label{hwjheecc}
\begin{split}
\E\big[\varphi(\bar{x}_{(n+1)\delta})\big]&-\E\big[\varphi(\bar{x}_{n\delta})\big]=\\&\delta \E\Big[g(\bar{x}_{n\delta},\bar{y}_{n\delta}) \nabla \varphi(\bar{x}_{n\delta})+ \sigma \sigma^T(\bar{x}_{n\delta},\bar{y}_{n\delta}) :\Hess \varphi(\bar{x}_{n\delta})
\Big]+I_0
\end{split}
\end{equation}
with
\begin{equation}\label{klwehwkesdsdjh}
\begin{split}
|I_0|\leq C  \Big(\delta^\frac{3}{2}+\big(\frac{\tau}{\epsilon}\big)^\frac{3}{2}\Big) .
\end{split}
\end{equation}
\end{Lemma}
\begin{proof}
Write $(\bar{x}_{n\delta+\tau},\bar{y}_{n\delta+\tau}):=\psi^\epsilon_\tau(\bar{x}_{n\delta},\bar{y}_{n\delta},\omega_n)$.
Using Equation \eref{hgfjdshzzgfadsjhgdfhf} we obtain that there exists an $\mathcal{N}(0,1)$ random vector $\xi_n$ independent from $(\bar{x}_{n\delta},\bar{y}_{n\delta})$ and
such that
\begin{equation}
\begin{split}
\bar{x}_{n\delta+\tau}-\bar{x}_{n\delta}=g(\bar{x}_{n\delta}) \tau+ \sqrt{\tau}\sigma(\bar{x}_{n\delta},\bar{y}_{n\delta}) \xi_n+I_1
\end{split}
\end{equation}
with
\begin{equation}
\begin{split}
\big(\E[(I_1)^2]\big)^\frac{1}{2}\leq C \big(\frac{\tau}{\epsilon}\big)^\frac{3}{2} .
\end{split}
\end{equation}
Hence
\begin{equation}\label{hwjhassaeecc}
\begin{split}
\Bigg|\E\big[\varphi(\bar{x}_{n\delta+\tau})\big]&-\E\big[\varphi(\bar{x}_{n\delta})\big]-\tau \E\Big[g(\bar{x}_{n\delta},\bar{y}_{n\delta}) \nabla \varphi(\bar{x}_{n\delta})\\&+ \sigma \sigma^T(\bar{x}_{n\delta},\bar{y}_{n\delta}) :\Hess \varphi(\bar{x}_{n\delta})
\Big]\Bigg|\leq C \big(\frac{\tau}{\epsilon}\big)^\frac{3}{2}
\end{split}
\end{equation}
Similarly, using Equation \eref{hgfzzjhgfxxxadsjhgdfhf} we obtain that there exists an $\mathcal{N}(0,1)$ random vector $\xi_n'$, independent from
$(\bar{x}_{n\delta+\tau},\bar{y}_{n\delta+\tau})$, and such that
\begin{equation}
\begin{split}
\bar{x}_{(n+1)\delta}-\bar{x}_{n\delta+\tau}=g(\bar{x}_{n\delta+\tau},\bar{y}_{n\delta+\tau}) (\delta-\tau)+ \sigma(\bar{x}_{n\delta+\tau},\bar{y}_{n\delta+\tau}) \sqrt{\delta-\tau}\xi_n'+I_2
\end{split}
\end{equation}
with
\begin{equation}
\begin{split}
\big(\E[(I_2)^2]\big)^\frac{1}{2}\leq C (\delta-\tau)^\frac{3}{2} .
\end{split}
\end{equation}
Whence
\begin{equation}\label{hwdsdsjheesdsdcc}
\begin{split}
\Bigg|\E\big[\varphi(\bar{x}_{(n+1)\delta})\big]&-\E\big[\varphi(\bar{x}_{n\delta+\tau})\big]-(\delta -\tau) \E\Big[g(\bar{x}_{n\delta+\tau},\bar{y}_{n\delta+\tau})
\nabla \varphi(\bar{x}_{n\delta+\tau})\\&+ \sigma \sigma^T(\bar{x}_{n\delta+\tau},\bar{y}_{n\delta+\tau}) :\Hess \varphi(\bar{x}_{n\delta+\tau})
\Big]\Bigg|\leq C \big(\delta-\tau\big)^\frac{3}{2} .
\end{split}
\end{equation}
Using  the regularity of $\sigma$, we obtain that
\begin{equation}\label{jjhejhjwheghe}
\begin{split}
\Big(\E\Big[\big|\sigma(\bar{x}_{n\delta+\tau},\bar{y}_{(n+1)\delta})-\sigma(\bar{x}_{n\delta},\bar{y}_{n\delta})\big|^2\Big]\Big)^\frac{1}{2}\leq C\big(\delta^\frac{1}{2}+\sqrt{\frac{\tau}{\epsilon}}\big) .
\end{split}
\end{equation}
The proof of \eref{hwjheecc} follows from \eref{hgfjdshzzgfadsjhgdfhf}, \eref{hwjhassaeecc}, \eref{hwdsdsjheesdsdcc}, \eref{jjhejhjwheghe} and the regularity of $g$ and $\varphi$.
\end{proof}

\begin{Lemma}\label{djskjdshgdjhs}
We have
\begin{equation}
\begin{split}
\Big|\frac{\E\big[\varphi(\bar{x}_{n\delta})\big]-\varphi(x_{0})}{n \delta}-L\varphi(x_0)\Big|\leq J_5
\end{split}
\end{equation}
with (for $\delta \leq C \tau/\epsilon$)
\begin{equation}\label{sjhdsgdjhsdgdsg}
\begin{split}
|J_5|\leq C \Big(\big(\frac{\delta \epsilon}{\tau}\big)^\frac{1}{4}+\big(\frac{\tau}{\epsilon}\big)^\frac{3}{2}\frac{1}{\delta}+\sqrt{\frac{\tau}{\epsilon}}\Big)+C E\big(\frac{1}{C}\ln \frac{\tau}{\delta \epsilon}\big) .
\end{split}
\end{equation}
\end{Lemma}
\begin{proof}
Define $\hat{B}_t$ by $\hat{B}_0=0$ and
\begin{equation}
\hat{B}_{t}-\hat{B}_{n\tau}=B_{n\delta+t}-B_{n \delta} \quad \text{for}\quad n\tau \leq t \leq (n+1)\tau .
\end{equation}
Define $\tilde{y}_{s}$ by $\tilde{y}_0=y_0$ and
\begin{equation}
d \tilde{y}_{t}=\frac{1}{\epsilon}f(x_0,\tilde{y}_{t}) \,dt+\frac{1}{\sqrt{\epsilon}}
Q(x_0,\tilde{y}_{t})d\hat{B}_t .
\end{equation}
Write
\begin{equation}\label{sklghgjhdkhdjkher}
\bar{g}(x_0):=\int g(x_0,y)\,\mu(x_0,dy) .
\end{equation}
Using Lemma \ref{jshdjhsdjhgdjh} we obtain
\begin{equation}
\begin{split}
\frac{\E\big[\varphi(\bar{x}_{n\delta})\big]-\varphi(x_{0})}{n \delta}=L\varphi(x_0)+
J_1+J_2+J_3+J_4 ,
\end{split}
\end{equation}
with
\begin{equation}
\begin{split}
L\varphi(x_0):=\bar{g}(x_0)\nabla \varphi(x_0)+\bar{\sigma} \bar{\sigma}^T(x_0) :\Hess \varphi(x_0) ,
\end{split}
\end{equation}
\begin{equation}
\begin{split}
J_1=&\frac{1}{n}\sum_{k=0}^{n-1} \E\Big[g(\bar{x}_{k\delta},\bar{y}_{k\delta}) \nabla \varphi(\bar{x}_{k\delta})+ \sigma \sigma^T(\bar{x}_{k\delta},\bar{y}_{k\delta}) :\Hess \varphi(\bar{x}_{k\delta})
\Big]\\&- \frac{1}{n}\sum_{k=0}^{n-1} \E\Big[g(\bar{x}_{0},\bar{y}_{k\delta}) \nabla \varphi(\bar{x}_{0})+ \sigma \sigma^T(\bar{x}_{0},\bar{y}_{k\delta}) :\Hess \varphi(\bar{x}_{0})
\Big] ,
\end{split}
\end{equation}
\begin{equation}
\begin{split}
J_2=&\frac{1}{n}\sum_{k=0}^{n-1} \Big( \E\Big[g(\bar{x}_{0},\bar{y}_{k\delta}) \nabla \varphi(\bar{x}_{0})+ \sigma \sigma^T(\bar{x}_{0},\bar{y}_{k\delta}) :\Hess \varphi(\bar{x}_{0})\Big]\\&
-\frac{1}{\tau}\int_{k\tau}^{(k+1)\tau}\E\Big[g(x_0,\tilde{y}_{s}) \nabla \varphi(x_0)+ \sigma \sigma^T(x_0,\tilde{y}_s) :\Hess \varphi(x_0)
\Big]\,ds \Big) ,
\end{split}
\end{equation}

\begin{equation}
\begin{split}
J_3=\frac{1}{n\tau}\int_0^{n\tau} \E\Big[g(x_0,\tilde{y}_{s}) \nabla \varphi(x_0)+ \sigma \sigma^T(x_0,\tilde{y}_s) :\Hess \varphi(x_0)
\Big]\,ds-L\varphi(x_0) ,
\end{split}
\end{equation}
\begin{equation}
\begin{split}
|J_4|\leq C \Big(\delta^\frac{1}{2}+\big(\frac{\tau}{\epsilon}\big)^\frac{3}{2}\frac{1}{\delta}\Big) .
\end{split}
\end{equation}
Using the regularity of $\sigma,g,\varphi$, \eref{ksjhdskdxddxjjhwue} and \eref{lsfde3eldeiuuddiuskjlkd} we obtain

\begin{equation}
\begin{split}
|J_1|\leq C \Big((n\delta)^\frac{1}{2}+n\delta+n \big(\frac{\tau}{\epsilon}\big)^\frac{3}{2}\Big) .
\end{split}
\end{equation}
Using Property 3 of Condition \ref{lkjhaswkehljhsdjjkehx} and Property 3 of Condition \ref{lkjhewlkehlksdsedsdaeehx} we obtain
\begin{equation}
\begin{split}
|J_3|\leq C E(\frac{n\tau}{\epsilon}).
\end{split}
\end{equation}
Using \eref{hgfzzjhgfxxxadsjhgdfhf} and \eref{hgfjdshzzgfadsjhgdfhf}, we obtain the following inequality
\begin{equation}
\begin{split}
\Big(\E\Big[\big|\bar{y}_{n\delta} -\tilde{y}_{n\tau}\big|^2\Big] \Big)^\frac{1}{2}
\leq C \Big(\sqrt{\frac{\tau}{\epsilon}}+(n\delta)^\frac{1}{2}+n\delta+n \big(\frac{\tau}{\epsilon}\big)^\frac{3}{2}\Big)\frac{n\tau}{\epsilon}e^{C\frac{n\tau}{\epsilon}},
\end{split}
\end{equation}
which leads to
\begin{equation}
\begin{split}
|J_2|\leq C \Big(\sqrt{\frac{\tau}{\epsilon}}+(n\delta)^\frac{1}{2}+n\delta+n \big(\frac{\tau}{\epsilon}\big)^\frac{3}{2}\Big)e^{C\frac{n\tau}{\epsilon}} .
\end{split}
\end{equation}
Hence, we have obtained
\begin{equation}
\begin{split}
\Big|\frac{\E\big[\varphi(\bar{x}_{n\delta})\big]-\varphi(x_{0})}{n \delta}-L\varphi(x_0)\Big|\leq J_5 ,
\end{split}
\end{equation}
with
\begin{equation}
\begin{split}
|J_5|\leq C \Big(\sqrt{\frac{\tau}{\epsilon}}+(n\delta)^\frac{1}{2}+n\delta+n \big(\frac{\tau}{\epsilon}\big)^\frac{3}{2}\Big)
e^{C\frac{n\tau}{\epsilon}}+E(\frac{n\tau}{\epsilon})+
C\big(\frac{\tau}{\epsilon}\big)^\frac{3}{2}\frac{1}{\delta} .
\end{split}
\end{equation}
Choosing $n$ such that $\sqrt{\frac{n\tau}{\epsilon}} e^{C\frac{n\tau}{\epsilon}}\sim \big(\frac{\tau}{\epsilon \delta}\big)^\frac{1}{4}$ we obtain \eref{sjhdsgdjhsdgdsg} for $\delta \leq C \tau/\epsilon$.
\end{proof}

We now combine Lemma \ref{djskjdshgdjhs} with Theorem 1 of Chapter 2 of \cite{MR1020057} which states that the uniform convergence (in $x_0$, $y_0$) of $\frac{\E\big[\varphi(\bar{x}_{n\delta})\big]-\varphi(x_{0})}{n \delta}$ to $L\varphi(x_0)$  as $\epsilon \downarrow 0$, $\tau \leq \delta$, $\frac{\tau}{\epsilon} \downarrow 0$, $\frac{\delta \epsilon }{\tau} \downarrow 0$ and $\big(\frac{\tau}{\epsilon}\big)^\frac{3}{2}\frac{1}{\delta} \downarrow 0$
implies the convergence in distribution of $\bar{x}_{n\delta}$ to the Markov process generated by $L$.

The $F$-convergence of $(\bar{x}_t,\bar{y}_t)$ can be deduced from the convergence
in distribution of $\bar{x}_t$ and Equation \eref{kshlkshe} of Condition \ref{lkjhaswkehljhsdjjkehx}. The proof follows the same lines as above, which will not be repeated here.

\bibliographystyle{siam}
\bibliography{houman10}

\def\cprime{$'$} \def\cprime{$'$} \def\cprime{$'$}
  \def\cydot{\leavevmode\raise.4ex\hbox{.}}
\begin{thebibliography}{100}

\bibitem{MR1923724}
{\sc A.~Abdulle}, {\em Fourth order {C}hebyshev methods with recurrence
  relation}, SIAM J. Sci. Comput., 23 (2002), pp.~2041--2054 (electronic).

\bibitem{MR2385896}
{\sc A.~Abdulle and S.~Cirilli}, {\em S-{ROCK}: {C}hebyshev methods for stiff
  stochastic differential equations}, SIAM J. Sci. Comput., 30 (2008),
  pp.~997--1014.

\bibitem{Allaire1992}
{\sc G.~Allaire}, {\em Homogenization and two-scale convergence}, SIAM J. Math.
  Anal., 23 (1992), pp.~1482--1518.

\bibitem{RATTLE}
{\sc H.~Anderson}, {\em {RATTLE}: A velocity version of the {SHAKE} algorithm
  for molecular dynamics calculations}, J. Comput. Phys., 52 (1983),
  pp.~24--34.

\bibitem{Ariel:08}
{\sc G.~Ariel, B.~Engquist, and Y.-H. Tsai}, {\em A multiscale method for
  highly oscillatory ordinary differential equations with resonance}, Math.
  Comput., 78 (2009), p.~929.

\bibitem{Ariel:09}
\leavevmode\vrule height 2pt depth -1.6pt width 23pt, {\em A reversible
  multiscale integration method}, To appear, Comm. Math. Sci.,  (2009).

\bibitem{MR2174834}
{\sc A.~Armaou and I.~Kevrekidis}, {\em Equation-free optimal switching
  policies for bistable reacting systems}, Internat. J. Robust Nonlinear
  Control, 15 (2005), pp.~713--726.

\bibitem{MR1447964}
{\sc S.~Artem{\cprime}ev and K.~Shurts}, {\em Zhestkie sistemy
  stokhasticheskikh differentsialnykh uravnenii s malym shumom i ikh chislennoe
  reshenie}, vol.~1039, Ross. Akad. Nauk Sibirsk. Otdel. Vychisl. Tsentr,
  Novosibirsk, 1995.

\bibitem{Art07b}
{\sc Z.~Artstein, I.~G. Kevrekidis, M.~Slemrod, and E.~S. Titi}, {\em Slow
  observables of singularly perturbed differential equations}, Nonlinearity, 20
  (2007), pp.~2463--2481.

\bibitem{Art07}
{\sc Z.~Artstein, J.~Linshiz, and E.~S. Titi}, {\em Young measure approach to
  computing slowly advancing fast oscillations}, Multiscale Model. Simul., 6
  (2007), pp.~1085--1097.

\bibitem{BeLiPa78}
{\sc A.~Bensoussan, J.~L. Lions, and G.~Papanicolaou}, {\em Asymptotic analysis
  for periodic structure}, North Holland, Amsterdam, 1978.

\bibitem{Bi1981}
{\sc J.~Bismut}, {\em M\'ecanique al\'eatoire}, Springer, 1981.

\bibitem{MR0043001}
{\sc N.~Bogolyubov}, {\em Problemy dinami\v cesko\u\i\ teorii v statisti\v
  cesko\u\i\ fizike}, Gosudarstv. Izdat. Tehn.-Teor. Lit., Moscow-Leningrad,
  1946.

\bibitem{BoSc95}
{\sc F.~A. Bornemann and C.~Sch{\"u}tte}, {\em A mathematical approach to
  smoothed molecular dynamics: Correcting potentials for freezing bond angles}.
\newblock Prprint SC 9-30 (Dember 1995), 1995.

\bibitem{MR1436164}
\leavevmode\vrule height 2pt depth -1.6pt width 23pt, {\em Homogenization of
  {H}amiltonian systems with a strong constraining potential}, Phys. D, 102
  (1997), pp.~57--77.

\bibitem{MR2496560}
{\sc N.~Bou-Rabee and J.~Marsden}, {\em Hamilton-{P}ontryagin integrators on
  {L}ie groups. {I}. {I}ntroduction and structure-preserving properties},
  Found. Comput. Math., 9 (2009), pp.~197--219.

\bibitem{BoOw:09}
{\sc N.~Bou-Rabee and H.~Owhadi}, {\em Long-run accuracy of variational
  integrators in the stochastic context}.
\newblock arXiv:0712.4123. Accepted for publication in SINUM., 2009.

\bibitem{MR2491434}
\leavevmode\vrule height 2pt depth -1.6pt width 23pt, {\em Stochastic
  variational integrators}, IMA J. Numer. Anal., 29 (2009), pp.~421--443.

\bibitem{BoVa:09}
{\sc N.~Bou-Rabee and E.~Vanden-Eijnden}, {\em Pathwise accuracy and ergodicity
  of metropolized integrators for {SDEs}}.
\newblock Submitted, 2009.

\bibitem{LeBris:07}
{\sc C.~L. Bris and F.~Legoll}, {\em Integrators for highly oscillatory
  {H}amiltonian systems: an homogenization approach}, Tech. Report 6252, Inria
  Rapport de recherche, 2007.

\bibitem{CHARMM}
{\sc B.~Brooks, R.~Bruccoleri, B.~Olafson, D.~States, S.~Swaminathan, and
  M.~Karplus}, {\em {CHARMM}: A program for macromolecular energy,
  minimization, and dynamics calculations}, J Comp. Chem., 4 (1983),
  pp.~187--217.

\bibitem{MR1835724}
{\sc K.~Burrage and T.~Tian}, {\em The composite {E}uler method for stiff
  stochastic differential equations}, J. Comput. Appl. Math., 131 (2001),
  pp.~407--426.

\bibitem{MR1876404}
\leavevmode\vrule height 2pt depth -1.6pt width 23pt, {\em Stiffly accurate
  {R}unge-{K}utta methods for stiff stochastic differential equations}, Comput.
  Phys. Comm., 142 (2001), pp.~186--190.
\newblock Computational physics 2000. ``New challenges for the new millenium''
  (Gold Coast).

\bibitem{CaSer08}
{\sc M.~Calvo and J.~Sanz-Serna}, {\em Heterogeneous multiscale methods for
  mechanical systems with vibrations.}, preprint,  (2008).

\bibitem{CalvoSena09}
\leavevmode\vrule height 2pt depth -1.6pt width 23pt, {\em Instabilities and
  inaccuracies in the integration of highly oscillatory problems}, SIAM J. Sci.
  Comput., 31 (2009), pp.~1653--1677.

\bibitem{CaCha:09}
{\sc F.~Castella, P.~Chartier, and E.~Faou}, {\em An averaging technique for
  highly-oscillatory hamiltonian problems}, SIAM J. Numer. Anal., 47 (2009),
  pp.~2808--2837.

\bibitem{MR1750741}
{\sc A.~Chorin, O.~Hald, and R.~Kupferman}, {\em Optimal prediction and the
  {M}ori-{Z}wanzig representation of irreversible processes}, Proc. Natl. Acad.
  Sci. USA, 97 (2000), pp.~2968--2973 (electronic).

\bibitem{MR1915310}
\leavevmode\vrule height 2pt depth -1.6pt width 23pt, {\em Optimal prediction
  with memory}, Phys. D, 166 (2002), pp.~239--257.

\bibitem{MR1619894}
{\sc A.~Chorin, A.~Kast, and R.~Kupferman}, {\em Optimal prediction of
  underresolved dynamics}, Proc. Natl. Acad. Sci. USA, 95 (1998),
  pp.~4094--4098 (electronic).

\bibitem{CiLeVa2008}
{\sc G.~Ciccotti, T.~Lelievre, and E.~Vanden-Eijnden}, {\em Projections of
  diffusions on submanifolds: Application to mean force computation}, CPAM, 61
  (2008), pp.~0001--0039.

\bibitem{MR2275175}
{\sc D.~Cohen, T.~Jahnke, K.~Lorenz, and C.~Lubich}, {\em Numerical integrators
  for highly oscillatory {H}amiltonian systems: a review}, in Analysis,
  modeling and simulation of multiscale problems, Springer, Berlin, 2006,
  pp.~553--576.

\bibitem{MR2542877}
{\sc M.~Condon, A.~Dea{\~n}o, and A.~Iserles}, {\em On highly oscillatory
  problems arising in electronic engineering}, M2AN Math. Model. Numer. Anal.,
  43 (2009), pp.~785--804.

\bibitem{AMBER}
{\sc W.~Cornell, P.~Cieplak, C.~Bayly, I.~Gould, K.~Merz, D.~Ferguson,
  D.~Spellmeyer, T.~Fox, J.~Caldwell, and P.~Kollman}, {\em A second generation
  force field for the simulation of proteins, nucleic acids, and organic
  molecules}, J. Am. Chem. Soc., 117 (1995), pp.~5179--5197.

\bibitem{MR0080998}
{\sc G.~Dahlquist}, {\em Convergence and stability in the numerical integration
  of ordinary differential equations}, Math. Scand., 4 (1956), pp.~33--53.

\bibitem{MR2069938}
{\sc W.~E}, {\em Analysis of the heterogeneous multiscale method for ordinary
  differential equations}, Commun. Math. Sci., 1 (2003), pp.~423--436.

\bibitem{MR2314852}
{\sc W.~E, B.~Engquist, X.~Li, W.~Ren, and E.~Vanden-Eijnden}, {\em
  Heterogeneous multiscale methods: a review}, Commun. Comput. Phys., 2 (2007),
  pp.~367--450.

\bibitem{MR2165382}
{\sc W.~E, D.~Liu, and E.~Vanden-Eijnden}, {\em Analysis of multiscale methods
  for stochastic differential equations}, Comm. Pure Appl. Math., 58 (2005),
  pp.~1544--1585.

\bibitem{Nested07}
\leavevmode\vrule height 2pt depth -1.6pt width 23pt, {\em Nested stochastic
  simulation algorithms for chemical kinetic systems with multiple time
  scales}, J. Comput. Phys., 221 (2007), pp.~158--180.

\bibitem{Seamless09}
{\sc W.~E, W.~Ren, and E.~Vanden-Eijnden}, {\em A general strategy for
  designing seamless multiscale methods}, J. Comput. Phys., 228 (2009),
  pp.~5437--5453.

\bibitem{MR2164093}
{\sc B.~Engquist and Y.-H. Tsai}, {\em Heterogeneous multiscale methods for
  stiff ordinary differential equations}, Math. Comp., 74 (2005),
  pp.~1707--1742 (electronic).

\bibitem{MR2026176}
{\sc L.~Evans}, {\em A survey of partial differential equations methods in weak
  {KAM} theory}, Comm. Pure Appl. Math., 57 (2004), pp.~445--480.

\bibitem{FPU:55}
{\sc E.~Fermi, J.~Pasta, and S.~Ulam}, {\em Studies of nonlinear problems},
  Tech. Report LA-1940, Los Alamos Scientific Laboratory, 1955.

\bibitem{Fixman:74}
{\sc M.~Fixman}, {\em Classical statistical mechanics of constraints: A theorem
  and application to polymers}, Proc. Nat. Acad. Sci. USA, 71-8 (1974),
  pp.~3050--3053.

\bibitem{Skeel:99}
{\sc B.~Garc\'{i}a-Archilla, J.~Sanz-Serna, and R.~Skeel}, {\em Long-time-step
  methods for oscillatory differential equations}, SIAM J.Sci.Comput., 20 (3)
  (1999), pp.~930--963.

\bibitem{MR0138200}
{\sc W.~Gautschi}, {\em Numerical integration of ordinary differential
  equations based on trigonometric polynomials}, Numer. Math., 3 (1961),
  pp.~381--397.

\bibitem{MR0315898}
{\sc C.~Gear}, {\em Numerical initial value problems in ordinary differential
  equations}, Prentice-Hall Inc., Englewood Cliffs, N.J., 1971.

\bibitem{MR654346}
{\sc C.~Gear and K.~Gallivan}, {\em Automatic methods for highly oscillatory
  ordinary differential equations}, in Numerical analysis ({D}undee, 1981),
  vol.~912 of Lecture Notes in Math., Springer, Berlin, 1982, pp.~115--124.

\bibitem{MR1976207}
{\sc C.~Gear and I.~Kevrekidis}, {\em Projective methods for stiff differential
  equations: problems with gaps in their eigenvalue spectrum}, SIAM J. Sci.
  Comput., 24 (2003), pp.~1091--1106 (electronic).

\bibitem{MR0073876}
{\sc I.~Gihman}, {\em On the theory of differential equations of stochastic
  processes. {I}, {II}}, Amer. Math. Soc. Transl. (2), 1 (1955), pp.~111--137,
  139--161.

\bibitem{GiKeKup06}
{\sc D.~Givonand, I.~G. Kevrekidis, and R.~Kupferman}, {\em Strong convergence
  of projective integration schemes for singularly perturbed stochastic
  differential systems}, Commun. Math. Sci., 4 (2006), pp.~707--729.

\bibitem{Grubmuller:91}
{\sc H.~Grubmuller, H.~Heller, A.~Windemuth, and K.~Schulten}, {\em Generalized
  {V}erlet algorithm for efficient molecular dynamics simulations with
  long-range interactions}, Mol. Sim., 6 (1991), pp.~121--142.

\bibitem{MR1704287}
{\sc S.~Gusev}, {\em Algoritm peremennogo shaga dlya chislennogo resheniya
  zhestkikh sistem stokhasticheskikh differentsialnykh uravnenii}, vol.~1094,
  Rossi\u\i skaya Akademiya Nauk Sibirskoe Otdelenie, Institut
  Vychislitel\cprime no\u\i\ Matematiki i Matematichesko\u\i\ Geofiziki,
  Novosibirsk, 1997.

\bibitem{MR2249159}
{\sc E.~Hairer, C.~Lubich, and G.~Wanner}, {\em Geometric numerical integration
  illustrated by the {S}t{\"o}rmer-{V}erlet method}, Acta Numer., 12 (2003),
  pp.~399--450.

\bibitem{Hairer:04}
\leavevmode\vrule height 2pt depth -1.6pt width 23pt, {\em Geometric Numerical
  Integration: Structure-Preserving Algorithms for Ordinary Differential
  Equations}, Springer, Heidelberg Germany, second~ed., 2004.

\bibitem{MR1227985}
{\sc E.~Hairer, S.~N{\o}rsett, and G.~Wanner}, {\em Solving ordinary
  differential equations. {I}}, vol.~8 of Springer Series in Computational
  Mathematics, Springer-Verlag, Berlin, second~ed., 1993.
\newblock Nonstiff problems.

\bibitem{MR1439506}
{\sc E.~Hairer and G.~Wanner}, {\em Solving ordinary differential equations.
  {II}}, vol.~14 of Springer Series in Computational Mathematics,
  Springer-Verlag, Berlin, second~ed., 1996.
\newblock Stiff and differential-algebraic problems.

\bibitem{Ha2007}
{\sc C.~Hartmann}, {\em An ergodic sampling scheme for constrained
  {H}amiltonian systems with applications to molecular dynamics}, J. Stat.
  Phys., 130 (2008), pp.~687--711.

\bibitem{Iserles02}
{\sc A.~Iserles}, {\em On the global error of discretization methods for
  highly-oscillatory ordinary differential equations}, BIT, 42 (2002),
  pp.~561--599.

\bibitem{MR1936107}
\leavevmode\vrule height 2pt depth -1.6pt width 23pt, {\em Think globally, act
  locally: solving highly-oscillatory ordinary differential equations}, Appl.
  Numer. Math., 43 (2002), pp.~145--160.
\newblock 19th Dundee Biennial Conference on Numerical Analysis (2001).

\bibitem{MR2182817}
\leavevmode\vrule height 2pt depth -1.6pt width 23pt, {\em On the numerical
  analysis of rapid oscillation}, in Group theory and numerical analysis,
  vol.~39 of CRM Proc. Lecture Notes, Amer. Math. Soc., Providence, RI, 2005,
  pp.~149--163.

\bibitem{MR2303638}
{\sc A.~Iserles, S.~P. N{\o}rsett, and S.~Olver}, {\em Highly oscillatory
  quadrature: the story so far}, in Numerical mathematics and advanced
  applications, Springer, Berlin, 2006, pp.~97--118.

\bibitem{JiKoOl91}
{\sc V.~Jikov, S.~Kozlov, and O.~Oleinik}, {\em Homogenization of Differential
  Operators and Integral Functionals}, Springer-Verlag, 1991.

\bibitem{Kapitza:65}
{\sc P.~Kapitza}, {\em Collected Papers of P.L.Kapitza, Volume II., edited by
  D. Ter Haar}, Pergamon Press, Oxford UK, 1965.

\bibitem{MR2041455}
{\sc I.~Kevrekidis, C.~Gear, J.~Hyman, P.~Kevrekidis, O.~Runborg, and
  C.~Theodoropoulos}, {\em Equation-free, coarse-grained multiscale
  computation: enabling microscopic simulators to perform system-level
  analysis}, Commun. Math. Sci., 1 (2003), pp.~715--762.

\bibitem{KevGio09}
{\sc I.~Kevrekidis and G.~Samaey}, {\em Equation-free multiscale computation:
  Algorithms and applications}, Annual Review of Physical Chemistry, 60 (2009),
  pp.~321--344.
\newblock PMID: 19335220.

\bibitem{MR1165724}
{\sc H.-O. Kreiss}, {\em Problems with different time scales}, Acta Numer., 1
  (1992), pp.~101--139.

\bibitem{KrBo1937}
{\sc B.~Kryloff and N.~Bogoliouboff}, {\em La th\'eorie g\'en\'erale de la
  mesure dans son application \`a l'\'etude des syst\`emes dynamiques de la
  m\'ecanique non lin\'eaire}, Ann. of Math. (2), 38 (1937), pp.~65--113.

\bibitem{Kry39}
{\sc N.~Kryloff and N.~Bogoliouboff}, {\em On some problems in the ergodic
  theory of stochastic systems}, Zap. Kafedr. Mat. Fiz. Inst. Budivel. Mat.
  Akad. Nauk. Ukrain. SSR, 4 (1939), pp.~243--287.

\bibitem{MR2408499}
{\sc J.-A. L{\'a}zaro-Cam{\'{\i}} and J.~Ortega}, {\em Stochastic {H}amiltonian
  dynamical systems}, Rep. Math. Phys., 61 (2008), pp.~65--122.

\bibitem{MR0443314}
{\sc V.~Lebedev and S.~Finogenov}, {\em The use of ordered \v {C}eby\v sev
  parameters in iteration methods}, \v Z. Vy\v cisl. Mat. i Mat. Fiz., 16
  (1976), pp.~895--907, 1084.

\bibitem{MR1843642}
{\sc B.~Leimkuhler and S.~Reich}, {\em A reversible averaging integrator for
  multiple time-scale dynamics}, J. Comput. Phys., 171 (2001), pp.~95--114.

\bibitem{MR2132573}
\leavevmode\vrule height 2pt depth -1.6pt width 23pt, {\em Simulating
  {H}amiltonian dynamics}, vol.~14 of Cambridge Monographs on Applied and
  Computational Mathematics, Cambridge University Press, Cambridge, 2004.

\bibitem{MR1428715}
{\sc B.~J. Leimkuhler, S.~Reich, and R.~D. Skeel}, {\em Integration methods for
  molecular dynamics}, in Mathematical approaches to biomolecular structure and
  dynamics ({M}inneapolis, {MN}, 1994), vol.~82 of IMA Vol. Math. Appl.,
  Springer, New York, 1996, pp.~161--185.

\bibitem{LeMaOrWe2003}
{\sc A.~Lew, J.~Marsden, M.~Ortiz, and M.~West}, {\em Asynchronous variational
  integrators}, Arch. Rational Mech. Anal., 167 (2003), pp.~85--145.

\bibitem{LeMaOrWe2004b}
\leavevmode\vrule height 2pt depth -1.6pt width 23pt, {\em Variational time
  integrators}, Int. J. Numer. Methods Eng., 60 (2004), pp.~153--212.

\bibitem{LiAbE:08}
{\sc T.~Li, A.~Abdulle, and W.~E}, {\em Effectiveness of implicit methods for
  stiff stochastic differential equations}, Commun. Comput. Phys., 3 (2008),
  pp.~295--307.

\bibitem{MR2385877}
{\sc S.~Malham and A.~Wiese}, {\em Stochastic {L}ie group integrators}, SIAM J.
  Sci. Comput., 30 (2008), pp.~597--617.

\bibitem{MaWe:01}
{\sc J.~Marsden and M.~West}, {\em Discrete mechanics and variational
  integrators}, Acta Numerica,  (2001), pp.~357--514.

\bibitem{FPUcomp07}
{\sc R.~McLachlan and D.~O{\'\i}Neale}, {\em Comparison of integrators for the
  {F}ermi-{P}asta-{U}lam problem.}, preprint NI07052-HOP, Isaac Newton
  Institute for Mathematical Sciences.,  (2007).
\newblock http://www.newton.ac.uk/preprints/NI07052.pdf.

\bibitem{McPe2001}
{\sc R.~McLachlan and M.~Perlmutter}, {\em Conformal {H}amiltonian systems}, J.
  Geom. Phys., 39 (2001), pp.~276--300.

\bibitem{McQu02}
{\sc R.~McLachlan, G.~Reinout, and W.~Quispel}, {\em Splitting methods}, Acta
  Numerica,  (2002), pp.~341--434.

\bibitem{MiReTr2002}
{\sc G.~Milstein, Y.~Repin, and M.~Tretyakov}, {\em Symplectic methods for
  {H}amiltonian systems with additive noise}, SIAM J. Num. Anal., 39 (2002),
  pp.~1--9.

\bibitem{MiReTr2003}
\leavevmode\vrule height 2pt depth -1.6pt width 23pt, {\em Symplectic methods
  for stochastic systems preserving symplectic structure}, SIAM J. Num. Anal.,
  40 (2003), pp.~1--9.

\bibitem{MiTr2003}
{\sc G.~Milstein and M.~Tretyakov}, {\em Quasi-symplectic methods for
  {L}angevin-type equations}, IMA J. Numer. Anal., 23 (2003), pp.~593--626.

\bibitem{MiTr2004}
{\sc G.~N. Milstein and M.~V. Tretyakov}, {\em Stochastic Numerics for
  Mathematical Physics}, Springer, 2004.

\bibitem{Ngu90}
{\sc G.~Nguetseng}, {\em A general convergence result for a functional related
  to the theory of homogenization}, SIAM J. Math. Anal., 20 (1989),
  pp.~608--623.

\bibitem{PaKo74}
{\sc G.~Papanicolaou and W.~Kohler}, {\em Asymptotic theory of mixing
  stochastic ordinary differential equations}, Comm. Pure Appl. Math., 27
  (1974), pp.~641--668.

\bibitem{MR2382139}
{\sc G.~A. Pavliotis and A.~M. Stuart}, {\em Multiscale methods}, vol.~53 of
  Texts in Applied Mathematics, Springer, New York, 2008.
\newblock Averaging and homogenization.

\bibitem{PeSkYa85}
{\sc D.~Perchak, J.~Skolnick, and R.~Yaris}, {\em Dynamics of rigid and
  flexible constraints for polymers. {E}ffect of the {F}ixman potential},
  Macromolecules, 18 (1985), pp.~519--525.

\bibitem{MR1489260}
{\sc L.~Petzold, L.~Jay, and J.~Yen}, {\em Problems with different time
  scales}, Acta Numer., 6 (1997), pp.~437--483.

\bibitem{Reich2000210}
{\sc S.~Reich}, {\em Smoothed langevin dynamics of highly oscillatory systems},
  Phys. D, 138 (2000), pp.~210--224.

\bibitem{butane1}
{\sc R.~Rosenberg, B.~Berne, and D.~Chandler}, {\em Isomerization dynamics in
  liquids by molecular dynamics}, Chem. Phys. Lett., 75 (1980), p.~162.

\bibitem{butane2}
{\sc J.~Ryckaert and A.~Bellemans}, {\em Molecular dynamics of liquid n-butane
  near its boiling point}, Chem. Phys. Lett., 30 (1975), p.~123.

\bibitem{SHAKE}
{\sc J.~Ryckaert, G.~Ciccotti, and H.~Berendsen}, {\em Numerical integration of
  the cartesian equations of motion of a system with constraints: Molecular
  dynamics of n-alkanes}, J. Comput. Phys., 23 (1977), pp.~327--341.

\bibitem{MR810620}
{\sc J.~A. Sanders and F.~Verhulst}, {\em Averaging methods in nonlinear
  dynamical systems}, vol.~59 of Applied Mathematical Sciences,
  Springer-Verlag, New York, 1985.

\bibitem{Sanz-Serna:08}
{\sc J.~Sanz-Serna}, {\em Mollified impulse methods for highly oscillatory
  differential equations}, SIAM J. Numer. Anal., 46 (2) (2008), pp.~1040--1059.

\bibitem{SernaHammer}
\leavevmode\vrule height 2pt depth -1.6pt width 23pt, {\em Stabilizing with a
  hammer}, Stoch. Dyn., 8 (2008), pp.~47--57.

\bibitem{SSerna09}
\leavevmode\vrule height 2pt depth -1.6pt width 23pt, {\em Modulated {F}ourier
  expansions and heterogeneous multiscale methods}, IMA J. Numer. Anal., 29
  (2009), pp.~595--605.

\bibitem{SerArTs09}
{\sc J.~Sanz-Serna, G.~Ariel, and Y.-H. Tsai}, {\em Multiscale methods for
  stiff and constrained mechanical systems.}, preprint,  (2009).

\bibitem{MR717698}
{\sc R.~Scheid}, {\em The accurate numerical solution of highly oscillatory
  ordinary differential equations}, Math. Comp., 41 (1983), pp.~487--509.

\bibitem{MR1490094}
{\sc C.~Sch{\"u}tte and F.~A. Bornemann}, {\em Homogenization approach to
  smoothed molecular dynamics}, in Proceedings of the {S}econd {W}orld
  {C}ongress of {N}onlinear {A}nalysts, {P}art 3 ({A}thens, 1996), vol.~30,
  1997, pp.~1805--1814.

\bibitem{MR2161717}
{\sc R.~Sharp, Y.-H. Tsai, and B.~Engquist}, {\em Multiple time scale numerical
  methods for the inverted pendulum problem}, in Multiscale methods in science
  and engineering, vol.~44 of Lect. Notes Comput. Sci. Eng., Springer, Berlin,
  2005, pp.~241--261.

\bibitem{Skeel2002}
{\sc R.~Skeel and J.~Izaguirre}, {\em J.a.: An impulse integrator for
  {L}angevin dynamics}, Mol. Phys, 100 (2002), pp.~3885--3891.

\bibitem{MR1020057}
{\sc A.~Skorokhod}, {\em Asymptotic methods in the theory of stochastic
  differential equations}, vol.~78 of Translations of Mathematical Monographs,
  American Mathematical Society, Providence, RI, 1989.
\newblock Translated from the Russian by H. H. McFaden.

\bibitem{Stern:09}
{\sc A.~Stern and E.~Grinspun}, {\em Implicit-explicit variational integration
  of highly oscillatory problems}, Multiscale Model. Simul.,  (2009).
\newblock Accepted, to appear.

\bibitem{MTS78}
{\sc W.~B. Streett, D.~J. Tildesley, and G.~Saville}, {\em Multiple time-step
  methods in molecular dynamics}, Molecular Physics, 35 (1978), pp.~639 -- 648.

\bibitem{Takens:80}
{\sc F.~Takens}, {\em Motion under the influence of a strong constraining
  force, in Global Theory of Dynamical Systems. edited by Z. Nitecki and C.
  Robinson}, Springer-Verlag, Berlin-Heidelberg Germany, 1980.

\bibitem{MR1834774}
{\sc T.~Tian and K.~Burrage}, {\em Implicit {T}aylor methods for stiff
  stochastic differential equations}, Appl. Numer. Math., 38 (2001),
  pp.~167--185.

\bibitem{TBM92}
{\sc M.~Tuckerman, B.~Berne, and G.~Martyna}, {\em Reversible multiple time
  scale molecular dynamics}, J. Chem. Phys., 97 (1992), pp.~1990--2001.

\bibitem{Tuckerman:92}
{\sc M.~Tuckerman, B.~J. Berne, and G.~J. Martyna}, {\em Reversible multiple
  time scale molecular dynamics}, J. Chem. Phys., 97 (1992), pp.~1990--2001.

\bibitem{Eric07HMMlike}
{\sc E.~Vanden-Eijnden}, {\em On {HMM}-like integrators and projective
  integration methods for systems with multiple time scales}, Commun. Math.
  Sci., 5 (2007), pp.~495--505.

\bibitem{VaCi2006}
{\sc E.~Vanden-Eijnden and G.~Ciccotti}, {\em Second-order integrators for
  {L}angevin equations with holonomic constraints}, Chem.~Phys.~Letters, 429
  (2006), pp.~310--316.

\bibitem{Verhulst:96}
{\sc F.~Verhulst}, {\em Nonlinear Differential Equations and Dynamical
  Systems}, Springer, Berlin-Heidelberg Germany, second~ed., 1996.

\bibitem{Verlet1967}
{\sc L.~Verlet}, {\em Computer ``experiments'' on classical fluids. {I}.
  thermodynamical properties of {L}ennard-{J}ones molecules}, Physical Review,
  159 (1967), pp.~98+.

\bibitem{Viswanath01}
{\sc D.~Viswanath}, {\em Global errors of numerical {ODE} solvers and
  {L}yapunov's theory of stability}, IMA J. Numer. Anal., 21 (2001),
  pp.~387--406.

\bibitem{Yanao:09}
{\sc T.~Yanao, W.~S. Koon, and J.~E. Marsden}, {\em Intramolecular energy
  transfer and the driving mechanisms for large-amplitude collective motions of
  clusters}, J. Chem. Phys.,  (2009), p.~144111.

\bibitem{ZhSc:93}
{\sc G.~Zhang and T.~Schlick}, {\em {LIN}: A new algorithm to simulate the
  dynamics of biomolecules by combining implicit-integration and normal mode
  techniques}, J. Comp. Chem., 14 (1993), pp.~1212--1233.

\end{thebibliography}
\end{document}